\documentclass{amsart}
\usepackage[T1]{fontenc}
\usepackage[UKenglish]{babel}%
\usepackage{amssymb}
\usepackage{amsmath}
\usepackage{lmodern}
\usepackage{graphicx}
\usepackage{dsfont}
\usepackage{amsthm}
\usepackage{enumerate}
\usepackage[usenames]{color}
\usepackage[all]{xy}
\usepackage{mathtools}
\usepackage{appendix}
\usepackage{pgf,tikz}
\usepackage{tikz-cd}
\usetikzlibrary{arrows}
\usepackage{multirow}
\usepackage{minitoc}
\usepackage{stmaryrd}
\usepackage{centernot}
\usepackage{verbatim}

\DeclareFontFamily{U}{mathx}{\hyphenchar\font45}
\DeclareFontShape{U}{mathx}{m}{n}{
      <5> <6> <7> <8> <9> <10>
      <10.95> <12> <14.4> <17.28> <20.74> <24.88>
      mathx10
      }{}
\DeclareSymbolFont{mathx}{U}{mathx}{m}{n}
\DeclareFontSubstitution{U}{mathx}{m}{n}
\DeclareMathAccent{\widecheck}{0}{mathx}{"71}
\DeclareMathAccent{\wideparen}{0}{mathx}{"75}

\usepackage{mathrsfs}  

\newcommand{\grad}{\nabla}
\newcommand{\vc}[1]{\mathbf{#1}}
\newcommand{\C}{\mathbb{C}}
\newcommand{\N}{\mathbb{N}}
\newcommand{\R}{\mathbb{R}}

\newcommand{\NB}{\mathcal{H}}

\renewcommand{\S}{\mathbb{S}}

\newcommand{\ps}[2]{\left\langle #1 , #2 \right\rangle}
\newcommand{\norm}[1]{\| #1 \|}
\newcommand{\nnorm}[1]{{\vert\kern-0.25ex\vert\kern-0.25ex\vert #1 
   \vert\kern-0.25ex\vert\kern-0.25ex\vert}}
\renewcommand{\bar}[1]{\overline{#1}}
\newcommand{\sign}{\textup{sign}}

\renewcommand{\epsilon}{\varepsilon}

\newcommand{\bal}{\boldsymbol{\alpha}}
\newcommand{\bdel}{\boldsymbol{\delta}}
\newcommand{\bbe}{\boldsymbol{\beta}}
\newcommand{\bga}{\boldsymbol{\gamma}}

\newcommand{\bt}{\boldsymbol{\theta}}

\newcommand{\cwl}{\widehat{\operatorname{cw}}_{\bb}}
\newcommand{\cwu}{\widecheck{\operatorname{cw}}_{\bb}}

\renewcommand{\O}{\mathcal O}
\newcommand{\Dn}{\Delta_{++}^d}

\newcommand{\bxi}{\boldsymbol{\xi}}
\newcommand{\blam}{\boldsymbol{\lambda}}
\newcommand{\bnu}{\boldsymbol{\nu}}
\newcommand{\ba}{\vc a}
\newcommand{\bb}{\vc b}
\newcommand{\bc}{\vc c}

\newcommand{\be}{\vc e}
\newcommand{\bp}{\vc p}

\newcommand{\bx}{\vc x}
\newcommand{\bxb}{\vc x}

\newcommand{\by}{\vc y}
\newcommand{\byb}{\vc y}

\newcommand{\bz}{\vc z}
\newcommand{\bzb}{\vc z}

\newcommand{\bv}{\vc v}
\newcommand{\bvb}{\vc v}

\newcommand{\bu}{\vc u}
\newcommand{\bub}{\vc u}

\newcommand{\bw}{\vc w}
\newcommand{\bwb}{\vc w}

\newcommand{\bs}{\vc s}

\newcommand{\s}{s}

\newcommand{\algo}{power method}

\newcommand{\A}{\mathcal{A}}
\newcommand{\bmin}[2]{\text{\delimitershortfall=0pt
\delimiterfactor=1001 $\eta_-\!\left(#1\middle/#2\right)$}}
\newcommand{\bmax}[2]{\text{\delimitershortfall=0pt
\delimiterfactor=1001 $\eta_+\!\left(#1\middle/#2\right)$}}
\newcommand{\Bmin}[2]{\text{\delimitershortfall=0pt
\delimiterfactor=1001 $\mathbf{\mathfrak{m}}\!\left(#1\middle/#2\right)$}}
\newcommand{\Bmax}[2]{\text{\delimitershortfall=0pt
\delimiterfactor=1001 $\mathfrak{M}\!\left(#1\middle/#2\right)$}}

\newcommand{\bmini}[3]{\text{\delimitershortfall=0pt
\delimiterfactor=1001 $\mathfrak{m}_{#1}\!\left(#2\middle/#3\right)$}}
\newcommand{\bmaxi}[3]{\text{\delimitershortfall=0pt
\delimiterfactor=1001 $\mathfrak{M}_{#1}\!\left(#2\middle/#3\right)$}}

\newcommand{\ind}[1]{\operatorname{int}(#1)}

\renewcommand{\and}{\quad\text{and}\quad}
\newcommand{\andd}{\qquad\text{and}\qquad}

\newcommand{\E}{\mathcal{E}}

\newcommand{\kone}{\mathcal{K}}

\newcommand{\lek}{\leq_{\kone}}
\newcommand{\lekk}{\lneq_{\kone}}
\newcommand{\lekkk}{<_{\kone}}

\newcommand{\G}{\mathcal{G}}
\newcommand{\Gm}{\mathcal{G}^-}

\newcommand{\krog}{\otimes}

\newcommand{\J}{\mathcal J}

\newcommand{\bj}{\vc j}
\newcommand{\bl}{\vc l}

\newcommand{\I}{\mathcal I}
\newcommand{\saufzero}{\setminus\{0\}}

\newcommand{\diag}{\operatorname{diag}}
\renewcommand{\t}{\tilde}

\newcommand{\multihomo}{multi-homogeneous}

\newcommand{\noarg}{\,\cdot\,}
\newcommand{\NF}{G}

\newcommand{\irrF}{H}

\newcommand{\ones}{\mathbf{1}}

\newcommand{\sauf}{\setminus}

\newcommand{\bphi}{\boldsymbol{\phi}}
\newcommand{\bphib}{\boldsymbol{\phi}}
\newcommand{\Sn}{\S^{\norm{\cdot}_{\bga}}}

\newtheorem{thm2}{Theorem}
\newtheorem{thm}{Theorem}[section]
\newtheorem{lem}[thm]{Lemma}
\newtheorem{prop}[thm]{Proposition}
\newtheorem{cor}[thm]{Corollary}

\theoremstyle{definition}
\newtheorem{defi}[thm]{Definition}
\newtheorem{ex}[thm]{Example}

\theoremstyle{remark}
\newtheorem{rmq}[thm]{Remark}

\numberwithin{equation}{section}

\begin{document}

\title{The Perron-Frobenius Theorem for Multi-homogeneous Maps}

\author{Antoine Gautier}
\address{Department of Mathematics and Computer Science, Saarland University, 66041 Saarbr\"{u}cken, Germany}
\curraddr{}
\email{ag@cs.uni-saarland.de}
\thanks{The authors acknowledge support by the ERC starting grant NOLEPRO 307793 and thank Shmuel Friedland and Lek-Heng Lim for insightful discussions and pointing out relevant references.}

\author{Francesco Tudisco}
\address{Department of Mathematics, University of Padua, via trieste 63 - 35121 - Padova, Italy}
\email{francesco.tudisco@math.unipd.it}

\author{Matthias Hein}
\address{Department of Mathematics and Computer Science, Saarland University, 66041 Saarbr\"{u}cken, Germany}
\email{hein@cs.uni-saarland.de}


\date{\today}

\dedicatory{}

\begin{abstract}
We introduce the notion of order-preserving multi-homogeneous mapping which allows to study Perron-Frobenius type theorems and nonnegative tensors in unified fashion. We prove a weak and strong Perron-Frobenius theorem for these maps and provide a Collatz-Wielandt principle for the maximal eigenvalue. Additionally, we propose a generalization of the power method for the computation of the maximal eigenvector and analyse its convergence. We show that the general theory provides new results and strengthens existing results for various spectral problems for nonnegative tensors.
\end{abstract}

\maketitle
%\tableofcontents
\section{Introduction}
The classical Perron-Frobenius theory addresses properties such as existence, uniqueness and maximality of eigenvectors and eigenvalues of matrices with nonnegative entries. Two important generalizations of this theory arise in the study of eigenvectors of order-preserving homogeneous maps defined on cones \cite{Nussb,NB,Gaubert,KreRut48,Bir57,Bir62,Thompson,Hop63,BK66,Pot77,Bus73,Nus89,Kra01,NVL99,AGLN06,HJ10b,HJ10a,GV12,NB_specradcont}, and in multilinear algebra where spectral problems for tensors with nonnegative coefficients are considered \cite{Boyd,Lim,QIZH,Quynhn,Chang,QIspetraltheory,Lim2013,SymFried1,SymFried2,NQZ,Qi_rect,Qi_rect_1,Chang_rect_eig,linlks,Yao2016,Fried,us,Yang1,Yang2}. Examples include $\ell^p$-eigenvectors, (rectangular) $\ell^{p,q}$-singular vectors and $\ell^{p_1,\ldots,p_d}$-singular vectors of nonnegative tensors (in particular, the $\ell^{p,q}$-singular vector problem for nonnegative matrices is a special case of the latter problem). 
\begin{figure}[h]
\begin{center}
\begin{tikzpicture}[scale =1, ->,>=stealth']
\tikzset{
    state/.style={
           rectangle,
           rounded corners,
           draw=black, thick,
           minimum height=2em,
           inner sep=2pt,
           text centered,
           },
}
\node[state,anchor=center] (Lbox) 
   {\begin{tabular}{c}Nonnegative\\ matrices \end{tabular}};
   
\node[state, right of = Lbox, xshift=5cm] (Hbox) 
 {\begin{tabular}{c}
  Order-preserving\\
homogeneous maps
 \end{tabular}};
 
 \node[state, below of= Lbox, yshift=-0.8cm] (Tbox) 
  {\begin{tabular}{c}
   Nonnegative\\
   tensors
  \end{tabular}};
  
  \node[state, right of= Tbox, xshift=5cm] (HHbox) 
   {\begin{tabular}{c}
   Order-preserving\\
    multi-homogeneous maps
   \end{tabular}};
   
\path[draw=black,solid, line width=0.3mm, fill=black, preaction={-triangle 90,thin,draw,shorten >=-1mm}]
		 (Lbox) edge (Hbox)
		 (Lbox) edge (Tbox)
         (Hbox) edge (HHbox)
         (Tbox) edge (HHbox);
 \end{tikzpicture}
\caption{The Perron-Frobenius theorem was originally developed for nonnegative matrices and then generalized on the one side for order-preserving homogeneous mappings and on the other side for nonnegative tensors. The study of order-preserving multi-homogeneous mappings unifies these theories. }\label{generalization_explanation}
\end{center}
\end{figure}
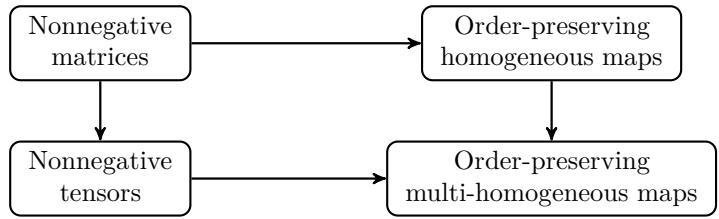
A first connection between these fields has been established in \cite{Fried} where, in order to apply the results of \cite{Nussb, Nus89}, the authors prove the equivalence between a class of spectral problems involving nonnegative tensors and a class of nonlinear eigenvalue problems involving order-preserving homogeneous mappings on the nonnegative orthant $\R^n_+$.
Their approach has been successfully extended to other spectral problems related to nonnegative tensors, see for instance \cite{us,Qi_rect,Yang2}. However, there are ranges of choices for $p,q$ and $p_1,\ldots,p_d$ where this approach can not be used but conclusions about Perron-Frobenius theory, such as existence and uniqueness of a positive eigenvector, still hold. 

Motivated by this observation, we propose a higher-order formulation of the eigenvalue problem for homogeneous mappings on $\kone_+=\R^{n_1}_+\times \ldots \times\R^{n_d}_+$. More precisely, we consider multi-homogeneous order-preserving mappings and prove conditions for existence, uniqueness, maximality and convergence of an algorithm for the computation of positive eigenvectors. The $\ell^p$-eigenvectors of squared nonnegative tensors \cite{Lim,QIZH,Quynhn,Fried,Yang1,Yang2}, the $\ell^{p,q}$-singular vectors of nonnegative matrices \cite{Boyd,us,Bhaskara}, the $\ell^{p,q}$-singular vectors of rectangular nonnegative tensors \cite{Chang_rect_eig,Qi_rect,Qi_rect_1} and the $\ell^{p_1,\ldots,p_d}$-singular vectors of nonnegative tensors \cite{Lim,Fried,us,SymFried1,SymFried2,BanachA} are all particular cases of our formulation. These problems all have a common structure that is discussed in the last section of this paper. Moreover, the eigenvector problem of an order-preserving homogeneous map \cite{Nussb,Nus89,NB} on the cone $\R^n_{+}$ is also a special case of the class considered here. 

We show that many of these spectral problems can be reformulated as fixed point problems involving strictly contractive maps defined on the interior of $\kone_+$, endowed with a weighted Thompson \cite{Thompson} or Hilbert metric \cite{Bir57, samelson1957perron}. While the contractive case has been exploited for order-preserving $p$-homogeneous maps (with $p\in(0,1)$) defined on cones \cite{Bus73,Kra01}, we are not aware of such observation in the nonnegative tensor community. Indeed, a drawback of the method proposed in \cite{Fried} is that contractive information is often lost when the higher-order spectral problem is transformed into a first order problem. The approach proposed here exploits this structure. In particular, our results improve results in \cite{Boyd, us, Fried, Qi_rect, Yang1, Yang2}. Additionally, the reformulation of spectral problems for nonnegative tensors as fixed point problems allows in the contractive case to formulate a generalized power method with a linear convergence rate. This result is either novel or has been proved under more restrictive assumptions in the literature. We study in detail the spectral properties of multi-homogeneous order-preserving mappings which are non-expansive but not contractive. In this case, when applied to nonnegative tensor problems, our results imply various known results of the literature with often similar or weaker assumptions \cite{Boyd, us, Fried, Qi_rect, Yang1, Yang2, Hu2014,Wu2013}. Typically, when our assumptions are more restrictive, it is because the specific structure of a particular problem is used and such results do not hold for other problems involving nonnegative tensors. Finally, we note that both in the contractive and the non-expansive case, some of the presented results are not known even in particular cases (e.g. a Collatz-Wielandt principle for $\ell^{p,q}$-singular vectors of rectangular nonnegative tensors, convergence rate for the generalized power method, etc.). 

Generalizing the Perron-Frobenius theorem to multi-homogeneous mappings is a delicate task. Indeed, due to the degrees of freedom induced by higher-order spectral problems, most of the usual assumptions of the Perron-Frobenius theory have to be reformulated carefully as their straightforward generalizations are either too restrictive or the arguments in the proofs that requires them do not work for multi-homogeneous mappings. Furthermore, known relationships between these assumptions in the homogeneous case, do not hold anymore in the higher-order case. On the other hand, as already discussed above, we note that for a vast class of higher-order problems, identifying $\kone_{+}$ with $\R^{n_1+\ldots+n_d}_+$ prevent the application of existing nonlinear Perron-Frobenius theorems as the mapping is expansive with respect to both the Thompson metric and the Hilbert projective metric on $\R^{n_1+\ldots+n_d}_+$. However, it is non-expansive or even contractive with respect to their weighted counter-parts defined on $\kone_+$ provided that the weights are chosen in a suitable manner. These differences and technical difficulties are discussed and illustrated through particular examples all along this text. 

The paper is organized as follows: At the beginning of each section, we state and discuss the main results presented therein. Their proofs are given within the section as they require to establish some intermediate results first. Each section deals with one aspect of the Perron-Frobenius theory besides the last one which is devoted to the application of our results to nonnegative tensors. More precisely,
in Section \ref{contract_section} we discuss first properties of order-preserving multi-homogeneous mappings. Then, we introduce the weighted Hilbert and Thompson metrics and prove a contraction principle. In Section \ref{existence_section}, we discuss conditions for the existence of nonnegative and positive eigenvectors and prove a weak form of the Perron-Frobenius theorem. Section \ref{CW_M_U} is concerned with a generalization of the Collatz-Wielandt principle. In particular, we discuss the maximality of eigenvalues associated with nonnegative and positive eigenvectors. Moreover, we give a condition for the uniqueness of a positive eigenvector. We discuss a generalization of the power method for the computation of positive eigenvectors and analyse its convergence in Section \ref{PM_section}. Finally, in Section \ref{appli} we show how our theory can be applied to the study of spectral problems involving nonnegative tensors and compare our results with those of the existing literature.

\subsection{Notation and preliminaries}\label{notation}
For the sake of clarity and in order to avoid ambiguities, we fix here the main notation used throughout this paper. 

For $n\in\N$, define $[n]=\{1,\ldots,n\}$,
$\R^n_+=\big\{\bz\in\R^n\ \big| \ z_i\geq 0, \ \forall i\in [n]\big\}$, $\R^n_{+,0}=\R^n_+\saufzero$, $\R^n_{++}=\big\{\bz\in\R^n\ \big| \ z_i>0, \  \forall i\in [n]\big\}$, $ \Delta_{++}^n=\big\{\bz\in\R^n_{++}\ \big|\ \sum_{i=1}^dz_i=1\big\}.$
For $p\in[1,\infty]$, we write $\norm{\bz}_p$ to denote the usual $p$-norm of $\bz\in\C^{n}$. %i.e.  $\norm{\bz}^p_p=\sum_{j=1}^n|z|^2_{j}$. 
Moreover, for $p\in(1,\infty)$ we write $p'$ to denote the H\"{o}lder conjugate of $p$, i.e. $p'=p/(p-1)$. Let
\begin{equation*}
\psi_p\colon\R^n\to\R^n,\qquad \psi_p(\bz)=\big(|z_1|^{p-1}\sign(z_1),\ldots,|z_n|^{p-1}\sign(z_n)\big),
\end{equation*}
where $\sign(t)= t/|t|$ if $t\neq 0$ and $\sign(0)=0$. We write $|\bz|$ to denote the component-wise absolute value of $\bz$, i.e. $|\bz|=(|z_1|,\ldots,|z_n|)$. On $\R^n_+$ we consider the partial ordering induced by $\R^n_+$, i.e. for every $\by,\bz\in\R^n$ we write $\bz\leq\by$, $\bz\lneq \by$ and $\bz <\by$ if, and only if $\by-\bz\in\R^n_+$, $\by-\bz\in\R^n_+\saufzero$ and $\by-\bz\in\R^n_{++}$, respectively. We write $I\in\R^{n\times n}$ and $\ones\in\R^n$ to denote the identity matrix and the vector of all ones respectively. We write $\rho(A)$ for the spectral radius of a matrix $A\in\R^{n\times n}$. We recall that a matrix $A\in\R^{n\times n}_+$ is irreducible if $(I+A)^{n-1}\in\R^{n\times n}_{++}$ and primitive if there exists $\nu\in\N$ such that $A^\nu\in\R^{n\times n}_{++}$ where $\R^{n\times n}_{+}$ and $\R^{n\times n}_{++}$ denote the sets of matrices with nonnegative, respectively positive, entries. 

For $\bal\in\R^n_{++}$ and $B\in\R^{n\times n}$ (or $\bal\in\R^{n}_+$ and $B\in\R^{n\times n}_+$), define $\bal^B\in\R^n_+$ as
\begin{equation*}
\bal^{B}= \left(\prod_{k=1}^n\alpha_k^{B_{1,k}},\ldots,\prod_{k=1}^n\alpha_k^{B_{n,k}}\right).
\end{equation*}
A direct computation shows that for every $\bal,\bbe\in\R^n_{++}$ and every $B,C\in\R^{n\times n}$, following identities hold
\begin{equation}\label{homoident}
\bal^{B}\circ \bal^{C}=\bal^{B+C}, \qquad \big(\bal^{C}\big)^B=\bal^{BC} \qquad \text{and}\qquad \big(\bal\circ\bbe\big)^B=\bal^B\circ\bbe^B,
\end{equation}
where $\circ$ denotes the entrywise product, i.e. $\bal\circ \bbe =(\alpha_1\beta_1,\ldots,\alpha_n\beta_n)$. Moreover, if $\ba\in\R_{++}^n$ and $\lambda>0$, then 
\begin{equation*}
\prod_{i=1}^n\big(\bal^B\big)_i^{a_i}=\prod_{i=1}^n\alpha_i^{(B^T\ba)_i}\quad\text{ and }\quad (\lambda^{a_1},\ldots,\lambda^{a_n})^B =(\lambda^{(B\ba)_1},\ldots,\lambda^{(B\ba)_n}).
\end{equation*}
Now, for $d\in\N$ and $n_1,\ldots,n_d\in\N$, define $V= \R^{n_1}\times \ldots \times \R^{n_d}$,
$\kone_+ = \R^{n_1}_+\times \ldots \times\R^{n_d}_+$,
$\kone_{+,0} = \R^{n_1}_{+,0}\times \ldots\times \R^{n_d}_{+,0}$ and
$\kone_{++} = \R^{n_1}_{++}\times \ldots \times\R^{n_d}_{++}.$
We use bold letters without index to denote  elements of $V$, bold letters with index $i\in [d]$ denote vectors in $\R^{n_i}$, whereas components of $\bx_i$ are written in normal font. Namely \begin{equation*}
\bx = (\bx_1,\ldots,\bx_d)\in V, \qquad \bx_i=(x_{i,1},\ldots,x_{i,n_i})\in \R^{n_i} \andd x_{i,j_i} \in \R.\end{equation*}
For $i\in[d]$, let $\norm{\cdot}_{\gamma_i}$ be a norm on $\R^{n_i}$. We consider
\begin{equation*}
\Sn_+=\big\{\bxb\in\kone_{+}\ \big|\ \norm{\bx_i}_{\gamma_i}=1, \forall i \in[d]\big\}, \qquad \Sn_{++}= \Sn_+\cap\kone_{++},
\end{equation*}
and, for $\bphi\in\kone_{++}$,
\begin{equation*}
\S^{\bphi}_{+}=\big\{\bx\in\kone_+ \ \big|\ \ps{\bx_i}{\bphi_i}=1, \ \forall i \in[d]\big\},\qquad \S^{\bphi}_{++}= \S^{\bphi}_+\cap\kone_{++}.
\end{equation*}
The norm $\norm{\cdot}_{\gamma_i}$ is said to be monotonic if $\norm{\bx_i}_{\gamma_i} \leq \norm{\by_i}_{\gamma_i}$ whenever $0\leq \bx_i\leq \by_i$. Although not always needed, we assume that all the norms considered in the following are monotonic.
Likewise $\R^n_+$, $\kone_+$ induces a partial ordering on $V$. We write $\bx\lek\bu$, $\bx\lekk \bu$, $\bx\lekkk\bu$ if and only if $\bu-\bx\in\kone_+$, $\bu-\bx\in\kone_{+}\saufzero$ and $\bu-\bx\in\kone_{++}$ respectively.

We consider the index sets $\I, \J$ defined as
\begin{equation*}
\I = \bigcup_{k=1}^d \big(\{k\}\times [n_k]\big) \andd \J= [n_1]\times\ldots\times [n_d],
\end{equation*}
and the product $\krog\colon \R^d\times V\to V$ defined as
\begin{equation*}
\bal\krog\bx= (\alpha_1\bx_1,\ldots,\alpha_d\bx_d).
\end{equation*}
For $F\colon V\to V$ we use the same notation as for vectors, i.e.
$F=(F_1,\ldots,F_d)$ and $F_i=(F_{i,1},\ldots,F_{i,n_i})$ with $F_i\colon V\to\R^{n_i}$ and $F_{i,j_i}\colon V\to\R$. For $k\in\N$, we denote the iterates of $F$ as $F^{k}$, where 
$F^{0}(\bxb)=\bxb$ and $F^{k}(\bxb)= F\big(F^{k-1}(\bx)\big)$.
If $F$ is differentiable at $\bv\in V$, we write $D_kF_i(\bv)\in\R^{n_i\times n_k}$ to denote the Jacobian matrix of the map $\bx_k\mapsto F_i(\bx)$ at $\bx= \bv$. Similarly, if $f\colon V \to \R$, we write $\grad_if(\bx)$ to denote the gradient of $\bx_i\mapsto f(\bx)$.

Both of the following notions are fundamental in our study, namely the concepts of order-preserving and multi-homogeneous maps.
\begin{defi}
Let $U,U'\subset V $ and $F\colon U\to U'$. We say that $F$ is order-preserving if, for every $\bx,\by\in U$ such that $\bx\lek \by$, we have $F(\bx)\lek F(\by)$.
\end{defi}
\begin{defi}
Let $U,U'\subset V $ and $F\colon U\to U'$. $F$ is said to be (positively) multi-homogeneous if there exists a matrix $A\in\R^{d\times d}$ such that for every $\bx\in U$ and every $\bal\in\R^d_{++}$ with $\bal\krog\bx\in U$, we have $F(\bal\krog\bx)=\bal^A\krog F(\bx)$. We call $A$ the homogeneity matrix of $F$, and denote it as $\A(F)$.
\end{defi}
Inspired by the assumptions arising in the Perron-Frobenius theory for order-preserving homogeneous mappings \cite{NB}, we introduce the following set. 
\begin{defi}
Let $\NB^d$ be the set of mappings $F\colon\kone_+\to\kone_+$ such that:
\begin{enumerate}
\item $F$ is continuous and order-preserving, 
\item $F$ is \multihomo{},
\item $F(\ones)\in\kone_{++}$ and $\A(F)$ has at least one nonzero entry per row.
\end{enumerate}
\end{defi}
Finally, inspired by the spectral theory for nonnegative tensors, we introduce the following notion of eigenvector.
\begin{defi}\label{defeigevect}
Let $F\colon V\to V$ be continuous and \multihomo{}, $\bx\in V$ is called an eigenvector of $F$, if $\bx_i\neq 0$ for all $i\in[d]$ and there exists $\blam\in\R^d$ such that $F(\bx)=\blam\krog\bx$.
\end{defi}
Note that if $F\in\NB^1$, i.e. $d=1$, then $F$ is $p$-homogeneous with $p=\A(F)$ and $\bx\in\kone_{+}=\R^{n_1}_+$ is an eigenvector of $F$ if $\bx \neq 0$ and there exists $\lambda\in\R_+$ such that $F(\bx)=\lambda\bx$.
Observe also that if $F\in\NB^d$, then $F$ is \multihomo{}, and for every eigenvector $\bxb$ of $F$ and every $\bal\in\R^d_{++}$, $\bal\krog\bxb$ is also an eigenvector of $F$. Indeed, if $F(\bxb)=\blam\krog\bxb$, then $F(\bal\krog\bxb)=(\bal^{A-I}\circ\blam)\krog(\bal\krog\bxb)$ with $A=\A(F)$. Thus, the associated eigenvalue may not be scaling invariant. For this reason we say that two eigenvectors $\by,\bx\in\kone_{+,0}$ of $F$ are equivalent, if there exists $\bal\in\R^d_{++}$ such that $\bx=\bal\krog\by$. 

The properties of a multi-homogeneous map $F\in\NB^d$ are governed by those of its homogeneity matrix. In particular, as $\A(F)$ is nonnegative, the linear Perron-Frobenius theorem, which is recalled here for the convenience of the reader, will be a useful tool in the following. While this result can be found in most of the modern textbooks on matrices such as \cite{Horn}, we refer to \cite{Perron1907},\cite{Frobenius1908},\cite{Collatz},\cite{Wielandt1950} for the original works of Perron, Frobenius, Collatz and Wielandt respectively.
\begin{thm2}\label{linear_PF}
Let $n\geq 2$ and $M\in\R^{n\times n}$ be an irreducible nonnegative matrix. Then:
\begin{enumerate}[(i)]
\item $M$ has an eigenvector $\bu\in\R^n_{++}$, i.e. $M\bu=\lambda\bu$.
\item It holds $\lambda = \rho(M)$ and the Collatz-Wielandt $\min$-$\max$ characterization
\begin{equation*}\label{linear_CW}
 \min_{\bx\in\R^n_{++}}\max_{i=1,\ldots,n}\frac{(M\bx)_i}{x_i}= \lambda = \max_{\bx \in\R^n_{+}\saufzero}\min_{\substack{i=1,\ldots,n\\ x_i>0}} \frac{(M\bx)_i}{x_i}.
\end{equation*}
Moreover, $\bu$ is the unique nonnegative eigenvector of $M$. 
\item If additionally $M$ is primitive, the sequence $(\bx^k)_{k=0}^{\infty}\subset \R^n_{++}$\label{PMlin} defined as 
\begin{equation*}
\bx^0\in\R^n_{++}\andd \bx^{k+1}=\frac{M\bx^k}{\norm{M\bx^k}_2} \qquad \forall k\in \N,
\end{equation*}
satisfies $\lim_{k\to\infty} \bx^k = \bu$.
\end{enumerate}
\end{thm2}
\section{Multi-homogeneous maps and contraction principle}\label{contract_section}
We start by discussing first properties of mappings in $\NB^d$. This discussion provides intuition on the richness of the class of maps $\NB^d$ and is helpful to determine whether a mapping belongs to $\NB^d$ or not. More precisely, we recall a known characterization of order-preserving maps in Theorem \ref{opcharac}. Then, in Lemma \ref{Eulerthm}, we prove a generalization of Euler's theorem for homogeneous maps. We will need this result in order to derive a condition for the uniqueness of the positive eigenvector of a map in $\NB^d$. In Theorem \ref{extend}, we recall a result on continuous extensions of homogeneous order-preserving maps which are defined in the interior of the cone. This result will be used to obtain a sufficient condition for the existence of a positive eigenvector of a mapping in $\NB^d$. Finally, in Lemma \ref{homoNBd}, we describe some operations that leave $\NB^d$ invariant. 

In a second step, we consider a particular instance of a spectral problem involving multi-homogeneous maps, namely the $\ell^{p,q}$-singular values of a nonnegative matrix. Through this example, we emphasize differences between the classical notion of eigenvectors for nonlinear maps and the definition proposed here, as well as techniques commonly used for the computation of the projective norm of nonnegative matrices and tensors. Moreover, it is a first illustration of how \multihomo{} mappings appear naturally in the study of higher-order spectral problems. We then briefly discuss how one can homogenize non homogeneous mappings in order to apply our theory.

In Section \ref{Hilbert_metric}, we recall the definitions of the Hilbert projective metric and the Thompson metric. We introduce two weighted versions of these metrics extended to $\kone_{++}$ that will be crucial for the proofs in the following. In Lemma \ref{contract}, we compute the Lipschitz constant of a mapping $F\in\NB^d$ with respect to the introduced weighted metrics. We show that the Lipschitz constant is minimized by a certain choice of the weights. The optimal Lipschitz constant is the spectral radius of the homogeneity matrix of $F$. This is the key observation for the following main result of this section.

\begin{thm}\label{Banachcor}
Let $F\in\NB^d$ and $A=\A(F)$. If $\rho(A)<1$, then there exists $\bb\in\R^d_{++}$ with $A^T\bb\leq r\bb$ for some $r\in [\rho(A),1)$ and there exists a unique $\bub\in\Sn_{++}$ such that $F(\bub)=\blam\krog\bub$ for some $\blam \in\R^d_{++}$. For $\bxb^0\in\Sn_{++}$, let
\begin{equation*}
\bxb^k\,=\, \bigg(\frac{F_1(\bx^{k-1})}{\norm{F_1(\bx^{k-1})}_{\gamma_1}},\ldots,\frac{F_d(\bx^{k-1})}{\norm{F_d(\bx^{k-1})}_{\gamma_d}}\bigg)\qquad \forall k \in\N,
\end{equation*}
then
\begin{equation*}
\lim_{j\to\infty}\bx^j\,=\,\bub\andd \mu_{\bb}\big(\bub,\bxb^{k}\big) \,\leq \,\frac{r^k}{1-r}\,\mu_{\bb}\big(\bx^{1},\bx^{0}\big) \qquad \forall k\in\N.
\end{equation*}
where $\mu_{\bb}\colon\kone_{++}\times\kone_{++}\to\R_+$ is the weighted Hilbert metric defined as 
\begin{equation*}
\mu_{\bb}(\bx,\by)\,=\,\sum_{k=1}^d b_i\,\ln\!\bigg[\Big(\max_{j_i\in[n_i]}\frac{x_{i,j_i}}{y_{i,j_i}}\,\Big)\Big(\max_{l_i\in[n_i]}\frac{y_{i,l_i}}{x_{i,l_i}}\Big)\bigg] \qquad \forall \bx,\by\in\kone_{++}.
\end{equation*}
In particular, if $A$ has a positive left-eigenvector $\bc\in\R^d_{++}$, then one can choose $\bb=\bc$ so that $r=\rho(A)$.
\end{thm}
Finally, we discuss an example of a mapping which is expansive with respect to the Thompson and Hilbert metrics on $\R^4_{++}$ but is a contraction on $\kone_{++}=\R^{2}_{++}\times \R^2_{++}$ with respect to the metric of Theorem \ref{Banachcor}.
\subsection{Multi-homogeneous maps and their eigenvectors}
First of all we recall a known theorem that characterizes the property of being order-preserving. 
\begin{thm}[Theorem 1.3.1, \cite{NB}]\label{opcharac}
Let $U\subset \kone_+$ be an open convex set. If $F\colon U \to \kone_{+}$ is locally Lipschitz, then $DF(\bx)$ exists for Lebesgue almost all $\bx\in U$, and $F$ is order-preserving if and only if $DF(\bx)\kone_{+}\subset\kone_{+}$ for all $\bx\in U$ for which $DF(\bx)$ exists.
\end{thm}
The following lemma can be seen as a generalization of the well-known Euler theorem for homogeneous mappings.
\begin{lem}\label{Eulerthm}
Let $U\subset V$ be open and such that $\bal\krog\bx=(\alpha_1\bx_1,\ldots,\alpha_d\bx_d)\in U$ for all $\bal\in\R^d_{++}$ and $\bx\in U$. Let $\ba\in\R^d$ and $f\colon U\to\R$, a differentiable map. The following are equivalent:
\begin{enumerate}
\item It holds\label{homomulti1} $f(\bal\krog\bx)=f(\bx)\displaystyle\prod^d_{k=1}\alpha_k^{a_k}$ for every $\bal\in\R^d_{++},\ \bx\in U.$
\item It holds
\label{homomulti2}
$\ps{\grad_if(\bx)}{\bx_i}=a_if(\bx)$ for every $i \in[d],\ \bx\in U.$
\end{enumerate}
Moreover, if $f$ satisfies \eqref{homomulti1} or \eqref{homomulti2}, then:
\begin{enumerate}\setcounter{enumi}{2}
\item It holds $\grad_i f(\bal\krog\bx)=\displaystyle\grad_i f(\bx)\alpha_i^{-1}\prod_{k=1}^d\alpha_k^{a_k} $ for every $i\in[d],\bal\in\R^d_{++}, \bx\in U.$
\end{enumerate}
\end{lem}
\begin{proof}
Let $\bx\in U$ and consider $g_{\bx},h_{\bx}\colon\R^d_{++}\to \R$ defined as
\begin{equation*}
g_{\bx}(\bal)=  f(\bal\krog\bx)-f(\bx)\prod_{k=1}^d\alpha_k^{a_k}\andd
h_{\bx}(\bal)= f(\bal\krog\bx)\prod_{k=1}^d\alpha_k^{-a_k}-f(\bx).
\end{equation*}
If \eqref{homomulti1} holds, then $g_{\bx}$ is constant and \eqref{homomulti2} follows from $\grad g_{\bx}(\ones)=0$. If \eqref{homomulti2} holds, then $\grad h_{\bx}(\bal)=0$ for every $\bal$ and \eqref{homomulti1} follows from $h_{\bx}(\bal)=h_{\bx}(\ones)=0$. To show the last part, let $(i,j_i)\in \J$ and consider $\be^{(i,j_i)}\in \kone_+$, the vector such that $(\be^{(i,j_i)})_{k,l_k}=1$ if $(k,l_k)=(i,j_i)$ and $(\be^{(i,j_i)})_{k,l_k}=0$ else. Then, for every small enough $h$, it holds 
\begin{equation*}
\frac{f(\bal\krog\bx+h\be^{(i,j_i)})-f(\bal\krog\bx)}{h}
= \bigg(\alpha_i^{-1}\prod_{k=1}^d\alpha_k^{a_k}\bigg) \frac{f(\bx+\alpha_i^{-1}h\be^{(i,j_i)})-f(\bx)}{\alpha_i^{-1}h}.
\end{equation*}
Letting $h\to 0$ concludes the proof.
\end{proof}
Before discussing continuous extensions of order-preserving, \multihomo{} maps defined in the interior of the cone, we show in the following lemma that for any $F\in\NB^d$, it holds $F(\kone_{++})\subset\kone_{++}$ and $\A(F)$ is nonnegative. 
\begin{lem}\label{basics_lem}
Let $F\in\NB^d$, then $F(\kone_{++})\subset \kone_{++}$ and $\A(F)\in \R^{d\times d}_+$. Moreover, if $F$ is differentiable at some point $\bx\in\kone_{++}$, then for every $i,j\in [d]$, we have $\A(F)_{i,j}>0$ if and only if $D_jF_i(\bx)$ has at least one nonzero entry per row.
\end{lem}
\begin{proof}
Let $\by\in\kone_{++}$, there exists $\bal\in\R^d_{++}$ such that $\bal\krog\ones\lek\by$ and thus $0\lekkk\bal^{A}\krog F(\ones)\lek F(\by)$ as $F(\ones)\in\kone_{++}$ for $F\in\NB^d$.
Now, let $A=\A(F)$ and fix $i,j\in[d]$. Let $\bt\colon\R_{++}\to\R^d_{++}$ be defined as $\theta_l(t)=1$ if $l\neq j$ and $\theta_j(t)=1+t$. Then, we have $\ones \lek \bt(t)\krog\ones$ for every $t>0$. As $F$ is order-preserving, it follows
\begin{equation*}
F_i(\ones) \lek F_i(\bt(t)\krog \ones) = (1+t)^{A_{i,j}} F_i(\ones) \qquad \forall t>0. 
\end{equation*}
Hence, $1 \leq (1+t)^{A_{i,j}}$ for every $t>0$ which implies $A_{i,j}\geq 0$. Finally, Lemma \ref{Eulerthm} implies that $D_jF_i(\bx)\bx_j=A_{i,j}F_i(\bx)$. The discussion above shows that $F(\bx)\in\kone_{++}$ if $\bx\in\kone_{++}$. It follows that $A_{i,j}>0$ if and only if $D_jF_i(\bx)\in\R^{n_i\times n_j}$ has at least one nonzero entry per row. 
\end{proof}
There exist order-preserving multi-homogeneous mappings which are naturally defined on $\kone_{++}$ rather than on $\kone_+$. This frequently happens in the case $d=1$ when considering the log-exp transform of a topical map (see e.g. \cite{Gaub_survey}). We also face such a situation when deriving the dual condition for the existence of a positive eigenvector in Corollary \ref{existdual}. It is then useful to know whether the considered mapping can be continuously extended to a map in $\NB^d$. In the case $d=1$, such an extension has been proved to exist in Theorem 3.10 \cite{cont_ext} and Theorem 5.1.2 \cite{NB}. As the proof of this result can be easily generalized for $d>1$ (with the help of Lemma \ref{contract}), we omit it here.
\begin{thm}\label{extend}
Let $F\colon\kone_{++}\to\kone_{++}$ be order-preserving and \multihomo{}. If $\A(F)$ has at least one positive entry per row and there exists $\bb\in\R^d_{++}$ such that $\A(F)^T\bb\leq \bb$, then there exists $\bar F\in\NB^d$ such that $F=\bar F|_{\kone_{++}}$ and $\A(\bar F)=\A(F)$.
\end{thm}
We describe operations under which $\NB^d$ is closed in the following straightforward lemma.
\begin{lem}\label{homoNBd}
Let $F,G\in\NB^d$, $A=\A(F)$ and $B=\A(G)$. Moreover, let $D\in\R^{d\times d}_+$ with $D\geq A,B$ and $\xi_1,\ldots,\xi_d\colon\kone_{+}\to\R_+$ be continuous, order-preserving, homogeneous maps such that $\xi_i(\kone_{+,0})\subset\R_{++}$ for every $i\in[d]$. Define $N\colon\kone_+\to\R^d_+$ as $N(\bx)=\big(\xi_1(\bx),\ldots,\xi_d(\bx)\big)$. Finally, let $H^{(1)},H^{(2)},H^{(3)}\colon\kone_{+}\to\kone_+$ with
\begin{equation*}
H^{(1)}(\bx)= F\big(G(\bx)\big), \qquad H^{(2)}(\bx)=F(\bx)\circ G(\bx)
\end{equation*} 
and
\begin{equation*}
H^{(3)}(\bx)=N(\bx)^{D-A}\krog F(\bx)+N(\bx)^{D-B}\krog G(\bx).
\end{equation*}
Then $H^{(1)},H^{(2)},H^{(3)}\in\NB^d$ with homogeneity matrices $AB$, $A+B$, $D$ respectively.
\end{lem}
In particular, it follows that for every $F\in\NB^d$, we have $\A(F^k)=\A(F)^k$.

\subsection{First motivating examples}
In this section, we make observations on the properties of maps in $\NB^d$ by analysing simple examples. The first one is the well-known matrix $\ell^{p,q}$-singular value problem \cite{Boyd} which we recast in terms of multi-homogeneous mappings.
\begin{ex}[Matrix $\ell^{p,q}$-singular value problem]\label{lpq_pb}
Let $M\in\R_+^{m\times n}$ and $p,q\in(1,\infty)$. The $\ell^{p,q}$-singular vectors of $M$ are the critical points of the Rayleigh quotient associated to the $(p,q)$-norm of $M$. More precisely, let $R\colon \R^m\saufzero\times \R^n\saufzero\to \R$ be defined as
\begin{equation*}
R(\bx,\by) = \frac{\bx^TM\by}{\norm{\bx}_p\norm{\by}_q}\andd
 \norm{M}_{p,q}=\max_{\bx,\by\neq 0} R(\bx,\by),
\end{equation*}
then $\norm{M\by}_{p'} \leq \norm{M}_{p,q}\norm{\by}_q$ for every $\by\in \R^n$ and $(\bx,\by)$ is an $\ell^{p,q}$-singular vector of $M$ if $\grad R(\bx,\by)=0$. Note that, as $R(\bx,\by)\leq R(|\bx|,|\by|)$ for every $\bx,\by$, the maximum above is attained in $\kone_{+}=\R^{m}_{+}\times\R^n_{+}$. So, define $F\colon\kone_+\to \kone_+$ as
\begin{equation}\label{lpq_F}
F(\bx,\by)=\big(\psi_{p'}(M\by),\psi_{q'}(M^T\bx)\big)\qquad \forall (\bx,\by)\in\kone_+.
\end{equation}
Then, the eigenvectors of $F$ correspond to the critical points of $R$ in the product of spheres $\Sn_+=\{(\bx,\by)\in\kone_+\mid \norm{\bx}_{p}=\norm{\by}_q=1\}$. Assume that $(\bx,\by)\in\Sn_+$ satisfies $\grad R(\bx,\by) = 0$, then
\begin{equation*}
F(\bx,\by)= (\lambda_1\bx,\lambda_2\by) \quad \text{with} \quad (\lambda_1,\lambda_2)=\big(R(\bx,\by)^{p'-1},R(\bx,\by)^{q'-1}\big),
\end{equation*}
i.e. the critical points of $R$ in $\Sn_{+}$ are eigenvectors of $F$. Conversely, if $(\bx,\by)\in\Sn_{+}$ satisfies $F(\bx,\by)=(\lambda_1\bx,\lambda_2\by)$, then 
\begin{equation*}
\lambda_1^{p-1}=\lambda_1^{p-1}\norm{\bx}_p=\big\langle\bx,\lambda_1^{p-1}\psi_p(\bx)\big\rangle=\ps{\bx}{M\by} =R(\bx,\by)= \ps{\by}{M^T\bx}=\lambda_2^{q-1},
\end{equation*}
that is the eigenvectors of $F$ in $\Sn_+$ are critical points of $R$. 

Note that $F$ is order-preserving because $(\bx,\by)\lek (\t\bx,\t\by)$ implies $M^T\bx \leq M^T\t\bx$ and $M\by\leq M\t\by$, as $M$ has nonnegative entries. Moreover, $F$ is multi-homogeneous as for every $\alpha,\beta>0$ and $(\bx,\by)\in\kone_+$ we have
\begin{align*}
F(\alpha\bx,\beta\by)&=\big(F_1(\alpha\bx,\beta\by),F_2(\alpha\bx,\beta\by)\big)=\big(\alpha^0\beta^{p'-1} F_1(\bx,\by),\alpha^{q'-1}\beta^0 F_2(\bx,\by)\big)\\&= (\alpha,\beta)^{A}\krog F(\bx,\by) \qquad \text{with}\qquad A = \begin{pmatrix} 0 & p'-1 \\ q'-1 & 0 \end{pmatrix}=\A(F).
\end{align*}
Finally, for $F$ to be in $\NB^2$, we have to ensure $F(\ones)\in\kone_{++}$. This is the case when $M$ has at least one nonzero entry per row and per column. When the latter assumption is not fulfilled, one can introduce a decoupled problem as follows: Let $\t F\colon\R^m_{+}\to\R^m_{+}$ with $\t F(\bx)= F_1\big(\ones,F_2(\bx,\ones)\big) $,
then, for $\blam\in\R^2_+$ and $(\bx,\by)\in\kone_{+,0}$, we have
\begin{equation}\label{decouple}
F(\bx,\by)=(\lambda_1\bx,\lambda_2\by) \qquad \iff \qquad \begin{cases} \t F(\bx)=\lambda_2^{p'-1}\lambda_1\bx\\ \lambda_2\by = F_2(\bx,\ones).\end{cases}
\end{equation}
This is possible because $A_{1,1}=A_{2,2}=0$. This transformation can be useful as there are matrices $M$ for which $F\notin\NB^d$ and $\t F\in\NB^d$. For example, when $m=n=2$, $M_{1,1}=M_{2,1}=1$ and $M_{1,2}=M_{2,2}=0$, then $F(\kone_{++})\not\subset \kone_{++}$ while $\t F(\R^m_{+,0})\subset\R^m_{++}$. 
\end{ex}
We use the following remark to emphasize three relevant observations made in Example \ref{lpq_pb}.
\begin{rmq}\label{lpq_rmq}
(a) When looking at the necessary condition for the critical points of a Rayleigh quotient, we obtain an equation of the form $G(\bx)=\blam \krog \Psi(\bx)$. E.g. for the $\ell^{p,q}$-singular vectors of a matrix $M$, we have $G(\bx_1,\bx_2)=(M\bx_2,M^T\bx_1)$ and $\Psi(\bx_1,\bx_2)=(\psi_p(\bx_1),\psi_q(\bx_2))$. When $G\in\NB^d$ and $\Psi$ is an invertible map such that $\Psi^{-1}\in\NB^d$, we can then recast the necessary condition into a spectral problem for $F\in\NB^d$ with $F(\bx)=\Psi^{-1}\big(G(\bx)\big)$. Such approach is widely used for the computation of projective norms of nonnegative matrices and tensors (see for instance \cite{Boyd,Lim,Fried,us,Quynhn,Qi_rect}).\\
(b) To establish the correspondence between the eigenvectors of $F$ defined in \eqref{lpq_F} and the $\ell^{p,q}$-singular vectors of $M$, we have shown that if $(\bx_1,\bx_2)\in\Sn_{+}$ satisfies $F(\bx_1,\bx_2)=(\lambda_1\bx_1,\lambda_2\bx_2)$, then $\lambda_1^{p-1}=\lambda_2^{q-1}$. In order to prove this implication, we used a technique that is common in the study of spectral problems of nonnegative tensors, especially to prove existence of a positive eigenvector (see \cite{us,Fried,Qi_rect}). This technique can be formulated for maps in $\NB^d$ as follows: Let $G\in\NB^d$ and suppose that there exists $\ba\in\R^d_{++}$ and $R\colon\kone_{+,0}\to\R_+$ such that
\begin{equation*}
\ps{G_i(\bx)}{\bx_i}=a_iR(\bx) \qquad \forall i\in[d],\ \bx\in\kone_{+,0}.
\end{equation*}
Then, for every $\Psi\colon\kone_{+,0}\to\kone_{+,0}$ and $(\blam,\bx)\in\R^d_+\times\kone_{+,0}$ we have
\begin{equation}\label{Rayequal}
\begin{cases}G(\bx)=\blam\krog\Psi(\bx)\\ \ps{\bx_i}{\Psi_i(\bx)}=1 \ \ \forall i\in[d]\end{cases} \qquad \implies \qquad \frac{\lambda_i}{a_i}=R(\bx)\quad \forall i \in[d]. %\frac{\lambda_1}{a_1}=\ldots=\frac{\lambda_d}{a_d}=R(\bx).
\end{equation}
Indeed, if the system of equations in \eqref{Rayequal} is satisfied, then
\begin{equation*}
\lambda_i=\ps{\bx_i}{\lambda_i\Psi_i(\bx)}=\ps{\bx_i}{G_i(\bx)}=a_iR(\bx) \qquad \forall i\in[d].
\end{equation*}
(c) The decoupling technique in \eqref{decouple} is known, especially for computing projective norms of matrices and tensors (see for instance \cite{Boyd,Bhaskara,us,Higham}). This principle can under certain conditions also be applied to maps in $\NB^d$. More precisely, let $F\colon\kone_+\to\kone_+$ be continuous, order-preserving and multi-homogeneous and suppose that there exists $i\in[d]$ such that $A=\A(F)\in\R^{d\times d}_+\saufzero$ satisfies $A_{i,i}=0$. Set
$\t\kone_+=\R^{n_1}_+\times \ldots\times \R^{n_{i-1}}_+\times \R^{n_{i+1}}_+\times \ldots\times \R^{n_d}_+$ and let $\t F\colon \t\kone_+\to \t\kone_+$ with $\t F=\big(\t F_1,\ldots,\t F_{i-1},\t F_{i+1},\ldots,\t F_d\big)$ and, for $\t\bx=(\t \bx_{1},\ldots,\t\bx_{i-1},\t\bx_{i+1},\ldots,\t\bx_d)\in\t\kone_+$,
\begin{equation*}
\t F_k(\t\bx)\!=\! F_k\big(\t\bx_1,\ldots,\t\bx_{i-1},F_i(\t \bx_{1},\ldots,\t\bx_{i-1},\ones,\t\bx_{i+1},\ldots,\t\bx_d),\t\bx_{i+1},\ldots,\t\bx_d\big),\,  k\in[d]\sauf\{i\}.
\end{equation*}
Then, $\t F$ is continuous, order-preserving, \multihomo{} and
\begin{equation}\label{homomatdecouple}
\big(\A(\t F)\big)_{k,l}= A_{k,l}+A_{k,i}A_{i,l}\qquad \forall k,l\in[d]\setminus\{i\}.
\end{equation}
Moreover, given $(\blam,\bx)\in\R^{d}_{++}\times\kone_{+,0}$, we have $F(\bx)=\blam\krog\bx$ if and only if $\t F(\t\bx)=\t\blam\krog\t\bx$ with $\t\bx_k=\bx_k$ and $\t\lambda_k=\lambda_k\lambda_i^{A_{k,i}}$ for every $k\in[d]\sauf\{i\}$. As already observed in Example \ref{lpq_pb}, this transformation can be useful when $F$ does not satisfies some required assumptions, e.g. it does not hold $F(\ones)\in\kone_{++}$.
\end{rmq}
Another important example problem which can be analyzed with the Perron-Frobenius theory for multi-homogeneous mappings is the eigenvalue problem for a class of polynomial maps with nonnegative coefficients. This problem was considered for instance in \cite{Fried} and \cite{Ibrahim}. It is discussed in the following Example \ref{exsum} where it is pointed out that this problem is a special case of the eigenvalue problem of a sum of continuous, order-preserving, multi-homogeneous mappings which is described in Remark \ref{nonomo}. 
\begin{rmq}[Sums of mappings]\label{nonomo} 
Let $F^{(1)},\ldots,F^{(\nu)}\colon\kone_{+}\to\kone_+$ be continuous, multi-homogeneous and order-preserving with $A^{(i)}=\A(F^{(i)})$ for every $i\in[\nu]$. Let $A\in\R^{d\times d}_+$ with $A\geq A^{(i)}$ for all $i\in[\nu]$ and $A\ones\in\R^d_{++}$. Define $N\colon\kone_+\to\R^d_+$ and $G,H\colon\kone_{+}\to\kone_+$ as $N(\bx)=(\norm{\bx_1}_{\gamma_1},\ldots,\norm{\bx_d}_{\gamma_d})$,
\begin{equation*}
H(\bx)=\sum_{k=1}^{\nu}F^{(k)}(\bx) \andd G(\bx)=\sum_{k=1}^{\nu}N(\bx)^{A-A^{(i)}}\krog F^{(k)}(\bx).
\end{equation*}
Then, from Lemma \ref{homoNBd}, we know that if $H(\ones)\in\kone_{++}$, then $G\in\NB^d$ with $\A(G)=A$. Moreover, for every $\bx\in\Sn_{+}$ we have $G(\bx)= H(\bx)$. In particular, we have
\begin{equation*}
G(\bx)=\blam\krog \bx \qquad \iff \qquad H(\bx)=\blam\krog \bx \qquad \forall \bx\in \Sn_+.
\end{equation*}
This shows that the eigenvectors of the non-homogeneous map $H$ on $\Sn_+$ coincide with those of $G\in\NB^d$. We will see that a natural choice for $A$ is to set the entries to be as small as possible. Indeed, Theorem \ref{Banachcor} indicates that one should minimize the spectral radius of $A$.
\end{rmq}
\begin{ex}\label{exsum}
Let $H\colon\R^2\to\R^2$ be given by
\begin{equation*}
H(s,t)=\Bigg(\sum_{k=1}^N c_ks^{\alpha_{k}}t^{\beta_k},\sum_{k=1}^{\t N} \t c_ks^{\t \alpha_{k}}t^{\t \beta_k}\Bigg)
\end{equation*}
then $H$ can be written as the sum
\begin{equation*}
H(s,t)=\sum_{k=1}^NF^{(k)}(s,t)+\sum_{k=1}^{\t N}\t F^{(k)}(s,t) 
\end{equation*}
where 
$F^{(k)}(s,t)=\big(c_ks^{\alpha_{k}}t^{\beta_k},0\big) $ and $\t F^{(k)}(s,t)=\big(0,\t c_ks^{\t \alpha_{k}}t^{\t \beta_k}\big)$. Moreover, if $\alpha_i,\beta_i$, $\t \alpha_j,\t \beta_j>0$ for every $i,j$ and there exists $k_1\in[N],k_2\in[\t N]$ such that $c_{k_1},\t c_{k_2}>0$, then for any $\delta>0$ such that
\begin{equation*}
\textstyle\delta \geq \max_{k\in[N]}\alpha_k\beta_k \quad\and\quad \delta \geq \max_{k\in[\t N]}\t \alpha_k\t \beta_k,
\end{equation*}
it holds $G\in \NB^1$ where
\begin{equation*}
G(s,t)=\sum_{k=1}^N\norm{(s,t)}_{2}^{\delta-\alpha_k\beta_k}F^{(k)}(s,t)+\sum_{k=1}^{\t N}\norm{(s,t)}_{2}^{\delta-\t\alpha_k\t\beta_k}\t F^{(k)}(s,t) ,
\end{equation*}
\end{ex}
\subsection{Weighted Hilbert and Thompson metric and contraction principle}\label{Hilbert_metric}
A natural metric for the study of positive eigenvectors of mappings in $\NB^1$ is the Hilbert semi-metric $\mu\colon\R^{n}_{++}\times \R^{n}_{++}\to \R_{+}$ defined as follows:
\begin{equation*}
\mu(\bz,\bv)= \ln\!\bigg(\frac{\bmax{\bz}{\bv}}{\bmin{\bz}{\bv}}\bigg) \qquad \forall \bz,\bv\in\R^n_{++},
\end{equation*}
where $\bmax{\noarg}{\noarg},\ \bmin{\noarg}{\noarg}\colon\R^{n}_{++}\times \R^{n}_{++}\to \R_{++}$ are defined as
\begin{equation*}
\bmax{\bz}{\bv}= \max_{j\in[n]}\frac{z_j}{v_j} \andd \bmin{\bz}{\bv}= \min_{j\in[n]}\frac{z_j}{v_j}.
\end{equation*}
The Thompson metric $\bar\mu\colon\R^{n}_{++}\times \R^{n}_{++}\to \R_{+}$ is defined as:
\begin{equation*}
\bar\mu(\bz,\bv)= \ln\!\Big(\max\big\{\bmax{\bz}{\bv},\bmax{\bv}{\bz}\big\}\Big) =\norm{\ln(\bv)-\ln(\bz)}_{\infty}\qquad \forall \bz,\bv\in\R^n_{++}.
\end{equation*}
These metrics were introduced in \cite{Bir57, samelson1957perron} and \cite{Thompson} respectively.
It is known that $(\kone_{++},\bar\mu)$ is a complete metric space as well as $(\S_{++}^{\bphi},\mu)$, for any fixed $\bphi\in\R^n_{++}$. Moreover, the topologies of $(\kone_{++},\bar\mu)$ and $(\S_{++}^{\bphib},\mu)$ coincide with the norm topology and, for every $\bx,\by\in\kone_{++}$, if we set $\bz(t)= t\bx+(1-t)\by$, then
\begin{equation}\label{simplegeod}
\mu(\bx,\by)=\mu\big(\bx,\bz(t)\big)+\mu\big(\bz(t),\by\big) \qquad \forall t\in[0,1].
\end{equation}
%-------------------------------->>>>>>>>>>>>>>
%More precisely: 2.1.1 ; p.30 ; 2.6.4 ; 2.5.4 ; 2.5.2 ; 2.1.4 additional interesting facts: 2.2.1 ; 2.2.3; 2.6.9\\
%-------------------------------->>>>>>>>>>>>>>
A key property of these metrics is the following (see Chapter 2 of \cite{NB}): Suppose that $F\colon\R^{n}_{++}\to\R^{n}_{++}$ is continuous, order-preserving and $p$-homogeneous with $p>0$, then
\begin{equation}\label{simplennexp}
\mu(F(\bx),F(\by))\,\leq\, p\,\mu(\bx,\by) \andd \bar\mu(F(\bx),F(\by))\,\leq\, p\,\bar\mu(\bx,\by)
\end{equation}
for every $\bx,\by\in\kone_{++}$.

Based on these observations, we build a metric space appropriate for maps in $\NB^d$ with $d>1$. We define $\Bmax{\noarg}{\noarg},\Bmin{\noarg}{\noarg}\colon\kone_{++}\times\kone_{++}\to\R^d_{++}$ as
\begin{equation*}
\Bmax{\bxb}{\byb}=\big(\bmaxi{1}{\bxb}{\byb},\ldots,\bmaxi{d}{\bxb}{\byb}\big) =\bigg(\max_{j_1\in[n_1]}\frac{x_{1,j_1}}{y_{1,j_1}} ,\ldots,\max_{j_d\in[n_d]}\frac{x_{d,j_d}}{y_{d,j_d}}\bigg)
\end{equation*}
and
\begin{equation*}
\Bmin{\bxb}{\byb}=\big(\bmini{1}{\bxb}{\byb},\ldots,\bmini{d}{\bxb}{\byb}\big) =\bigg(\min_{j_1\in[n_1]}\frac{x_{1,j_1}}{y_{1,j_1}} ,\ldots,\min_{j_d\in[n_d]}\frac{x_{d,j_d}}{y_{d,j_d}}\bigg).
\end{equation*}
for every $\bxb,\byb\in\kone_{++}$. Note that
\begin{equation*}
\Bmin{\bxb}{\byb}\krog\byb\ \lek \ \bxb \ \lek \ \Bmax{\bxb}{\byb}\krog\byb \qquad \forall \bxb,\byb\in\kone_{++}.
\end{equation*}
For $\bb\in\R^d_{++}$, let $\mu_{\bb},\bar\mu_{\bb}\colon\kone_{++}\times\kone_{++}\to\R_{+}$ be respectively the weighted Hilbert and Thompson metrics defined as
\begin{equation*}
\mu_{\bb}(\bxb,\byb)=\sum_{i=1}^db_i\ln\!\bigg(\frac{\bmaxi{i}{\bx}{\by}}{\bmini{i}{\bx}{\by}}\bigg)=\ln\!\bigg(\prod_{i=1}^d\frac{\bmaxi{i}{\bx}{\by}^{b_i}}{\bmini{i}{\bx}{\by}^{b_i}}\bigg)
\end{equation*}
and
\begin{align*}
\bar\mu_{\bb}(\bxb,\byb)&=\sum_{i=1}^db_i\ln\!\Big(\max\big\{\bmaxi{i}{\bx}{\by},\bmaxi{i}{\by}{\bx}\big\}\Big) \end{align*}
Note that for every $\bphi\in\kone_{++}$, $\mu_{\bb}$ is the weighted product metric on $\S_{++}^{\bphi}$ induced by the product of the metric spaces $(\S_{++}^{\bphi_i},b_i\, \mu)$ for $i\in[d]$. Hence, $(\S_{++}^{\bphi},\mu_{\bb})$ is a complete metric space and its topology coincides with the topology of $(\R^N_+,\norm{\noarg}_2)$, where $N=n_1+\ldots+n_d$. Similarly, we know that $(\Sn_{++},\mu_{\bb})$ and $(\kone_{++},\bar\mu_{\bb})$ are complete metric spaces and $(\kone_{++},\bar\mu_{\bb})$ has the same topology as $(\R^N_+,\norm{\noarg}_2)$. The subsequent lemma is a generalization of \eqref{simplennexp}. It motivates the use of $\mu_{\bb}$ and $\bar\mu_{\bb}$ in the study of eigenvectors of $F\in\NB^d$.
\begin{lem}\label{contract}
Let $F\in\NB^d$, $A=\A(F)$ and $\bb\in\R^d_{++}$. Set
\begin{equation*}
C= \max_{i\in[d]}\frac{(A^T\bb)_i}{b_i},
\end{equation*}
then, for every $\bx,\by\in\kone_{++}$ it holds
\begin{equation*}
\mu_{\bb}\big(F(\bx),F(\by)\big)\,\leq\, C\, \mu_{\bb}(\bx,\by) \andd \bar\mu_{\bb}\big(F(\bx),F(\by)\big)\,\leq\, C\, \bar\mu_{\bb}(\bx,\by). 
\end{equation*}
\end{lem}
\begin{proof}
For any $\bxb,\byb\in\kone_{++}$, we have
\begin{equation}\label{homoineq}
\Bmin{\bx}{\by}^A\krog F(\by)\ \lek\ F(\bx)\ \lek\ \Bmax{\bx}{\by}^A\krog F(\by).
\end{equation}
It follows that for every $(j_1,\ldots,j_d)\in\J$ it holds
\begin{equation*}
\prod_{i=1}^d\bmini{i}{\bx}{\by}^{(A^T\bb)_i}\leq \prod_{i=1}^d\bigg(\frac{F_{i,j_i}(\bx)}{F_{i,j_i}(\by)}\bigg)^{b_i}\leq \prod_{i=1}^d\bmaxi{i}{\bx}{\by}^{(A^T\bb)_i}.
\end{equation*}
and
\begin{align*}
\mu_{\bb}\big(F(\bxb),F(\byb)\big)&= \sum_{i=1}^db_i\ln\!\bigg(\frac{\bmaxi{i}{F(\bx)}{F(\byb)}}{\bmini{i}{F(\bxb)}{F(\byb)}}\bigg)\leq \sum_{i=1}^d(A^T\bb)_i\ln\!\bigg(\frac{\bmaxi{i}{\bxb}{\byb}}{\bmini{i}{\bxb}{\byb}}\bigg)\\
&=  \sum_{i=1}^d\frac{(A^T\bb)_i}{b_i}b_i\ln\!\bigg(\frac{\bmaxi{i}{\bxb}{\byb}}{\bmini{i}{\bxb}{\byb}}\bigg)\leq C\,\mu_{\bb}(\bxb,\byb).
\end{align*}
Furthermore, Equation \eqref{homoineq} implies that
\begin{align*}
\bar\mu_{\bb}\big(F(\bxb),F(\byb)\big)
&\leq \ln\!\Bigg(\prod_{i=1}^d\max\!\bigg\{\prod_{k=1}^d\bmaxi{k}{\bx}{\by}^{A_{i,k}},\prod_{k=1}^d\bmaxi{k}{\by}{\bx}^{A_{i,k}}\bigg\}^{b_i}\Bigg)\\
&\leq \ln\!\bigg(\prod_{k=1}^d\max\!\Big\{\bmaxi{k}{\bx}{\by},\bmaxi{k}{\by}{\bx}\Big\}^{(A^T\bb)_k}\bigg)\leq C\,\bar\mu_{\bb}\big(\bxb,\byb\big).\qedhere
\end{align*}
\end{proof}
\begin{rmq}\label{newrmq}
Note that the Collatz-Wielandt principle in Theorem \ref{linear_PF} is useful to get bounds on the possible Lipschitz constant $C$ in Lemma \ref{contract}.
Moreover, for any $A\in\R^{d\times d}_+$, if $\rho(A)<1$, there exists $r\in[\rho(A),1)$ and $\bb\in\R^d_{++}$, such that $A^T\bb\leq r\bb$. Indeed, if $A^T$ has a positive eigenvector $\bc\in\R^d_{++}$, then we can choose $\bb=\bc$ so that $r=\rho(A)$. Otherwise, if $A$ has no positive eigenvector, define $A(t)=A+t(\ones\ones^T)$ for $t\in\R_{++}$. As $A< A(t)$ for any $t> 0$, by continuity, there exists $t_0>0$ such that $0\leq \rho(A)\leq \rho\big(A(t_0)\big)<1$. Theorem \ref{linear_PF} implies the existence of $\bb\in\R^d_{++}$ such that $A(t_0)^T\bb=r\bb$ with $r = \rho\big(A(t_0)\big)$. It follows that $A^T\bb<A(t_0)^T\bb=r\bb$. 
\end{rmq}
We are now ready to prove Theorem \ref{Banachcor} which is a combination of Lemma \ref{contract} and Remark \ref{newrmq}.
\begin{proof}[Proof of Theorem \ref{Banachcor}]
As $\rho(A)<1$ by assumption, Remark \ref{newrmq} implies the existence of $r\in[\rho(A),1)$ and $\bb\in\R^d_{++}$ such that $A^T\bb\leq r\bb$. By Lemma \ref{contract}, we have 
\begin{equation*}
\mu_{\bb}\big(F(\bxb),F(\byb)\big)\leq r \mu_{\bb}(\bxb,\byb) \qquad \forall \bxb,\byb\in\kone_{++}.
\end{equation*}
Now, let $G\colon\Sn_{++}\to\Sn_{++}$ be defined as
\begin{equation*}
G(\bxb)=\bigg(\frac{F_1(\bx)}{\norm{F_1(\bx)}_{\gamma_1}},\ldots,\frac{F_d(\bx)}{\norm{F_d(\bx)}_{\gamma_d}}\bigg).
\end{equation*}
Then we have 
\begin{equation*}
\mu_{\bb}\big(G(\bx),G(\by)\big)=\mu_{\bb}\big(F(\bx),F(\by)\big)\leq r \mu_{\bb}(\bxb,\byb)\qquad \forall \bx,\by\in\kone_{++}.
\end{equation*}
This shows that $G$ is a strict contraction on the complete metric space $(\Sn_{++},\mu_{\bb})$. The result is then a direct consequence of the Banach fixed point theorem (see e.g. Theorem 3.1 in \cite{pointfixe}).
\end{proof}
We give an example of a map which is expansive with respect to the Hilbert and Thompson metrics on $\R^{n_1+\ldots+n_d}_{++}$ but satisfies all the assumptions of Theorem \ref{Banachcor}. This example motivates the study of multi-homogeneous maps and illustrates that several arguments involving homogeneous maps do not hold anymore in the multi-homogeneous framework.
\begin{ex}
Let $\kone_{++}=\R^{2}_{++}\times \R^2_{++}$ and $F\colon \kone_{++}\to\kone_{++}$ be defined as
\begin{equation*}
F\big((a,b),(u,v)\big)=\Big(\big(\min\{u,v\}^{\sqrt{2}},\pi v^{\sqrt{2}}\big),\big((b^{3/2}+a\sqrt{b})^{1/3},\max\{a,b\}^{1/2}\big)\Big).
\end{equation*}
Clearly, $F$ is not homogeneous when $\kone_{++}$ is identified with $\R^{4}_{++}$. It is nevertheless subhomogeneous of degree $\sqrt{2}$, i.e. $\lambda^{\sqrt{2}} F(\bz)\leq F(\lambda \bz)$ for $\lambda \in (0,1)$ and $\bz\in\R^4_{++}$. It turns out that $F$ is expansive with respect to both the Hilbert and the Thompson metric on $\R^4_{++}$. However, we have $F\in \NB^2$, where
\begin{equation*}
\A(F)=A = \begin{pmatrix} 0 & \sqrt 2\\ 2^{-1} & 0 \end{pmatrix} \andd A^T\begin{pmatrix} 2^{-3/4}\\ 1\end{pmatrix} = 2^{-1/4} \begin{pmatrix} 2^{-3/4}\\ 1\end{pmatrix}.
\end{equation*}
Hence, $2^{-1/4}$ is a Lipschitz constant of $F$ with respect to the weighted Hilbert metric 
\begin{equation*} 
\mu_{\bb}(\bx,\by)=2^{-3/4}\mu(\bx_1,\by_1)+\mu(\bx_2,\by_2) \qquad \forall \bx,\by\in\kone_{++},
\end{equation*} 
where $\mu$ is the Hilbert metric on $\R^2_{++}$.  
\end{ex}
Finally, we prove that the transformation discussed in Remark \ref{lpq_rmq} (c) does not help to gain ``contractivity''.
\begin{prop}\label{lemdec}
Let $F\colon\kone_{+}\to\kone_+$ be continuous, order-preserving and multi-homogeneous. Let $A=\A(F)$ and suppose that $A^T\ones>0$ and $A_{i,i}=0$ for some $i\in[d]$. Define $\t F\in\NB^{d-1}$ and $\t A=\A(\t F)$ as in Remark \ref{lpq_rmq} (c). Then, there exists $\t\bb\in\R^{d-1}_{++}$ such that $\t A^T\t\bb\leq\t\bb$ or $\t A^T\t\bb<\t\bb$, if and only if there exists $\bb\in\R^d_{++}$ such that $A^T\bb\leq\bb$ or $A^T\bb<\bb$, respectively.
\end{prop}
\begin{proof}
First of all, note that for every $\bb\in\R^d_{++}$ and $\t \bb = (b_1,\ldots,b_{i-1},b_{i+1},\ldots,b_d)$, $A_{i,i}=0$ and Equation \eqref{homomatdecouple} imply that, for $l\in[d]\sauf\{i\}$, we have
\begin{equation}\label{deco_homoineq}
(A^T\bb)_l =(\t A^T\t \bb)_l+(A^T)_{l,i}\big(b_i-(A^T\bb)_i\big).
\end{equation}
Now, suppose that $A^T\bb\leq r\bb$ for some $r\in(0,1]$. Then, $b_i-(A^T\bb)_i\geq 0$ and
\begin{equation*}
r\t b_l = rb_l \geq (\t A^T\bb)_l+A_{l,i}(b_i-(A^T\bb)_i)\geq (\t A^T\bb)_l, \qquad l\in[d]\sauf\{i\}.
\end{equation*}
Now, let $\t\bb\in\R^{d-1}_{++}$ be such that $\t A^T\t \bb \leq \t \bb$, and, for $t>0$, set 
\begin{equation*}
\bb(t)= (\t b_1,\ldots,\t b_{i-1},t,\t b_{i+1},\ldots,\t b_d)\in\R^d_{++}.
\end{equation*}
We have $\delta_i=\big(A^T\bb(t)\big)_i = \big(A^T\bb(0)\big)_i>0$ and 
\begin{equation*}
\big(A^T \bb(t)\big)_k=\big(\t A^T\t\bb\big)_k+A^T_{k,i}(t-\delta_i)
\end{equation*}
for every $k \in [d]\sauf\{i\}$ and $t>0$. It follows that 
\begin{align}
\max_{k\in[d]}\frac{(A^T\bb(t))_k}{(\bb(t))_k}&=\max\bigg\{\frac{\delta_i}{t}\ ,\ \max_{k\in[d]\sauf\{i\}}\frac{\big(\t A^T\t\bb\big)_k+A^T_{k,i}(t-\delta_i)}{\t b_k}\bigg\}\label{uppercab2}\notag\\ &\leq \max\bigg\{\frac{\delta_i}{t}, \max_{k\in[d]\sauf\{i\}}\frac{(\t A^T\t\bb)_k}{\t b_k}+(t-\delta_i)\max_{k\in [d]\sauf\{i\}}A^T_{k,i}\bigg\}.
\end{align}
Hence, with $t=\delta_i$, we have $\bb(\delta_i)\in\R^{d}_{++}$ and
\begin{equation*}\max_{k\in[d]}\frac{(A^T\bb(\delta_i))_k}{(\bb(\delta_i))_k}\leq \max\bigg\{1,\max_{k\in[d]\sauf\{i\}}\frac{(\t A^T\t\bb)_k}{\t b_k}\bigg\}= 1.
\end{equation*}
Finally, suppose that $\t A^T\t\bb<\t\bb$. There exists $\epsilon$ with
\begin{equation*}
0 < \epsilon < \min_{k\in [d]\sauf\{i\}, A_{i,k}>0}\frac{\t b_k-\big(\t A^T\t\bb\big)_k}{A_{i,k}}.
\end{equation*}
and, from \eqref{uppercab2}, it follows that 
\begin{equation*}
\max_{k\in[d]}\frac{(A^T\bb(\delta_i+\epsilon))_k}{(\bb(\delta_i+\epsilon))_k}= \max\bigg\{\frac{\delta_i}{\delta_i+\epsilon},\max_{k\in[d]\sauf\{i\}}\frac{\big(\t A^T\t\bb\big)_k+\epsilon A^T_{k,i}}{\t b_k}\bigg\}<1.\qedhere
\end{equation*}
\end{proof}
\section{Existence of eigenvectors}\label{existence_section}
It follows from Theorem \ref{Banachcor} that any mapping $F\in\NB^d$ with $\rho(\A(F))<1$ has an eigenvector $\bx\in\kone_{+,0}$ which is strictly positive, i.e. $\bx\in\kone_{++}$. For mappings which are not strict contractions, it is well known that, even when $d=1$ and $F$ is linear, this result is not always true. The aim of this section is to provide sufficient conditions that ensure the existence of a nonnegative, respectively strictly positive eigenvector in the non-expansive case, that is $\rho(\A(F))=1$.

In a first step, we propose and discuss the following notion of (strong) irreducibility for mappings in $\NB^d$:
\begin{defi}\label{newirr}
Let $F\colon\kone_{+}\to\kone_{+}$ be continuous, multi-homogeneous and order-preserving. Then we say $F$ is irreducible, if for any $\bxb\in\kone_{+,0}$, there exists $m\in\N$ such that $\irrF^{m}(\bxb)\subset\kone_{++} $, where $\irrF(\bxb)=\bxb+F(\bxb)$.
\end{defi}
This definition is inspired from the following characterization of irreducible nonnegative matrices
(see for instance \cite{Horn}):
\begin{equation}\label{classicirr}
M\in\R^{n\times n}_+\ \text{is irreducible}\qquad\iff \qquad (I+M)^{n-1}\in\R^{n\times n}_{++}.
\end{equation}
In Proposition \ref{irrcharac}, we prove characterizations of irreducibility for mappings that are order-preserving and \multihomo{}. Some of these characterizations are useful for detecting irreducibility of a map in practice, some others instead are useful for theoretical purposes. Indeed, as in the classical Perron-Frobenius theory, irreducibility is desirable as it implies that any nonnegative eigenvector is positive (see Corollary \ref{irr_pos}). Combining this result with the Brouwer fixed point theorem, we get the following:
\begin{thm}\label{firstprinciple}
If $F\in\NB^d$ is irreducible and $F(\kone_{+,0})\subset\kone_{+,0}$, then there exists $\blam\in\R_{++}^d$ and $\bub\in\Sn_{++}=\big\{(\bx_1,\ldots,\bx_d)\ \big|\ \bx_i\in\R^{n_i}_{++} \ \text{and}\ \norm{\bx_i}_{\gamma_i}=1, \ \forall i\in[d]\big\}$ such that $F(\bub)=\blam\krog\bub$.
\end{thm}

As discussed in Example \ref{rect_ex} and Lemma \ref{NBdIRR}, the assumption $F(\kone_{+,0})\subset\kone_{+,0}$ which is necessary for applying the Brouwer fixed point theorem, is not implied by the irreducibility of $F$ when $d>1$. While the assumptions on $F$ in Theorem \ref{firstprinciple} are quite restrictive in terms of irreducibility, this result has the advantage to hold regardless of the magnitude of $\rho(\A(F))$.

In the second part of this section, we use the so-called continuity of the spectral radius of nonexpansive maps to prove a weak form of the Perron-Frobenius theorem. More precisely, we adapt the notion of Bonsall spectral radius \cite{Bonsall} and cone spectral radius \cite{cone_spec_rad} to maps $F\in\NB^d$ for which there exists $\bb\in\R^d_{++}$ such that $\A(F)^T\bb=\bb$ and prove that such $F$ always has a nonnegative eigenvector corresponding to this notion of spectral radius. To this end, we consider a sequence of mappings $(F^{(\delta_k)})_{k=1}^{\infty}\subset\NB^d$ that converges uniformly towards $F$ as $k\to \infty$ and such that $F^{(\delta_k)}$ satisfies the assumptions of Theorem \ref{firstprinciple} for every $k\in\N$. This is formalized in the following theorem:
\begin{thm}\label{weakPF}
Let $F\in\NB^d$ and $A=\A(F)$. If there exists $\bb\in\Dn$ such that $A^T\bb=\bb$, then there exists $\bu\in\Sn_+$ and $\blam\in\R^d_+$ such that $F(\bu)=\blam\krog\bu$ and 
\begin{equation}\label{specradequal}
\sup_{\bx\in\kone_{+,0}} \limsup_{m\to\infty}\prod_{i=1}^d\norm{F^m_i(\bx)}_{\gamma_i}^{b_i/m}=\prod_{i=1}^d\lambda_i^{b_i}=\lim_{m\to\infty}\sup_{\bx\in\Sn_+}\prod_{i=1}^d\norm{F^m_i(\bx)}_{\gamma_i}^{b_i/m}.
\end{equation}
Moreover, it holds
\begin{equation}\label{poseigequal}
\prod_{i=1}^d\lambda_i^{b_i}=\lim_{m\to\infty}\Big(\prod_{i=1}^d\norm{F^m_i(\bx)}_{\gamma_i}^{b_i}\Big)^{1/m}\qquad \forall \bx\in\kone_{++},
\end{equation}
and, for every $\by\in\Sn_+$ and $\bt\in\R^d_+$ such that $\bt\krog\by\lek F(\by)$, it holds
\begin{equation}\label{maxifirstprinciple}
\prod_{i=1}^d \theta_i^{b_i}\leq \prod_{i=1}^d \lambda_i^{b_i}.
\end{equation}
Finally, if $F$ is irreducible, then $\bu\in\kone_{++}$.
\end{thm}
Theorem \ref{weakPF} is insightful in many aspects that want to we discuss here. First, note that \eqref{specradequal} is the generalization of a known result (see Theorem 2.3 in \cite{cone_spec_rad}), namely that the cone spectral radius (LHS) equals the Bonsall spectral radius (RHS). As a consequence of \eqref{poseigequal}, we have that whenever $F$ satisfies the assumptions of Theorem \ref{weakPF} and $F$ has a positive eigenvector $\bv\in\Sn_{++}$ such that $F(\bv)=\bxi\krog\bv$ for some $\bxi\in\R^d_{++}$, then $\prod_{i=1}\xi_i^{b_i}=\prod_{i=1}^d\lambda_i^{b_i}$. Equation \eqref{maxifirstprinciple} can be a seen as a maximality principle in the following sense: if $\bw\in\Sn_+$ is a nonnegative eigenvector, i.e. there exists $\bbe\in\R^d_+$ such that $F(\bw)=\bbe\krog\bw$, then $\prod_{i=1}^d\beta_i^{b_i}\leq \prod_{i=1}^d\lambda_i^{b_i}$. Finally, combining the existence result (of $\bu$) with the fact every nonnegative eigenvector of an irreducible map is positive, we get a second condition for the existence of a positive eigenvector of $F\in\NB^d$ when $F$ is irreducible which, instead of requiring $F(\kone_{+,0})\subset\kone_{+,0}$ as in Theorem \ref{firstprinciple}, requires the existence of $\bb\in\R^d_{++}$ such that $\A(F)^T\bb = \bb$. 

The last part of this section is devoted to the derivation of a sufficient condition for the existence of a positive eigenvector for mappings that are nonexpansive in the sense of Theorem \ref{weakPF}. To this end, we propose another notion of irreducibility adapted from a graph approach proposed in \cite{Gaubert}. In the case $d=1$, this notion can be seen as a generalization of the fact that a nonnegative matrix is irreducible if and only if the associated adjacency graph is strongly connected. For the definition of this graph, we consider, for all $(i,j_i)\in \I$, the map $\bu^{(i,j_i)}\colon\R_{+}\to\kone_{+,0}$ defined as
\begin{equation*} %\label{defu}
\big(\bub^{(i,j_i)}(t)\big)_{k,l_k}=\begin{cases} t & \text{if } (k,l_k)=(i,j_i)\\ 1 & \text{otherwise},\end{cases} \qquad \forall (k,l_k)\in\I \ =\bigcup_{\nu=1}^d(\{\nu\}\times [n_{\nu}]).
\end{equation*}
Then, the (asymptotic) graph associated with $F\in\NB^d$ is given by the following:
\begin{defi}\label{graphdefi}
For $F\in\NB^d$, the directed graph $\G(F)=(\I,\E)$, is defined as follows: There is an edge from $(k,l_k)$ to $(i,j_i)$, i.e. $\big((k,l_k),(i,j_i)\big)\in\E$, if
\begin{equation*}
\lim_{t\to\infty} F_{k,l_k}\big(\bu^{(i,j_i)}(t)\big) =\infty.
\end{equation*}
\end{defi}
Our following existence result is a nontrivial generalization of Theorem 2 in \cite{Gaubert}. In the case $d=1$, our proof reduces to that of (the equivalent) Theorem 6.2.3 in \cite{NB} which assumes that the graph of $F$ is strongly connected. However, as discussed below, when $d>1$, our assumption is less restrictive than requiring $\G(F)$ to be strongly connected. Moreover, as noted in Corollary 6.2.4 \cite{NB} for the case $d=1$, a dual version of this result can be easily obtained by considering a graph that analyses the behaviour of $F$ around $0$:
\begin{defi}
For $F\in\NB^d$, the directed graph $\G^{-}(F)=(\I,\E^{-})$, is defined as follows: There is an edge from $(k,l_k)$ to $(i,j_i)$, i.e. $\big((k,l_k),(i,j_i)\big)\in\E^{-}$, if
\begin{equation*}
\lim_{t\to 0} F_{k,l_k}\big(\bu^{(i,j_i)}(t)\big) =0.
\end{equation*}
\end{defi}
We have the following theorem:
\begin{thm}\label{existence}
Let $F\in\NB^d$ be such that there exists $\bb\in\R^d_{++}$ with $\A(F)^T\bb=\bb$. Let $\G\in\{\G(F),\G^{-}(F)\}$. If for every $(\nu,l_{\nu})\in\I$ and $(j_1,\ldots,j_d)\in[n_1]\times\ldots\times [n_d]$ there exists $i_{\nu}\in[d]$ such that there is a path from $(i_\nu,j_{i_{\nu}})$ to $(\nu,l_{\nu})$ in $\G$, then $F$ has an eigenvector in $\Sn_{++}$.
\end{thm}
While the assumption in the above theorem is equivalent to requiring that $\G$ is strongly connected when $d=1$, this is not the case anymore when $d>1$ as shown by Example \ref{nonirrgraph}. Indeed, if $\G$ is strongly connected, then $\G$ satisfies the assumption of Theorem \ref{existence} but the converse is not true anymore for $d>1$. Moreover, Example \ref{noncomparableirr} shows that the connectivity of $\G(F)$ and $\G^{-}(F)$, and the irreducibility of $F$ can not be compared in the sense that there are mappings that satisfy exactly one of these assumptions and none of the other two.
\subsection{Irreducible maps and Brouwer fixed point theorem}
We start by a small observation: For $\alpha>0$, let $F^{(\alpha)}\in\NB^1$ with $F^{(\alpha)}_j(\bx)=(M\bx)_j^{\alpha}$ for every $j\in[n_1]$, where $M\in\R^{n_1\times n_1}_+$ and $M\ones\in\R^d_{++}$. Then, note that $F^{(1)}$ is irreducible in the sense of Definition \ref{newirr} if, and only if, $M$ is irreducible in the sense of Equation \eqref{classicirr}. Moreover, we observe that, with respect to Definition \ref{newirr}, $F^{(1)}$ is irreducible if and only if $F^{(\alpha)}$ is irreducible for any $\alpha>0$. A similar observation holds for the $\ell^{p,q}$-singular value problem of a nonnegative matrix (see Example \ref{lpq_pb}) where the irreducibility of $F$ defined as in \eqref{lpq_F} does not depend on $p,q\in(1,\infty)$. This can be even extended to $\ell^{p_1,\ldots,p_d}$-singular values problems of nonnegative tensors, where the irreducibility of the induced mapping $F\in\NB^d$ does, as a matter of fact, not depend on the choice of $p_1,\ldots,p_d\in(1,\infty)$. We formulate this observation for order-preserving, \multihomo{} mappings and prove several characterizations and properties of irreducibility. First, we need the following: %As in the classical Perron-Frobenius, the notion of irreducibility is desirable because, as shown later on, every nonnegative eigenvector of an irreducible map is positive.
\begin{lem}\label{irrlem}
Let $F,G\colon\kone_{+}\to\kone_{+}$ be continuous, multi-homogeneous and order-preserving. Set $H(\bx)=\bx+F(\bx)$ and $E(\bx)=\bx+G(\bx)$. Then:
\begin{enumerate}[\quad\, (a)]
\item For every $k\in\N$, $\bx\in\kone_{+,0}$ and $\bal\in\R^d_{++}$, there exists $\bbe,\bdel\in\R^d_{++}$ such that\label{iterHbound}
\begin{equation*}
\bdel\krog H^{k}(\bx)\lek H^k(\bal\krog\bx)\lek \bbe\krog H^{k}(\bx).
\end{equation*}
\item If for every $\bx\in\kone_{+,0}$ there exists $\bal\in \R^d_{++}$ with $\bal\krog F(\bx)\lek G(\bx)$, then for every $k\in\N$ and every $\bz\in\kone_{+,0}$, there exists $\bdel\in\R^d_{++}$ such that $\bdel\leq \ones$ and $\bdel\krog H^k(\bz) \lek E^k(\bz)$.\label{compareFG}
\end{enumerate}
\end{lem}
\newcommand{\btp}{\theta^+}
\newcommand{\btm}{\theta^-}
\begin{proof}
Let $A=\A(F)$.\\ 
\eqref{iterHbound} Define $\btm,\btp\colon\R^d_{++}\times\R^d_{++}\to\R^d_{++}$ as
$\btm_i(\ba,\bb)=\min\{a_i,b_i\}$ and $\btp_i(\ba,\bb)=\max\{a_i,b_i\}$ for every $i\in[d]$. Then, we have 
\begin{equation*}
\bdel^{(k)}\krog H^k(\bx)\lek H^k(\bal\krog\bx)\lek \bbe^{(k)}\krog H^k(\bx)
\end{equation*}
with $\bdel^{(1)}=\btm(\bal,\bal^{A})$, $\bbe^{(1)}=\btp(\bal,\bal^{A})$, 
\begin{equation*}
\bdel^{(k+1)}=\btm\Big(\bdel^{(k)},(\bdel^{(k)})^A\Big) \quad\text{and}\quad \bbe^{(k+1)}=\btp\Big(\bbe^{(k)},(\bbe^{(k)})^A\Big)
\end{equation*}
for every $k\in \N$. As $\bal\in\R^d_{++}$, we have $\bdel^{(k)},\bbe^{(k)}\in\R^d_{++}$ which proves the claim.\\
\eqref{compareFG} Let $\bz\in\kone_{+,0}$ and $\bal^{(1)}\in (0,1]^d$ be such that $\bal^{(1)}\krog F(\bz)\lek G(\bz)$. Then,
\begin{equation*}
\bal^{(1)}\krog H(\bz)\lek \bx+\bal^{(1)}\krog F(\bz) \lek E(\bz).
\end{equation*}
There exists $0<\bbe^{(1)}\leq \bal^{(1)}$ such that, for $\bu = \bal^{(1)}\krog H(\bz)$, we have
\begin{equation*}
(\bbe^{(1)}\circ(\bal^{(1)})^A)\krog F\big(H(\bz)\big)=\bbe^{(1)}\krog F(\bu) 
\lek G(\bu) \lek G\big(E(\bz)\big).
\end{equation*}
Let $\bal^{(2)}=\bbe^{(1)}\circ(\bal^{(1)})^A$, then $0<\bal^{(2)}\leq \bal^{(1)}$ and
\begin{equation*}
\bal^{(2)}\krog H^2(\bz) \lek \bal^{(1)}\krog H(\bz)+\bal^{(2)}\krog F\big(H(\bz)\big)\lek E^2(\bz).
\end{equation*}
This argument can be repeated for $k\geq 2$ by letting $\bu=\bal^{(k)}\krog H^k(\bz)$ and $\bal^{(k+1)}=\bbe^{(k)}\circ(\bal^{(k)})^A$ showing the existence of a sequence $(\bal^{(k)})_{k=1}^{\infty}\subset\R^d_{++}$ with the desired property.\end{proof}
The following proposition generalizes some well-known characterizations of irreducibility for matrices (see \cite{Plemmons}) to order-preserving, \multihomo{} mappings. Characterization \eqref{characirrE} can be found in the setting of rectangular tensors in Theorem 5.1 \cite{Yang2}. Characterization \eqref{characirrQ} is inspired by a similar result in \cite{Fried}. We recall that 
$\I = \bigcup_{k=1}^d \big(\{k\}\times [n_k]\big)$ and $\J= [n_1]\times\ldots\times [n_d]$.
\newcommand{\EE}{\mathbb{E}}
\begin{prop}\label{irrcharac}
Let $F\colon\kone_{+}\to\kone_+$ be continuous, \multihomo{} and order-preserving. Let $H(\bx)=\bx+F(\bx)$. The following statements are equivalent:
\begin{enumerate}
\item $F$ is irreducible.\label{FIRR}
\item There exists $m\in\N$ with $H^{m}(\{\be^{(\bj)}\mid\bj\in\J\})\subset\kone_{++}$, where, for $\bj\in\J$ and $(k,l_k)\in\I$, $\big(\be^{(\bj)}\big)_{k,l_k}=1$ if $l_k=j_k$ and $\big(\be^{(\bj)}\big)_{k,l_k}=0$ else.\label{characirrE}
\item The map $\bx\mapsto \bga_{\bp}\big(F(\bx)\big)$ is irreducible, where $\bga_{\bp}\colon\kone_{+,0}\to\kone_{+,0}$ is defined for $\bp\in\R^d_{++}$ as $\big(\bga_{\bp}(\bx)\big)_{k,l_k}=x_{k,l_k}^{p_k}$ for every $(k,l_k)\in\I$.\label{characirrG}
\item $Q\big(F(\bz)\big)\not\supset Q(\bz)$ for every $\bz\in\kone_{+,0}\sauf\kone_{++}$, where, for every $\bx\in\kone_{+,0}$,  $Q(\bx)=\{\bj\in\J\ | \ x_{i,j_i}=0, \ \forall i \in[d]\}$.\label{characirrQ}
\end{enumerate}
\end{prop}
\begin{proof}
\eqref{FIRR}$\iff$\eqref{characirrE} 
Since $H(\kone_{++})\subset\kone_{++}$,  $H^\nu(\bx)\in\kone_{++}$ implies 
$H^k(\bx)\in\kone_{++}$ for every $k\geq \nu$. It follows directly that \eqref{FIRR} implies \eqref{characirrE}. Now, assume that $F$ is not irreducible and let $\bx\in\kone_{+,0}$ be such that $H^k(\bx)\in\kone_{+,0}\sauf\kone_{++}$ for every $k$. There exists $\bal\in\R^d_{++}$ and $\bj\in\J$ such that $\bal\krog\be^{(\bj)}\lek \bx$. By Lemma \ref{irrlem}, \eqref{iterHbound} we know that for every $k\in\N$, there exists $\bdel^{(k)}\in\R^d_{++}$ with 
\begin{equation*}
\bdel^{(k)}\krog H^k(\be^{(\bj)})\lek H^{k}(\bal\krog\be^{(\bj)})\lek H^{k}(\bx).
\end{equation*}
Since $H^k(\bx)\in\kone_{+,0}\sauf\kone_{++}$, we have $H^k(\be^{(\bj)})\in\kone_{+,0}\sauf\kone_{++}$ for every $k$ and thus there exists no $m\in\N$ such that $H^m(\{\be^{(\bj)}\mid\bj\in\J\})\subset\kone_{++}$.\\
\eqref{FIRR}$\iff$\eqref{characirrG}  Let $G(\bx)=\bga_{\bp}\big(F(\bx)\big)$. Then $G$ is order-preserving and \multihomo{} (with $\A(G)=\diag(\bp)A$). Let $\bx\in\kone_{+,0}$ and $(k,l_k)\in\I$, we have $F_{k,l_k}(\bx) >0$ if and only if $F_{k,l_k}(\bx)^{p_k} >0$. It follows that
\begin{equation*}
\bal\krog F(\bx) \lek \bga_{\bp}\big(F(\bx)\big)\lek \bbe\krog F(\bx) 
\end{equation*}
with $\alpha_i=\beta_i=1$ if $\I_{\bx} =\{j_i\in[n_i]\mid F_{i,j_i}(\bx)>0\}=\emptyset$, and 
\begin{equation*}
\alpha_i = \min_{j_i\in \I_{\bx}} F_{i,j_i}(\bx)^{p_i-1},\qquad\beta_i = \max_{j_i\in\I_{\bx}}F_{i,j_i}(\bx)^{p_i-1}\qquad \text{if }\ \I_{\bx}\neq\emptyset.
\end{equation*}
The claim follows now from Lemma \ref{irrlem}, \eqref{compareFG}.\\
\eqref{FIRR}$\iff$\eqref{characirrQ} 
Note that for $\bx,\by\in\kone_{+,0}$, if $Q(\bx)=Q(\by)$, then there exists $\bal,\bbe\in\R^d_{++}$ such that $\bal\krog\by \lek \bx \lek \bbe\krog\by$ which implies that $Q\big(F(\bx)\big)=Q\big(F(\by)\big)$. Suppose that there exists $\bx\in\kone_{+,0}$ with $\emptyset \neq Q(\bx)\subset Q(F(\bx))$. Then we have
\begin{equation*}
Q\big(H(\bx)\big)=Q(\bx)\cap Q\big(F(\bx)\big) =Q(\bx).
\end{equation*}
Using induction, if $Q\big(H^{k}(\bx)\big)=Q(\bx)$ for some $k\in\N$, then 
\begin{equation*}
Q\big(H^{k+1}(\bx)\big)=Q\big(H^k(\bx)\big)\cap Q\big(F(H^k(\bx))\big) =Q(\bx)\cap Q\big(F(\bx)\big)=Q(\bx).
\end{equation*}
Since this is true for any $k\in\N$, we have $Q\big(H^{k}(\bx)\big)\neq \emptyset$ for every $k$ and thus, $F$ can not be irreducible.
Now, suppose that $Q(\bz)\not \subset Q\big(F(\bz)\big)$ for every $\bz\in\kone_{+,0}\sauf\kone_{++}$. We show that there exists $m\in\N$ such that $H^m(\kone_{+,0})\subset \kone_{++}$. Let $\bx \in\kone_{+,0}$, if $Q\big(H(\bx)\big)= Q(\bx)\neq \emptyset$, then $Q\big(F(\bx)\big)\supset Q(\bx)$, a contradiction to our assumption. Hence, $Q\big(H(\bx)\big)\subsetneq Q(\bx)$. It follows that $Q\big(H^{k+1}(\bx)\big)\subsetneq Q\big(H^{k}(\bx)\big)$ for every $k\in\N$ with $Q\big(H^{k}(\bx)\big)\neq \emptyset$. In particular, $Q\big(H^{|\I|}(\bx)\big)= \emptyset$, i.e. $H^{|\I|}(\bx)\in\kone_{++}$.
\end{proof}
Note that, if $F$ is irreducible, then the index $m$ in Proposition \ref{irrcharac}, \eqref{characirrE} can be chosen to be $m=\max_{\bj\in\J}m_{\bj}$ where $m_{\bj}\in\N$ is the smallest integer such that $H^{m_{\bj}}(\be^{(\bj)})\in\kone_{++}$. Moreover, if $H^m(\{\be^{(\bj)}\mid\bj\in\J\})\subset\kone_{++}$, then $H^m(\kone_{+,0})\subset\kone_{++}$. Indeed, if $\bx\in\kone_{+,0}$, then there exist $\bdel\in\R^d_{++}$ and $\bj\in\J$ such that $\bdel\krog\be^{(\bj)}\lek\bx$ and by Lemma \ref{irrlem}, \eqref{iterHbound} we know that there exists $\bal\in\R^d_{++}$ such that 
\begin{equation*}
0\lekkk\bal\krog H^m(\be^{(\bj)})\lek H^m(\bdel\krog\be^{(\bj)})\lek H^m(\bx).
\end{equation*}
\begin{cor}\label{irr_pos}
Let $F\colon\kone_{+}\to\kone_+$ be continuous, \multihomo{} and order-preserving. If $F$ is irreducible, then for every $(\blam,\bxb)\in\R^d_{++}\times\kone_{+,0}$ satisfying $F(\bxb)=\blam\krog\bxb$, we have $(\blam,\bxb)\in\R^d_{++}\times\kone_{++}$.
\end{cor}
\begin{proof}
Let $Q$ be as in Proposition \ref{irrcharac}, \eqref{characirrQ}. If $\bxb\in\kone_{+,0}\sauf\kone_{++}$, then we have the contradiction $Q\big(F(\bx)\big) = Q(\blam\krog\bx)\supset Q(\bx)\neq \emptyset$. 
\end{proof}
A first existence result for irreducible maps $F$ is stated in the next theorem and comes as a consequence of the following lemma which can be proved by applying the Brouwer fixed point theorem to the map 
\begin{equation*}
F'(\bz) =\big(\ps{\bphi_1}{F(\bz_1)}^{-1},\ldots,\ps{\bphi_d}{F(\bz_d)}^{-1}\big)\krog F(\bz).
\end{equation*}
\begin{lem}\label{Brouwer_exis}
Let $F\colon\kone_{+}\to\kone_+$ be continuous, \multihomo{} and order-preserving. If $F(\kone_{+,0})\subset\kone_{+,0}$, then, for every $\bphi\in\kone_{++}$, $F$ has an eigenvector $\bxb\in\S^{\bphi}_{+}$.
\end{lem}
We are now ready to prove Theorem \ref{firstprinciple}.
\begin{proof}[Proof of Theorem \ref{firstprinciple}]
Let $\bphi\in\kone_{++}$. Since $F(\kone_{+,0})\subset\kone_{+,0}$, Lemma \ref{Brouwer_exis} implies that $F$ has an eigenvector $\t\bub\in\S_{+}^{\bphib}$. Let $\bub=(\norm{\bub_1}_{\gamma_1}^{-1},\ldots,\norm{\bub_d}_{\gamma_d}^{-1})\krog\t\bub$ and $\blam \in\R^d_{+}$ be such that $F(\bub)=\blam\krog\bub$. As $F(\bub)\in\kone_{+,0}$, we have $\blam\in\R^d_{++}$ and Corollary \ref{irr_pos} implies that $(\blam,\bub)\in\R_{++}^d\times\Sn_{++}$ since $F$ is irreducible.
\end{proof}
Note that $G,H\colon\R^n_{+}\to\R^n_+$ with $G(\bzb)=\big(\ps{\bz}{\ones},0,\ldots,0\big)$ and $H(\bz)= z_{1}\ones$, satisfy $G(\R^n_{+,0})\subset\R^n_{+,0}$, $G(\R^n_{++})\not\subset\R^n_{++}$, $H(\R^n_{++})\subset\R^n_{++}$ and $H(\R^n_{+,0})\not\subset\R^n_{+,0}$.
Thus, it is not true in general that if $F(\kone_{+,0})\subset\kone_{+,0}$ then $F\in\NB^d$ and vice-versa. Moreover, the following example shows that, when $d>1$, if $F\in\NB^d$ is irreducible, then it is not necessarily true that $F(\kone_{+,0})\subset \kone_{+,0}$.
\begin{ex}\label{rect_ex}
Let $d=2$, $n_1=n_2=2$ and $F\in\NB^d$ with
\begin{equation*}
F\bigg(\begin{pmatrix}a\\ b \end{pmatrix},\begin{pmatrix}s\\ t \end{pmatrix}\bigg)=\bigg(\begin{pmatrix} at+(a+b)s \\ at\end{pmatrix},\begin{pmatrix} as+(a+b)t\\ as\end{pmatrix}\bigg)
\end{equation*}
Then, with $\irrF(\bx)=\bx+F(\bx)$, we have $\irrF^2(\kone_{+,0})\subset \kone_{++}$. However, note that $F\big((0,1),(1,0)\big)=\big((1,0),(0,0)\big)\notin\kone_{+,0}$.
\end{ex}
Nevertheless, one can prove the following.
\begin{lem}\label{NBdIRR}
Let $F\colon\kone_{+}\to\kone_+$ be continuous, \multihomo{} and order-preserving.
\begin{enumerate}
\item If $F$ is irreducible, then $F(\kone_{++})\subset \kone_{++}$.\label{nondeg}
\item If $F$ is irreducible and $d=1$, then $F(\kone_{+,0})\subset \kone_{+,0}$.\label{littlelema}
\end{enumerate}
\end{lem}
\begin{proof}
For every $\bx\in\kone_{+}$, let $H(\bx)=\bx+F(\bx)$.\\
\eqref{nondeg} Suppose that there exists $\bx\in\kone_{++}$ and $(i,l_i)\in \I$ with $F_{i,l_i}(\bx)=0$. Let $\t\bx\in\kone_{+,0}$ be defined as $\t x_{k,j_k} = x_{k,j_k}$ for every $(k,j_k)\in\I\sauf\{(i,l_i)\}$ and $\t x_{i,l_i} =0$. Then, $\t\bx\lek \bx$ and so $F_{i,l_i}(\t\bx)=0$. It follows from Proposition \ref{irrcharac}, \eqref{characirrQ} (with $\bz=\t\bx$) that $F$ can not be irreducible.\\
\eqref{littlelema} Suppose by contradiction that $F(\bx)=0$ for some vector $\bx\in\kone_{+,0}=\R^{n_1}_{+,0}$.
Then, $H(\bx)=\bx+F(\bx)=\bx$ and therefore $H^{k}(\bx)=\bx$ for every $k\in \N$. Since $F$ is irreducible, there exists $m$ such that $H^{m}(\bx)\in\kone_{++}$. It follows that $\bx\in\kone_{++}$, but then $F(\kone_{++})\not\subset\kone_{++}$ which is a contradiction to \eqref{nondeg}.
\end{proof}
\subsection{Weak Perron-Frobenius theorem}\label{specrad_sec}
\newcommand{\abs}[2]{\langle #1 \rangle_{#2} ZZ}
As discussed at the beginning of this section, we consider the notions of Bonsall spectral radius and cone spectral radius for mappings $F\in\NB^d$ such that there exists $\bb>0$ with $\A(F)^T\bb=\bb$. These notions will be particularly useful to show the existence of a nonnegative eigenvector when one can not apply directly the Brouwer fixed point theorem, that is when $F(\kone_{+,0})\not\subset \kone_{+,0}$.

Let $\bx\in\kone_+$, $F\in\NB^d$, $A=\A(F)$, $\bb\in\Dn=\{\bc\in\R^d_{++}\mid \sum_{i=1}^d c_i = 1\}$ and assume that $A^T\bb =\bb$. Define
\begin{equation*}
\nnorm{\bxb}_{\bb}= \prod_{i=1}^d\norm{\bx_i}_{\gamma_i}^{b_i}\andd \nnorm{F}_{\bb}= \sup_{\bxb \in \Sn_+}\nnorm{F(\bxb)}_{\bb}.
\end{equation*}
Then, for every $\bal\in\R^d_{++}$ and $\bx \in\kone_{+,0}$, it holds
\begin{equation*}
\nnorm{F(\bal\krog \bx)}_{\bb} = \nnorm{\bal^A\krog F(\bx)}_{\bb} =\nnorm{F(\bx)}_{\bb}\prod_{i=1}^d\alpha_i^{(A^T\bb)_i}=\nnorm{F(\bx)}_{\bb}\prod_{i=1}^d\alpha_i^{b_i}
\end{equation*}
Hence, with $\bbe= (\norm{\bx_1}^{-1}_{\gamma_1},\ldots,\norm{\bx_d}^{-1}_{\gamma_d})$, we have
\begin{equation}\label{opnormineq}
\nnorm{F(\bxb)}_{\bb}=\nnorm{F(\bbe\krog\bxb)}_{\bb}\nnorm{\bxb}_{\bb}\leq\nnorm{F}_{\bb}\nnorm{\bxb}_{\bb} \qquad \forall \bxb\in\kone_{+,0}.
\end{equation}
Now, consider
\begin{equation*}
r_{\bb}(F)= \sup_{\bxb\in\kone_{+,0}}\limsup_{m\to \infty}\nnorm{F^{m}(\bxb)}_{\bb}^{1/m}
\andd
\hat{r}_{\bb}(F)= \lim_{m\to\infty}\nnorm{F^{m}}_{\bb}^{1/m}.
\end{equation*} 
In the case $d=1$, $\hat r_{\bb}$ is known as Bonsall spectral radius \cite{Bonsall} and $r_{\bb}$ is known as cone spectral radius \cite{cone_spec_rad}. Note that for every $\lambda >0$, we have
\begin{equation*}
r_{\bb}(\lambda\, F)= \lambda\, r_{\bb}(F)\andd \hat r_{\bb}(\lambda\, F)= \lambda \,\hat r_{\bb}(F).
\end{equation*}
Moreover, if $M\in\R^{n\times n}_+$ and $F(\bx)=M\bx$, then the Gelfand formula \cite{Gelfand} implies that $\rho(M)=r_{1}(F)$. The proof of Theorem 5.31 \cite{NB}, a special case of Theorem 2.2 \cite{cone_spec_rad}, can be easily adapted to obtain the following:
\begin{thm}\label{Thm531} 
Let $F\in\NB^d$, $A=\A(F)$ and $\bb\in\Dn$ with $A^T\bb=\bb$, then
\begin{equation*}
0\ \leq \ r_{\bb}(F)\ =\ \hat{r}_{\bb}(F)\ <\ \infty.
\end{equation*}
\end{thm}
In the following proposition we extend the second part of Theorem 2.2 \cite{cone_spec_rad} to the multi-homogeneous case. In particular, this result tells us that the eigenvalue associated with a positive eigenvector corresponds to the notions of spectral radius presented in this section. Moreover, we use this proposition for the proof of the Collatz-Wielandt formula in Section \ref{CW_M_U1}.
\begin{prop}\label{Prop536}
Let $F\in\NB^d$, $A=\A(F)$ and $\bb\in\Dn$ with $A^T\bb=\bb$. Then,
\begin{equation*}
r_{\bb}(F)=\lim_{m\to\infty}\nnorm{F^{m}(\bxb)}_{\bb}^{1/m}\qquad \forall \bxb\in\kone_{++}.
\end{equation*}
Moreover, for every $\by\in \Sn_{+}$ and $\bt\in\R^d_{++}$ with $\bt\krog \by \lek F(\by)$, we have 
\begin{equation}\label{minicw}
\prod_{i=1}^d\theta_i^{b_i}\leq r_{\bb}(F).
\end{equation}
\end{prop}
\begin{proof}
Let $\bxb\in\kone_{++}$, there exists $\bs\in\R^d_{++}$ such that for every $\byb \in \Sn_+$, it holds $\byb \lek \bs\krog\bxb$. For $k \in \N$ and $\byb \in  \Sn_+$ we have 
\begin{equation*}
\nnorm{F^{k}(\byb)}_{\bb}\leq \nnorm{F^{k}(\bs\krog\bxb)}_{\bb}= \nnorm{\bs^{A^k}\krog F^{k}(\bxb)}_{\bb}=\nnorm{F^{k}(\bxb)}_{\bb}\prod_{i=1}^ds_i^{b_i}.
\end{equation*}
Taking the supremum over $\by\in\Sn_+$, we get 
\begin{equation*}
\nnorm{F^{k}}_{\bb}\bigg(\prod_{i=1}^d s_i^{-b_i}\bigg)\leq \nnorm{F^{k}(\bxb)}_{\bb} \leq \nnorm{F^{k}}_{\bb}\nnorm{\bxb}_{\bb} \qquad \forall k\in\N.
\end{equation*}
Theorem \ref{Thm531} implies $\lim_{k\to \infty}\nnorm{F^{k}}^{1/k}_{\bb}=r_{\bb}(F)$ and thus,
\begin{equation*}
r_{\bb}(F)=\lim_{k\to\infty}\nnorm{F^k(\bxb)}^{1/k}_{\bb}.
\end{equation*}
Now, if $\bt\krog\by\lek F(\by)$, then for all $k\in\N$
\begin{equation*}
\nnorm{F^{k}(\by)}_{\bb}\geq\nnorm{\by}_{\bb} \prod_{i=1}^d\theta_i^{(\sum_{j=0}^{k-1} (A^j)^T\bb)_i} =\nnorm{\by}_{\bb}\prod_{i=1}^d\theta_i^{kb_i} .
\end{equation*}
It follows that
\begin{equation*}
\prod_{i=1}^d\theta_i^{b_i} = \lim_{k\to \infty}\Big(\nnorm{\by}_{\bb}\prod_{i=1}^d\theta_i^{kb_i}\Big)^{1/k} \leq \limsup_{k\to \infty}\nnorm{F^{k}(\by)}_{\bb}^{1/k} \leq r_{\bb}(F).\qedhere
\end{equation*}
\end{proof}
The last tool we need to prove our weak Perron-Frobenius Theorem \ref{weakPF}, is the next result which is a generalization of Theorem 5.4.1 \cite{NB}. We prove the existence of an eigenvector corresponding to the spectral radius for a class of mappings in $\NB^d$. Although being of own interest, this theorem will also be helpful in Section \ref{wirr_thm} for the proof of the existence of a positive eigenvector in the case when $F\in\NB^d$ is not irreducible. Furthermore, we will use it in Section \ref{CW_M_U} to show that the Collatz-Wielandt characterization of the spectral radius holds without the assumption that there exists a positive eigenvector.
\begin{thm}\label{BIGTHM}
Let $F\in\NB^d$, $A=\A(F)$ and $\bb\in\Dn$ with $A^T\bb=\bb$. For each $\delta>0$, define $F^{(\delta)}\colon\kone_{+}\to \kone_{+}$ as
\begin{equation*}
F^{(\delta)}(\bx)= F(\bx)+\delta\big(\norm{\bx_1}_{\gamma_1},\ldots,\norm{\bx_d}_{\gamma_d}\big)^A\krog\ones,
\end{equation*}
where $\ones$ is the vector of all ones. Then, the following statements hold:
\begin{enumerate}
\item For every $\delta>0$, we have $F^{(\delta)}\in\NB^d$ and $\A(F^{(\delta)})=A$. Moreover, there exists $(\blam^{(\delta)},\bxb^{(\delta)})\in\R^d_{++}\times\Sn_{++}$ such that
\label{abra2} \begin{equation*}\label{eqqqqq}
F^{(\delta)}(\bxb^{(\delta)})=\blam^{(\delta)}\krog\bxb^{(\delta)}\andd \prod_{i=1}^d (\lambda_i^{(\delta)})^{b_i}=r_{\bb}(F^{(\delta)}).
\end{equation*}
\item If $0<\eta<\epsilon$, then $r_{\bb}(F^{(\eta)})<r_{\bb}(F^{(\epsilon)})$ and hence, there exists $r\geq 0$ such that $\lim_{\delta\to 0}r_{\bb}(F^{(\delta)}) = r$.\label{abra3}
\item There exists $(F^{(\delta_k)})_{k=1}^{\infty}\subset \{F^{(\delta)}\}_{\delta>0}$ such that $\lim_{k\to\infty}F^{(\delta_k)}= F$ and the corresponding sequence $(\blam^{(\delta_k)},\bx^{(\delta_k)})_{k=1}^{\infty}$ obtained from \eqref{abra2}, converges to a maximal eigenpair of $F$. That is, there exists $(\blam,\bx)\in\R^d_{+}\times \S^d_+$ such that\label{abra4} $\lim_{k\to\infty}(\blam^{(\delta_k)},\bxb^{(\delta_k)}) =(\blam,\bx)$, 
\begin{equation*} 
F(\bx)=\blam \krog \bx\andd r_{\bb}(F)=\prod_{i=1}^d \lambda_i^{b_i}=r.
\end{equation*}
\end{enumerate}
\end{thm}
\begin{proof}
Let $\delta>0$, then $F^{(\delta)}\in\NB^d$ and $A=\A(F^{(\delta)})$ follow from Lemma \ref{homoNBd}. Note that $F^{(\delta)}(\kone_{+,0})\subset \kone_{++}$ and thus $F^{(\delta)}$ is irreducible. Theorem \ref{firstprinciple} implies the existence of $(\blam^{(\delta)},\bx^{(\delta)})\in\R^d_{++}\times \Sn_{++}$ such that $F^{(\delta)}(\bxb^{(\delta)})=\blam^{(\delta)}\krog\bxb^{(\delta)}$. By Proposition \ref{Prop536}, we know $\prod_{i=1}^d (\lambda_i^{(\delta)})^{b_i}=r_{\bb}(F^{(\delta)})$. \\
We prove \eqref{abra3}:
Let $0<\eta<\epsilon$. As $F^{(\eta)}(\bxb^{(\eta)})=\blam^{(\eta)}\krog\bxb^{(\eta)}$, we have 
\begin{equation*}
F^{(\epsilon)}(\bxb^{(\eta)})=F(\bx^{(\eta)})+\epsilon\ones=F^{(\eta)}(\bxb^{(\eta)})+(\epsilon-\eta)\ones=\blam^{(\eta)}\krog\bxb^{(\eta)}+(\epsilon-\eta)\ones.
\end{equation*}
There exist $\zeta>0$ and $\xi>0$ such that 
\begin{equation*}
\zeta\bxb^{(\eta)}\leq (\epsilon-\eta)\ones\andd r_{\bb}(F^{(\eta)})+\xi<\prod_{i=1}^d(\lambda_i^{(\eta)}+\zeta)^{b_i}.
 \end{equation*}
We have $(\blam^{(\eta)}+\zeta\ones)\krog\bxb^{(\eta)} \lek F^{(\epsilon)}(\bxb^{(\eta)})$ and thus, with Proposition \ref{Prop536}, we get
\begin{equation*}
r_{\bb}(F^{(\eta)})+\xi<\prod_{i=1}^d(\lambda_i^{(\eta)}+\zeta)^{b_i}\leq r_{\bb}(F^{(\epsilon)}).
\end{equation*}
It follows that $r_{\bb}(F^{(\eta)})< r_{\bb}(F^{(\epsilon)})$ for every $0<\eta<\epsilon$ and $\lim_{\delta\to 0}r_{\bb}(F^{(\delta)})=r$ exists. 
Finally, we prove \eqref{abra4}. There exists $C>0$ such that $\by\leq C\ones$ for every $\by\in\Sn_+$. It follows that for every  $0<\epsilon \leq 1$, it holds 
\begin{equation*}
0 \lek\blam^{(\epsilon)}\krog \bx^{(\epsilon)} = F^{(\epsilon)}(\bx^{(\epsilon)})\lek F^{(1)}(\bx^{(\epsilon)}) \lek F^{(1)}(C\ones),
\end{equation*}
and thus $\{\blam^{(\epsilon)}\mid 0<\epsilon\leq 1\}$ is bounded in $\R^d_+$. Clearly, the set $\{\bxb^{(\epsilon)}\mid 0<\epsilon\leq 1\}\subset \Sn_+$ is bounded in $\kone_+$. Hence, there exists $(\epsilon_k)_{k=1}^{\infty}\subset\R_{++}$ with $\epsilon_k\to 0$, $\bx^{(\epsilon_k)}\to \bx$ and $\blam^{(\epsilon_k)}\to\blam$ as $k\to\infty$. Note that 
\begin{equation*}
r=\lim_{k\to\infty}\prod_{i=1}^d(\lambda_i^{(\epsilon_k)})^{b_i}=\prod_{i=1}^d\lambda_i^{b_i}.
\end{equation*}
Since $F^{(\epsilon_k)}(\bxb^{(\epsilon_k)})=\blam^{(\epsilon_k)}\krog\bxb^{(\epsilon_k)}$, we have
\begin{equation*}
F(\bxb^{(\epsilon_k)})=F^{(\epsilon_k)}(\bxb^{(\epsilon_k)})-\epsilon_k\ones=\blam^{(\epsilon_k)}\krog\bxb^{(\epsilon_k)}-\epsilon_k\ones.
\end{equation*}
So, by continuity of $F$, we get
\begin{equation*}
F(\bxb)=\lim_{k\to\infty}F(\bxb^{(\epsilon_k)})=\lim_{k\to\infty}\blam^{(\epsilon_k)}\krog\bxb^{(\epsilon_k)}-\epsilon_k\ones=\blam\krog\bxb.
\end{equation*}
On the one hand, by definition, we have 
\begin{equation*}
r_{\bb}(F)\geq\limsup_{m\to\infty}\nnorm{F^{m}(\bxb)}_{\bb}^{1/m}=\limsup_{m\to\infty}\Big(\nnorm{\blam^{\sum_{j=0}^{m-1}A^{j}}\krog\bxb}_{\bb}\Big)^{1/m}=\prod_{i=1}^d\lambda_i^{b_i}.
\end{equation*}
On the other hand, with Proposition \ref{Prop536}, we have $\blam\krog \bxb = F(\bxb)\lek F^{(\epsilon_k)}(\bxb)$, and thus
\begin{equation*}
r_{\bb}(F)\leq r_{\bb}(F^{(\epsilon_k)})=\prod_{i=1}^d(\lambda_i^{(\epsilon_k)})^{b_i}\qquad\forall k\in\N.
\end{equation*}
Letting $k\to\infty$, we finally get $r_{\bb}(F)\leq\prod_{i=1}^d\lambda_i^{b_i}$.
\end{proof}
The proof of Theorem \ref{weakPF} is now a collection of the results above.
\begin{proof}[Proof of Theorem \ref{weakPF}]
From Theorem \ref{BIGTHM}, we know that there exists $(\blam,\bub)\in\R^d_{+}\times \Sn_{+}$ such that $F(\bub)=\blam\krog\bub$ and $\prod_{i=1}^d\lambda_i^{b_i}=r_{\bb}(F)$. Equation \eqref{specradequal} follows from Theorem \ref{Thm531}, \eqref{poseigequal} and \eqref{maxifirstprinciple} follow from Proposition \ref{Prop536} and if $F$ is irreducible, then Corollary \ref{irr_pos} implies that $(\blam,\bub)\in\R^d_{++}\times \Sn_{++}$.
\end{proof}
\subsection{Further notions of irreducibility}\label{wirr_thm}
There are mappings $F\in\NB^d$ that have a positive eigenvector although they are not irreducible in the sense of Definition \ref{newirr}. We propose an additional notion of irreducibility adapted from \cite{Gaubert} where a graph is associated to positive, homogeneous order-preserving maps. When $d=1$, the graph of Definition \ref{graphdefi} coincides with the one proposed in \cite{Gaubert} and our existence Theorem \ref{exist} coincides with Theorem 2 \cite{Gaubert} where the graph is required to be strongly connected. The following is an example of a graph associated to a map $F\in\NB^2$ such that $\G(F)$ is not strongly connected but satisfies the connectivity assumption of Theorem \ref{exist}.
\begin{ex}\label{nonirrgraph}
Let $d=2$, $n_1=n_2=2$ and $F\in \NB^d$ with
\begin{equation*}
F((a,b),(s,t))=\Big(\big(a,b\big),\big(\max\{a,t\}^{1/2},\max\{b,s\}^{1/2}\big)\Big) %\begin{pmatrix}\max\{s,t\}\\ \max\{s,t\} \end{pmatrix}\Big),
\end{equation*}
Then, $F(\ones,\ones)=(\ones,\ones)$ and $\G(F)$ is given by
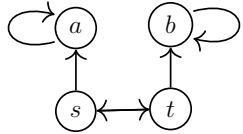
\begin{equation*}
\begin{tikzpicture}[baseline= (a).base]
\node[scale=.8] (a) at (0,0){
\begin{tikzcd}[cells={nodes={circle,draw,font=\sffamily\Large\bfseries}},thick]
 \arrow[loop left, distance=3em, start anchor={[yshift=-1.5ex]west}, end anchor={[yshift=1.5ex]west},thick]{}{} a  %\ar[thick]{ddr}
 &\arrow[loop right, distance=3em, start anchor={[yshift=1.5ex]east}, end anchor={[yshift=-1.5ex]east}, thick]{}{}b\\
 s \ar[thick]{r} \ar[thick]{u}%\ar[thick]{ddr}
  & t \arrow[thick]{u}\ar[thick]{l} 
  %  \\
    %& \arrow[loop right, distance=3em, start anchor={[yshift=1.5ex]east}, end anchor={[yshift=-1.5ex]east},thick]{}{}3\ar[thick]{uur}{}
\end{tikzcd}
};
\end{tikzpicture}
\end{equation*}
\end{ex}
The proof of the next theorem relies on the following construction which is a generalization of the technique proposed in Section 3.2 of \cite{Gaubert}:
Let $F\in\NB^d$, $\bb\in\R^d_{++}$, $\G(F)=(\I,\E)$ and, for $r>0$, define
\begin{equation*}
\Psi(r)=\sup\bigg\{ t\geq 0 \ \bigg| \min_{\substack{((i,j_i),(k,l_k))\in\E\\ \ba\in\J}} F_{i,j_i}\big(\bu^{(k,l_k)}(t)\big)^{b_i}\prod_{\substack{s=1\\ s\neq i}}^dF_{s,a_s}\big(\bu^{(k,l_k)}(t)\big)^{b_s}\leq r\bigg\}.
\end{equation*}
Note that, by definition of $\G(F)$, $\Psi(r)<\infty$ for any $r>0$ and $\Psi$ is an increasing function. Moreover, note that $\Psi$ has the following property: Let $(j_1,\ldots,j_d)\in\J$, $i\in[d]$, $(k,l_k)\in\I$ and $t>0$, if $\big((i,j_i),(k,l_k)\big)\in\E$ then
\begin{equation}\label{psimagic}
\prod_{s=1}^d F_{s,j_s}\big(\bu^{(k,l_k)}(t)\big)^{b_s}\leq r \qquad \implies\qquad t\leq \Psi(r).
\end{equation}
In the case $d=1$, the proof of Theorem 6.2.3 \cite{NB} relies on the following idea: if $F\in\NB^1$ is homogeneous, $\G(F)$ is strongly connected and its maximal nonnegative eigenvector $\bx\in\R^{n_1}\saufzero$ (given by Theorem \ref{Prop536}) has a zero entry, then one gets the contradiction $\bx=0$. Following the same idea, for a non-expansive map $F\in\NB^d$ with $d\geq 1$, we give a condition on $\G(F)$ so that, if its maximal eigenvector $\bx\in\kone_{+,0}$ has a zero entry, then $\bx_i=0$ for some $i\in[d]$ contradicting $\bx\in\kone_{+,0}$.
\begin{thm}\label{exist}
Let $F\in\NB^d$, $A=\A(F)$ and $\bb\in\Dn$ with $A^T\bb=\bb$. Suppose that for every $(\nu,l_\nu)\in\I$ and $(j_1,\ldots,j_d)\in\J$, there exists $i_\nu\in[d]$ such that there is a path from $(i_\nu,j_{i_\nu})$ to $(\nu,l_\nu)$ in $\G(F)$. Then $F$ has an eigenvector in $\Sn_{++}$.
\end{thm}
\newcommand{\bi}{\mathbf{i}}
\begin{proof}
By Theorem \ref{BIGTHM}, there exists $\big((\blam^{(\epsilon_k)},\bxb^{(\epsilon_k)})\big)_{k=1}^{\infty}\subset \R^d_{++}\times \Sn_{++}$ such that 
\begin{equation*}
\lim_{k\to\infty}(\blam^{(\epsilon_k)},\bxb^{(\epsilon_k)}) =(\blam,\bxb^*) \andd F(\bxb^*)=\blam \krog \bx^*\in\Sn_{+}.
\end{equation*} 
Since $\blam^{(\epsilon_k)}\to\blam$, there exists a constant $M_0>0$ such that 
\begin{equation}\label{defM0}
\prod_{s=1}^d(\lambda_s^{(\epsilon_k)})^{b_s}\leq M_0 \qquad \forall k \in \N.
\end{equation}
Suppose by contradiction that $\bxb^*\in\Sn_{+}\setminus\Sn_{++}$. By taking a subsequence if necessary, we may assume that there exists $\bj=(j_1,\ldots,j_d)\in\J$ and $\omega\in[d]$ such that 
\begin{equation*}
\min_{t_s\in[n_s]}x_{s,t_s}^{(\epsilon_k)}=x^{(\epsilon_k)}_{s,j_s} \qquad \forall s\in[d], \ k \in\N \andd \lim_{k\to\infty} x^{(\epsilon_k)}_{\omega,j_\omega}=x_{\omega,j_\omega}^*=0.
\end{equation*}
In particular, as $\bx^{(\epsilon_k)}\in\Sn_{+}$, there exists $\t C>0$ such that  $x^{(\epsilon_k)}_{s,t_s}\leq\t C$ for every $(s,t_s)\in \I$ and $k\in\N$. Thus,
\begin{equation}\label{contra1}
0\leq \lim_{k\to \infty} \prod_{s=1}^d(x^{(\epsilon_k)}_{s,j_s})^{b_s} \leq \t C^{1-b_{\omega}}\lim_{k\to\infty} (x^{(\epsilon_k)}_{\omega,j_\omega})^{b_\omega}=0.
\end{equation}
Since $\bx^*\in\Sn_{+}$, there exists $\bl=(l_1,\ldots,l_d)\in\J$ with $x_{s,l_s}^*>0$ for all $s\in[d]$. Thus,
\begin{equation}\label{contra2}
\lim_{k\to \infty} \prod_{s=1}^d(x^{(\epsilon_k)}_{s,l_s})^{b_s} =\prod_{s=1}^d(x^*_{s,l_s})^{b_s}>0.
\end{equation}
Let $\nu\in[d]$, by assumption on $\G(F)$, there exists $i_{\nu}\in[d]$ and a path $(i_\nu,j_{i_\nu})=(m_1,\xi_{m_1})\to (m_2,\xi_{m_2})\to\ldots\to (m_{N_{\nu}},\xi_{m_{N_{\nu}}})=(\nu,l_\nu)$ in $\G(F)$ with $N_{\nu}\leq n_1+\ldots+n_d$. Define $\bi{(1)}, \bi{(2)},\ldots,\bi{(N_{\nu})}\in\J$ as
\begin{equation*}
\bi_s{(a)}=\begin{cases} \xi_{m_a}&\text{if } s= m_a, \\ j_s &\text{otherwise.}\end{cases} \qquad \forall s \in [d], \ a\in[N_{\nu}].
\end{equation*}
Fix $k\in\N$ and let $t=x_{m_2,\xi_{m_2}}^{(\epsilon_k)}/x_{m_2,j_2}^{(\epsilon_k)}$ and $\bal= \big((x^{(\epsilon_k)}_{1,j_1})^{-1},\ldots,(x^{(\epsilon_k)}_{d,j_d})^{-1}\big)$. We have $\bu^{(m_2,\xi_{m_2})}(t)\lek \bal\krog\bx^{(\epsilon_k)}$ and
\begin{align*}
\prod_{s=1}^dF_{s,\bi_s(1)}\big(\bu^{(m_2,\xi_{m_2})}(t)\big)^{b_s}&\leq \prod_{s=1}^dF_{s,\bi_s{(1)}}(\bal\krog\bx^{(\epsilon_k)})^{b_s}\\
&=\bigg(\prod_{s=1}^d (x^{(\epsilon_k)}_{s,j_s})^{b_s}\bigg)^{-1}\prod_{s=1}^dF^{(\epsilon_k)}_{s,\bi_s{(1)}}(\bx^{(\epsilon_k)})^{b_s}\leq M_0.
\end{align*}
where $M_0>0$ satisfies \eqref{defM0}. Hence, by \eqref{psimagic}, $t=x_{m_2,\xi_{m_2}}^{(\epsilon_k)}/x_{m_2,j_2}^{(\epsilon_k)} \leq \Psi(M_0)$ and
\begin{equation*}
\prod_{s=1}^d (x^{(\epsilon_k)}_{s,\bi_s(2)})^{b_s} \leq  M_1\prod_{s=1}^d (x_{s,j_s}^{(\epsilon_k)})^{b_s}\qquad \text{with} \qquad M_1 =\Psi(M_0)^{b_{m_2}}.
\end{equation*}
Applying this procedure again to $ (m_3,\xi_{m_3})$, we get the existence of a constant $M_2>0$ independent of $k$, such that 
\begin{equation*}
\prod_{s=1}^d (x^{(\epsilon_k)}_{s,\bi_s(3)})^{b_s} \leq  M_2\prod_{s=1}^d (x^{(\epsilon_k)}_{s,j_s})^{b_s}.
\end{equation*}
Indeed, let $t=x^{(\epsilon_k)}_{m_3,\xi_{m_3}}/x_{m_3,j_3}^{(\epsilon_k)}$, then $\bu^{(m_3,\xi_{m_3})}(t)\lek \bal\krog\bx^{(\epsilon_k)}$ and 
\begin{align*}
\prod_{s=1}^dF_{s,\bi_s(2)}\big(\bu^{{(m_3,\xi_{m_3})}}(t)\big)^{b_s}&\leq \prod_{s=1}^dF_{s,\bi_s{(2)}}(\bal\krog\bx^{(\epsilon_k)})^{b_s}\\
%&=
& = \bigg(\prod_{s=1}^d(x^{(\epsilon_k)}_{s,j_s})^{b_s}\bigg)^{-1}\prod_{s=1}^dF^{(\epsilon_k)}_{s,\bi_s{(2)}}(\bx^{(\epsilon_k)})^{b_s}\leq M_0M_1.
\end{align*}
Hence, with $M_2=\Psi(M_0M_1)^{b_{m_3}}$, we get the desired inequality. Repeating this process at most $N_{\nu}$ times, we obtain a constant $C_{\nu}>0$ independent of $k$, such that 
\begin{equation}\label{onthewayeqnnnn}
(x^{(\epsilon_k)}_{\nu,l_{\nu}})^{b_\nu}\prod_{\substack{s=1\\ s\neq \nu}}^d (x^{(\epsilon_k)}_{s,j_s})^{b_s}=\prod_{s=1}^d (x^{(\epsilon_k)}_{s,\bi_s(N_{\nu})})^{b_s} \leq  C_{\nu}\prod_{s=1}^d (x^{(\epsilon_k)}_{s,j_s})^{b_s} \qquad \forall k\in\N.
\end{equation}
Take the product over $\nu\in[d]$ in \eqref{onthewayeqnnnn} and divide by $\prod_{s=1}^d (x^{(\epsilon_k)}_{s,j_s})^{(d-1)b_s}$ to obtain
\begin{equation*}
\prod_{\nu=1}^d(x^{(\epsilon_k)}_{\nu,l_{\nu}})^{b_\nu}\leq C\prod_{s=1}^d (x^{(\epsilon_k)}_{s,j_s})^{b_s} \qquad\forall k\in\N,
\end{equation*}
where $C=\prod_{\nu=1}^dC_{\nu}.$ Finally, using \eqref{contra1} and \eqref{contra2} we get a contradiction.
\end{proof}
As noted in Corollary 6.2.4 \cite{NB}, there exists a dual version of Theorem \ref{exist} that follows by considering the map $\tau \colon \R^{n}_{++}\to \R_{++}^n$ defined as $\tau(\bz)=(z_1^{-1},\ldots,z_n^{-1})$. 
More precisely, let $F\in\NB^d$ and define $\hat F\colon\kone_{++}\to\kone_{++}$ as $\hat F(\bx)=\tau\big(F(\tau(\bx))\big)$ for every $\bx\in\kone_{++}$. 
Then, $\tau$ is a bijection between the positive eigenvectors of $F$ and $\hat F$. Moreover, by Theorem \ref{extend}, $\hat F$ can be continuously  extended on $\kone_{+}$ so that $\hat{F}\in\NB^d$. By noting that $\G(\hat{F})=\G^{-}(F)$, the following corollary follows from Theorem \ref{exist} applied to $\hat F$.
\begin{cor}\label{existdual}
Let $F\in\NB^d$, $A=\A(F)$ and $\bb\in\Dn$ with $A^T\bb=\bb$. Suppose that, for every $(\nu,l_\nu)\in\I$ and $(j_1,\ldots,j_d)\in\J$ there exists $i_\nu\in[d]$ such that there is a path from $(i_\nu,j_{i_\nu})$ to $(\nu,l_\nu)$ in $\G^{-}(F)$. Then $F$ has an eigenvector in $\Sn_{++}$.
\end{cor}
Note that $\G(F)$ and $\G^{-}(F)$ can be very different. In fact, consider again the map $F\in\NB^1$ of Example \ref{nonirrgraph}. Then, $\G^{-}(F)$ contains only two self-loops. On the other hand, if we substitute the $\max$'s with $\min$'s in the definition of $F$, we obtain a map $F'\in\NB^1$ with $\G^{-}(F')=\G(F)$ and $\G(F')=\G^{-}(F)$. 
Moreover, as shown by the next example, there is no obvious relationship between irreducibility and the connectivity of $\G(F)$ and $\G^{-}(F)$.
\begin{ex}\label{noncomparableirr}
Let $d=1$, $n_1=3$ and let $F,G\in\NB^d$ be defined as
\begin{equation*}
F(a,b,c)= \big( \min\{b,c\}+a, \max\{a,b\}+c,b+c\big) \quad\text{and}\quad G(a,b,c)=(abc)^{1/3}\ones.
\end{equation*}
Then, $F$ is irreducible because $F^2(\kone_{+,0})\subset\kone_{++}$, but $\G(F)$ and $\Gm(F)$ are not strongly connected. On the other hand, $G$ is not irreducible while $\G(G)$ and $\Gm(G)$ are fully connected.
\end{ex}
Theorem \ref{existence} is now a direct consequence of Theorem \ref{exist} and Corollary \ref{existdual}.
\section{Collatz-Wielandt principle, maximality and uniqueness of positive eigenvectors}\label{CW_M_U}
Recall that for a matrix $M\in\R^{n\times n}_+$, the Collatz-Wielandt principle reads:
\begin{equation}\label{cw_discuss}
\inf_{\bz\in\R^n_{++}}\,\max_{j\in[n]}\frac{(M\bz)_j}{z_j} = \rho(M) = \max_{\bz\in\R^n_+\saufzero}\,\min_{\substack{j\in[n]\\ z_j>0}}\frac{(M\bz)_j}{z_j}.
\end{equation}
A generalization of this characterization is already known for maps in $\NB^1$, see Theorem 5.6.1 \cite{NB} or Theorem 1 \cite{GV12} for an even more general result. We start this section by proving a similar principle for mappings in $\NB^d$. To this end, for $\bb\in\R^d_{++}$, we consider the functions $\cwl\colon\NB^d\times \kone_{+,0}\to\R_{+}$ and $\cwu\colon\NB^d\times \kone_{++}\to\R_{++}$ defined as
\begin{equation*}
\cwu(F,\bu)=\prod_{i=1}^d \Big(\max_{j_i\in[n_i]}\frac{F_{i,j_i}(\bu)}{u_{i,j_i}}\Big)^{b_i} \quad \text{and}\quad \cwl(F,\bx)=\prod_{i=1}^d \Big(\min_{\substack{j_i\in[n_i]\\ x_{i,j_i}>0}}\frac{F_{i,j_i}(\bx)}{x_{i,j_i}}\Big)^{b_i}.
\end{equation*}
In particular, if $d=1$ and $F(\bx)=M\bx$ is positive, i.e. $F(\ones)>0$, then the characterization \eqref{cw_discuss} of $\rho(M)$ can be rewritten as follows:
\begin{equation*} %\label{cw_discuss}
\inf_{\bz\in\R^n_{++}}\widecheck{\operatorname{cw}}_{1}(F,\bz)= r_1(F) =
 \max_{\bz\in\R^n_+\saufzero}\widehat{\operatorname{cw}}_{1}(F,\bz),
\end{equation*}
where $r_\bb$, defined in Section \ref{specrad_sec}, satisfies
\begin{equation*}
r_{\bb}(F) = \lim_{k\to \infty} \prod_{i=1}^d \norm{F^k_i(\bx)}_{\gamma_i}^{b_i/k}\qquad \forall \bx\in\kone_{++}
\end{equation*}
for any $F\in\NB^d$ such that $\A(F)^T\bb=\bb$. While, in the particular case $F(\bx)=M\bx$ and more generally when $F\in\NB^1$ and $\A(F)=1$, the scaling of eigenvectors does not influence the associated eigenvalue, we know that for mappings in $\NB^d$ with $d>1$ this property may not hold anymore. Thus, we formulate our Collatz-Wielandt principle on the unit sphere instead of the whole cone. However, as discussed below, using the multi-homogeneity, these characterizations can be easily extended on the whole cone.
Our first Collatz-Wielandt principle is formulated in the following:
\begin{thm}\label{CW1}
Let $F\in\NB^d$, $A=\A(F)$ and $\bb\in\Dn$. If $A^T\bb = \bb$, then
\begin{equation}\label{CWeq1}
\inf_{\bx\in\Sn_{++}} \cwu(F,\bx) = r_{\bb}(F) = \max_{\bx\in\Sn_{+}}\cwl(F,\bx).
\end{equation}
\end{thm}
While Theorem \ref{Banachcor} proves the existence and the uniqueness of a positive eigenvector for mappings $F\in\NB^d$ such that $\rho\big(\A(F)\big)<1$, proving a Collatz-Wielandt principle for such mappings requires additional effort. This is why we also need the following theorem which generalizes Theorem 21 in \cite{us}, where the Collatz-Wielandt principle is proved in the context of the $\ell^{p_1,\ldots,p_d}$-singular values problem of nonnegative tensors. 
\begin{thm}\label{CW<1}
Let $F\in\NB^d$ and $A=\A(F)$. If $\rho(A)<1$, then there exist $\blam\in\R^d_{++}$ and $\bu\in\Sn_{++}$ such that $F(\bu)=\blam\krog\bu$. Moreover, there exists $\bb\in\R^d_{++}$ such that $A^T\bb\leq\bb$ and for every such $\bb$, it holds
\begin{equation}\label{CWeq<1}
\min_{\bx\in\Sn_{++}} \cwu(F,\bx) =\prod_{i=1}^d\lambda_i^{b_i} = \max_{\bx\in\Sn_{+}}\cwl(F,\bx).
\end{equation}
%Furthermore, $\by \in\Sn_{++}$ satisfies one of the following assertions:
%\begin{equation}\label{eqCW}
%\cwu(F,\by)=\prod_{i=1}^d\lambda_i^{b_i}\qquad \text{ or }\qquad \ \cwl(F,\by)=\prod_{i=1}^d\lambda_i^{b_i}, 
%\end{equation}
%if and only if $\by = \bu$.
\end{thm}
%We stress that the proof of Theorem \ref{CWeq<1} relies on the monotonicity of the norms $\norm{\cdot}_{\gamma_i}$, $i=1,\ldots,d$. 
The restriction in \eqref{CWeq1} and \eqref{CWeq<1} that $\bx$ has to belong to the product of unit spheres can be overcome by noticing that
\begin{equation*}
\cwl(F,\bz)=\Big(\prod_{i=1}^d\norm{\bz}_{\gamma_i}^{(\A(F)^T\bb)_i-b_i}\Big)\cwl\big(F,(\norm{\bz}^{-1}_{\gamma_1},\ldots,\norm{\bz}_{\gamma_d}^{-1})\krog \bz\big) \qquad \forall \bz\in\kone_{+,0}.
\end{equation*}
Let $F\in\NB^d$ and suppose that there exists $\bb\in\Dn$ such that $\A(F)^T\bb\leq \bb$. For simplicity in the following discussion, we define $R$ to be either $r_{\bb}(F)$ if $\rho\big(\A(F)\big)=1$ or $\prod_{i=1}^d\lambda_i^{b_i}$ if $\rho\big(\A(F)\big)<1$, where $\blam\in\R^d_{++}$ is the eigenvalue associated to the unique positive eigenvector of $F$ (see Theorem \ref{Banachcor}). On top of being helpful to obtain bounds on $R$, the Collatz-Wielandt principles in \eqref{CWeq1} and \eqref{CWeq<1} imply the maximality of $R$. Indeed, if $(\bt,\bx)\in\R^d_{+}\times \kone_{+,0}$ satisfies $F(\bx)=\bt\krog\bx$, then 
\begin{equation*}
\prod_{i=1}^d\theta_i^{b_i}=\cwl(F,\bx)\leq R.
\end{equation*} 
In particular, this shows that the eigenvalue $\t\blam\in\R^d_{++}$ associated to a positive eigenvector $\t\bu\in\Sn_{++}$ of $F$ is always maximal, because 
\begin{equation*}
\prod_{i=1}^d\t \lambda_i^{b_i} =\cwl(F,\t\bu)\leq R \leq \cwu(F,\t\bu)=\prod_{i=1}^d\t \lambda_i^{b_i}.
\end{equation*}
Suppose that the map $F$ is the restriction of a map $\t F$ defined on the whole vector space $V=\R^{n_1}\times \ldots\times \R^{n_d}$. For example, let $M\in\R^{n\times n}_+$ be such that $M\ones \in\R^n_{++}$ and define $\t F\colon\R^n\to\R^n$ as $\t F(\bx)=M\bx$. Then $\t F|_{\R^n_{+}}\in\NB^1$ and $\t F$ has a nonnegative eigenvector $\bx\in\R^n_{+}\saufzero$ such that $\t F(\bx)=\lambda\bx$ and $\lambda =\rho(M)$. In particular, if $(\theta,\by)\in\R\times(\R^n\saufzero)$ is any eigenpair of $M$, then $|\theta | \leq \lambda$. A similar observation for the $\ell^{p_1,\ldots,p_d}$-singular value problem of a nonnegative tensor has been proved in Corollary 4.3 \cite{Fried}. In Corollary \ref{extension_max}, we adapt this technique for a wider class of mappings in $\NB^d$. Another remarkable observation for linear maps that can be generalized to maps in $\NB^d$ is the following: Suppose that $M$ is irreducible and $M\by = \theta\by$ for some $\theta\geq 0 $ and $\by\in\R^n_{+,0}\sauf\R^n_{++}$, then $\theta<\rho(M)$. This fact is known for mappings in $\NB^1$ (see for example Theorem 6.1.7 in \cite{NB}) and we extend it to the case $d\geq 1$:
\begin{thm} \label{rad<}
Let $F\in\NB^d$ and $A=\A(F)$. Suppose that there exists $\bb\in\Dn$, $\blam\in\R^d_{++}$ and $\bu\in\Sn_{++}$ such that $A^T\bb=\rho(A)\bb$ and $F(\bu)=\blam\krog\bu$. If $\rho(A)\leq 1$, $F$ is differentiable at $\bub$ and there exist $i\in[d]$ and $\tau\in\N$ such that 
\begin{equation}\label{dirr}
\forall \bw\in\kone_{+}\saufzero, \quad \text{if} \quad \bx = \sum_{k=1}^\tau DF(\bu)^k\bw, \quad \text{then} \quad \bx_i\in\R^{n_i}_{++}.
\end{equation}
Then, for every $(\bt,\bvb)\in\R^d_{+}\times(\Sn_{+}\sauf\kone_{++})$ with $F(\bvb)=\bt\krog\bvb$, we have
\begin{equation*}
\prod_{j=1}^d\theta_j^{b_j}<\prod_{j=1}\lambda_j^{b_j}.
\end{equation*}
\end{thm}
We note that, when $d=1$ the assumption \eqref{dirr} on $DF(\bu)$ is equivalent to $DF(\bu)$ being irreducible. However, as noted in Example \ref{irrrrrr_ex}, the equivalence does not hold anymore for $d>1$, namely if $DF(\bu)$ is irreducible, then \eqref{dirr} is satisfied but the converse might not be true. Nevertheless, we prove in Proposition \ref{hardcorelema} that if $DF(\bu)$ satisfies \eqref{dirr} and $\A(F)$ is irreducible, then $DF(\bu)$ is irreducible.

The second part of this section is concerned with deriving a sufficient condition for the uniqueness of eigenvectors in $\Sn_{++}$. We know already from Theorem \ref{Banachcor} that when $F\in\NB^d$ is such that $\rho\big(\A(F)\big)<1$, then $F$ has a unique eigenvector in $\Sn_{++}$. This is, however, not true anymore when $\rho\big(\A(F)\big) = 1$, as illustrated by Example \ref{nonuniquevect} where an irreducible map $F\in\NB^1$ with a continuum of eigenvectors is given. Such sufficient condition is known for mappings $F\in\NB^1$ with $\rho(\A(F))=1$ and it requires that the Jacobian $DF(\bu)$ of $F$ at an eigenvector $\bu\in \Sn_{++}$ satisfies $\dim\!\big(\ker(\xi I-DF(\bu))\big)=1$ where $\xi=\rho\big(DF(\bu)\big)$ (see for instance Theorem 2.5 \cite{Nussb} or Theorem 6.4.6 \cite{NB}). 
The same condition can be derived for maps in $\NB^d$.
\begin{thm}\label{unique}
Let $\bphi\in\kone_{++}$, $F\in\NB^d$ and $A=\A(F)$. Suppose that $A$ is irreducible, $\rho(A)=1$, there exist $\blam \in\R^d_{++}$ and $\bu\in\S_{++}^{\bphi}$ such that $F(\bu)=\blam \krog \bu$ and $F$ is differentiable at $\bu$. Let $L\in\R^{|\I|\times |\I|}_+$ be defined as
\begin{equation}\label{crazyass}
L_{(i,j_i),(k,l_k)} = \frac{1}{\lambda_i} \frac{\partial  F_{i,j_i}(\bu)}{\partial x_{k,l_k}}\qquad \forall (i,j_i),(k,l_k)\in\I.
\end{equation}
Then, it holds $\rho(L)=1$. Moreover, if $\dim(\ker(I-L))=1$, then $\bu$ is the unique eigenvector of $F$ in $\S^{\bphi}_{++}$.
\end{thm}
If $DF(\bu)$ in Theorem \ref{unique} is irreducible, then $L$ is also irreducible because $L$ has the same zero pattern as $DF(\bu)$. Hence, the irreducibility of $DF(\bu)$ implies that $\dim(\ker(I-L))=1$. It is, however, not always true that every nonnegative matrix $M$ such that $\dim\!\big(\!\ker(\rho(M)I-M)\big)=1$ is irreducible as shown by Example \ref{Krein_notirr}. Moreover, we note that when $d=1$, the assumption \eqref{crazyass} is equivalent to that discussed in the previous paragraph because in this case $L$ is just a rescaling of $DF(\bu)$. As a final observation, we note that if $F\in\NB^d$ satisfies the assumptions of Theorem \ref{rad<} and $\A(F)$ is irreducible, then, by Proposition \ref{hardcorelema} we know that $DF(\bu)$ is irreducible and thus $F$ also satisfies the assumptions of Theorem \ref{unique}. 
\subsection{Collatz-Wielandt principle and maximality of positive eigenvectors}\label{CW_M_U1}
Before starting the proofs of Theorem \ref{CW1} and \ref{CW<1}, we note that for $F\in\NB^d$, $\bb\in\R^d_{++}$ and $\bx\in\kone_{++}$, it holds $\cwl(F,\bx)=\cwu(F,\bx)$ if and only if $\bx$ is an eigenvector of $F$. Indeed, the following identity holds
\begin{equation*}
\mu_{\bb}(F(\bx),\bx)= \ln\bigg(\frac{\cwu(F,\bx)}{\cwl(F,\bx)}\bigg) \qquad \forall \bx\in\kone_{++}.
\end{equation*}
The proof of Theorem \ref{CW1} is similar to that of Theorem 5.6.1 \cite{NB} for the case $d=1$.
\begin{proof}[Proof of Theorem \ref{CW1}]
First we show that $r_{\bb}(F)=\inf_{\bx\in\Sn_{++}}\cwu(F,\bx)$. % and then we show that $r_{\bb}(F)=\max_{\bx\in\Sn_{++}}\cwl(F,\bx)$. 
So, let $\bx\in\Sn_{++}$, then, for every $k\in\N$, we have 
\begin{equation*}
 F^{k}(\bx)\ \lek\ \Bmax{F(\bx)}{\bx}^{\sum_{j=0}^{k-1} A^j}\krog\bx\end{equation*}
It follows from Proposition \ref{Prop536} that
\begin{equation*}
r_{\bb}(F)=\lim_{k\to\infty}\nnorm{F^{k}(\bx)}_{\bb}^{1/k} \leq \lim_{k\to\infty}\prod_{i=1}^d\bmaxi{i}{F(\bx)}{\bx}^{(\frac{1}{k}\sum_{j=0}^{k-1} A^j\bb)_i}=\cwu(F,\bx).
\end{equation*}
To show equality, assume first that $F$ has an eigenvector $\bu\in\Sn_{++}$. Then $\cwu(F,\bu)=r_{\bb}(F)$ and we are done. Now, suppose that $F$ does not have an eigenvector in $\Sn_{++}$, let $F^{(\delta_k)}$ and $(\blam^{(\delta_k)},\bx^{(\delta_k)})\in\R^d_{++}\times \Sn_{++}$ be as in Theorem \ref{BIGTHM}. Note that $\cwu(F,\bx)\leq \cwu(F^{(\delta_k)},\bx)$ as $F(\bx)\lek F^{(\delta_k)}(\bx)$ for every $k\in\N$ and $\bx\in\kone_{++}$.
It follows that
\begin{equation*}
r_{\bb}(F)=\lim_{k\to\infty}r_{\bb}\big(F^{(\delta_k)}\big) =\lim_{k\to\infty} \inf_{\bx\in\Sn_{++}}\cwu(F^{(\delta_k)},\bx) \geq \inf_{\bx\in\Sn_{++}}\cwu(F,\bx).
\end{equation*}
Now, we prove $r_{\bb}(F)=\max_{\by\in\Sn_{+}}\cwl(F,\by)$. To this end, let $\by\in\Sn_+$, if there exists $(i,j_i)\in\I$ such that $y_{i,j_i}>0$ and $F_{i,j_i}(\by)=0$, then $\cwl(F,\by)=0 \leq r_{\bb}(F)$. If this is not the case, then $\bt\in\R^d_{++}$ defined as 
\begin{equation*}
\theta_i = \min_{j_i\in[n_i], \ y_{i,j_i}>0} \frac{F_{i,j_i}(\by)}{y_{i,j_i}}\qquad \forall i \in[d],
\end{equation*}
satisfies $\bt\krog\by\leq F(\by)$. Hence, by Proposition \ref{Prop536}, we get
\begin{equation*}
\cwl(F,\by) = \prod_{i=1}^d \theta_i^{b_i}\leq r_{\bb}(F).
\end{equation*}
To conclude, note that, by Theorem \ref{BIGTHM}, we know that there exists $(\blam,\bu)\in\R^d_{+}\times \Sn_{+}$ such that $r_{\bb}(F)=\prod_{i=1}^d\lambda_i^{b_i}=\cwl(F,\bu)$.
\end{proof}
The proof of Theorem \ref{CW<1}, is inspired by the one of Theorem 21 in \cite{us}.
\begin{proof}[Proof of Theorem \ref{CW<1}]
As $\rho(A)<1$, Theorem \ref{Banachcor} implies the existence of $(\blam,\bub)\in\R^d_{++}\times\Sn_{++}$ such that $F(\bub)=\blam\krog\bub$. Moreover, the existence of $\bb\in\R^d_{++}$ such that $A^T\bb\leq\bb$ follows from Remark \ref{newrmq}. Set $R=\prod_{i=1}^d \lambda_i^{b_i}$, then we have $\cwl(F,\bub)=\cwu(F,\bub)=R$.
To prove the right-hand side of \eqref{CWeq<1}, it suffices to prove that for every $\byb\in\Sn_{+}$, we have $\cwl(F,\byb)\leq R$. So, let $\byb\in\Sn_{+}$, if there exists $(i,j_i)\in\I$ such that $y_{i,j_i}>0$ and $F_{i,j_i}(\by)=0$, then the inequality is clear. Thus, we may assume without loss of generality that 
$F_{i,j_i}(\by)>0$ for every $(i,j_i)\in\I$ such that $y_{i,j_i}>0$.
Let $\bt\in\R^d_{++}$ be defined as 
\begin{equation*}
\theta_i = \min_{j_i\in[n_i],\ y_{i,j_i}>0}\  \frac{u_{i,j_i}}{y_{i,j_i}}\qquad \forall i \in[d],
\end{equation*}
then $\bt\leq \ones$ because $\theta_i= \norm{\theta_i\by_i}_{\gamma_i}\leq \norm{\bu_i}_{\gamma_i}=1$ for all $i \in[d].$
Let $\boldsymbol{\Theta}= \bt^{-I}$, then $\boldsymbol{\Theta}\geq \ones$ and $\byb\lek\boldsymbol{\Theta}\krog\bub.$
Thus, for $\bs= \bb-A^T\bb\in\R^d_{+}$, we have
\begin{equation*}
\prod_{i=1}^d\Theta_i^{-\s_i}\leq 1.
\end{equation*}
Now, note that
$F\big(\boldsymbol{\Theta}\krog\bub\big)=\big(\blam\circ\boldsymbol{\Theta}^{A}\big)\krog \bub$ and thus
\begin{equation*}
\cwl(F,\byb)\leq\prod_{i=1}^d\Big(\min_{\substack{j_i\in[n_i]\\ y_{i,j_i}>0}}\frac{F_{i,j_i}(\boldsymbol{\Theta}\krog\bu)}{y_{i,j_i}}\Big)^{b_i}
=\prod_{i=1}^d\Theta_i^{-\s_i}\lambda_i^{b_i}\leq R.
\end{equation*}
The right-hand side of \eqref{CWeq<1} can be proved in a similar way. Indeed, if $\byb\in\Sn_{++}$, then $\prod_{i=1}^d\bmini{i}{\byb}{\bub}^{-\s_i}\geq 1$ and
\begin{equation*}
\cwu(F,\bu)\geq\prod_{i=1}^d\bmaxi{i}{F(\Bmin{\by}{\bub}\krog\bub)}{\by}^{b_i} 
=\prod_{i=1}^d\bmini{i}{\by}{\bub}^{-\s_i}\lambda_i^{b_i}\geq R.\qedhere
\end{equation*}
\end{proof}

It is known (see e.g. \cite{Horn}) that if $M\in\R^{n\times n}_+$ and $(\lambda,\bu)\in\R_{++}\times\R^n_{++}$ satisfies $M\bu=\lambda\bu$, then $\lambda = \rho(M)$, i.e. if $\theta\in\R$ is an eigenvalue of $M$, then $|\theta|\leq \lambda$. The next corollary can be seen as a generalization of this fact. To this end, we assume that $F\in\NB^d$ is the restriction of a map $\t F\colon V\to V$ where $ V=\R^{n_1}\times\ldots\times\R^{n_d}$. Similarly, suppose that $\norm{\cdot}_{\gamma_i}$ is well defined and monotonic on $\R^{n_i}$, i.e. \begin{equation}\label{monormn}
\forall \bx_i,\by_i\in\R^{n_i} \quad \text{if } \quad 0\leq \bx_i\leq \by_i\quad\text{then}\quad\norm{\bx_i}_{\gamma_i}\leq \norm{\by_i}_{\gamma_i}.
\end{equation}
Theorem 1 in \cite{JOHNSON199143} implies that \eqref{monormn} holds if and only if $\norm{\bx_i}_{\gamma_i}=\norm{\, |\bx_i|\, }_{\gamma_i}$, where the absolute value is taken component-wise.
\begin{cor}\label{extension_max}
Let $\t F\colon V \to V$ be such that $F = \t F|_{\kone_+}\in\NB^d$ and \begin{equation*}
|\t F(\bv)|\lek \t F(|\bv|)\qquad \forall \bv\in V. 
\end{equation*}
Moreover, suppose that $F$ satisfies the assumptions of either Theorem \ref{CW1} or Theorem \ref{CW<1} and let $\bb\in\R^d_{++}$ be given by these assumptions. Then, for every $(\bt,\bvb)\in\R^d\times V$ such that $\t F(\bvb)=\bt\krog\bvb$ and $\norm{\bv_i}_{\gamma_i}=1$ for all $i\in[d]$, we have
\begin{equation*}
\prod_{i=1}^d|\theta_i|^{b_i} \ \leq \ \inf_{\bx\in\Sn_{++}}\cwu(F,\bx) %\prod_{i=1}^d\lambda_i^{b_i}.
\end{equation*}
\end{cor}
\begin{proof}
Set $\bxb=|\bvb|$, then $\bxb\in\Sn_{+}$ and
\begin{equation*}
|\bt|\krog\bxb=|\bt\krog\bv| = | \t F(\bvb)| \lek \t F(\bxb)=F(\bx).
\end{equation*}
Hence, with Theorem \ref{CW1} or \ref{CW<1}, we get
\begin{equation*}
\prod_{i=1}^d|\theta_i|^{b_i} \leq \cwl(F,\bxb) \leq \inf_{\bz\in\Sn_{++}}\cwu(F,\bz).\qedhere
\end{equation*}
\end{proof}
As noted in Corollary 4.3 \cite{Fried}, this principle also holds for complex eigenvalues.

Now we prove Theorem \ref{rad<} which gives a sufficient condition that the eigenvalue associated with a positive eigenvector is strictly larger than those associated with any nonnegative eigenvector not being positive. This condition, as well as the proof of the theorem, is inspired by Theorem 6.1.7 in \cite{NB}, where the result is proved for homogeneous, order-preserving maps on solid closed cones in a finite dimensional space that are semi-differentiable at their positive eigenvector.
\begin{proof}[Proof of Theorem \ref{rad<}]
Let $\norm{\cdot}$ be any norm on $\R^{n_1}\times \ldots\times \R^{n_d}$ and $\blam\in\R^d_{++}$ be such that $F(\bu)=\blam\krog\bu$. We first prove the statement for $\blam=\ones$, then we show how to transfer the proof to the case $\blam\neq\ones$. By the chain rule, we have $DF(\bub)^k=DF^{k}(\bub)$ for every $k\in\N$. For the sake of contradiction assume that there exists $(\bt,\bvb)\in\R^d_{+}\times(\Sn_{+}\sauf\kone_{++})$ with $F(\bvb)=\bt\krog\bvb$ and $\prod_{l=1}^d\theta_l^{b_l}=1$. Let $\bal \in\R^d_{++}$ be defined as
\begin{equation}\label{defalpha}
\alpha_k=\min_{l_k\in[n_k], \ v_{k,l_k}>0}\,\frac{u_{k,l_k}}{v_{k,l_k}}\qquad \forall k\in[d],
\end{equation}
then  $0\lekk\bub-\bal\krog\bvb\lek\bu$. Hence
\begin{equation*}
-\bigg(\sum_{k=1}^\tau DF(\bub)^k(\bal\krog\bvb-\bub)\bigg)_i \in\R^{n_i}_{++}.
\end{equation*}
For $t\in (0,1]$, define $\byb(t)=(1-t)\bub+t\bal\krog\bvb\lekk\bub$ and note that
\begin{equation*}
F^{k}\big(\byb(t)\big)=F^{k}(\bub)+t\ DF(\bub)^k(\bal\krog\bvb-\bub) + t\ \norm{\bal\krog\bvb-\bub}\ \epsilon_k\big(t(\bal\krog\bvb-\bub)\big)
\end{equation*}
where $\lim_{\norm{\bwb}\to 0}\epsilon_k(\bwb)=0$. If follows that, with $\bzb=\bal\krog\bvb-\bub$, we have
\begin{align*}
\sum_{k=1}^\tau \Big(F^k(\bub)-F^k\big(\byb(t)\big)\Big)= t \bigg(-\sum_{k=1}^\tau DF(\bu)^k\bzb -\norm{\bzb}\sum_{k=1}^\tau\epsilon_k\big(t\bzb\big)\bigg).
\end{align*}
Since $\lim_{t\to 0}\sum_{k=1}^\tau\epsilon_k\big(t\bzb\big)=0$ and $-\sum_{k=1}^\tau \big(DF(\bu)^k\bzb\big)_i\in\R^{n_i}_{++}$, there exists $s\in (0,1]$ such that for every $t\in (0,s]$, it holds
\begin{equation*}
\sum_{k=1}^\tau \Big(F_i^k(\bub)-F_i^k\big(\byb(t)\big)\Big)\in\R^{n_i}_{++}.
\end{equation*}
For all $t\in (0,1]$, we have $\bal\krog\bvb \lek\byb(t)$ and thus 
\begin{equation*}
\sum_{k=1}^\tau \Big(F^k\big(\byb(t)\big)-F^k(\bal\krog\bv)\Big)\in\kone_{+}.
\end{equation*}
It follows with $\blam = \ones$ and $F(\bu)=\bu$ that
\begin{equation*}
\sum_{k=1}^{\tau}F^k(\bal\krog\bv) \lek \tau \bu \andd \sum_{k=1}^{\tau}F_i^k(\bal\krog\bv)< \tau\bu_i.
\end{equation*}
So, for every $(j_1,\ldots,j_d)\in\J$, we have
\begin{equation*}
\tau\prod_{l=1}^du_{l,j_l}^{b_l}> \prod_{l=1}^d\bigg(\sum_{k=1}^{\tau}F^k_{l,j_l}(\bal\krog\bv)\bigg)^{b_l}=\prod_{l=1}^dv_{l,j_l}^{b_l}\bigg(\sum_{k=1}^\tau \big(\bal^{A^k}\big)_l\big(\bt^{\sum_{s=0}^{k-1}A^s}\big)_l\bigg)^{b_l}.
\end{equation*}
Using the inequality relating arithmetic and geometric mean, for $l\in[d]$ we get
\begin{equation*}
\sum_{k=1}^\tau \big(\bal^{A^k}\big)_l\big(\bt^{\sum_{s=0}^{k-1}A^s}\big)_l\geq \tau \prod_{k=1}^\tau \Big(\big(\bal^{A^k}\big)_l\big(\bt^{\sum_{s=0}^{k-1}A^s}\big)_l\Big)^{1/\tau}.
\end{equation*}
It follows that
\begin{align*}
\prod_{l=1}^d\bigg(\sum_{k=1}^\tau \big(\bal^{A^k}\big)_l\big(\bt^{\sum_{s=0}^{k-1}A^s}\big)_l\bigg)^{b_l} & \geq \tau \prod_{k=1}^\tau \prod_{l=1}^d\Big(\big(\bal^{A^k}\big)_l\big(\bt^{\sum_{s=0}^{k-1}A^s}\big)_l\Big)^{b_l/\tau}\\
&= \tau\prod_{k=1}^\tau \Big(\prod_{l=1}^{d}\alpha_l^{b_l}\Big)^{\frac{\rho(A)^k}{\tau}}\Big(\prod_{l=1}^{d}\theta_l^{b_l}\Big)^{\frac{1}{\tau}\sum_{s=0}^{k-1}\rho(A)^s}\\
&\geq \tau\prod_{l=1}^{d}\alpha_l^{b_l},
\end{align*}
where we have used that $\bal\leq \ones$ because $\bu,\bv\in\Sn_+$. Thus, we get
\begin{equation*}
\prod_{l=1}^du_{l,j_l}^{b_l}> \prod_{l=1}^d(v_{l,j_l}\alpha_l)^{b_l} \qquad \forall (j_1,\ldots,j_d)\in\J,
\end{equation*}
a contradiction to \eqref{defalpha}. \\ 
Now, suppose that $F(\bu)=\blam\krog\bu$ with $\blam\neq \ones$, then $F'\in\NB^d$ defined as $F'(\bx)=\blam^{-I}\krog F(\bx)$ satisfies all the assumptions in the statement and is such that $F'(\bu)=\bu$. Moreover, if $(\bt,\bvb)\in\R^d_{+}\times(\Sn_{+}\sauf\kone_{++})$ satisfies $F(\bvb)=\bt\krog\bvb$, then $F'(\bv)=(\blam^{-I}\circ\bt)\krog\bv$ and thus 
\begin{equation*}
\prod_{l=1}^d \bigg(\frac{\theta_l}{\lambda_l}\bigg)^{b_l}<1 \qquad \implies \qquad 
\prod_{j=1}^d\theta_j^{b_j}<\prod_{j=1}^d\lambda_j^{b_j}=r(F).\qedhere
\end{equation*}
\end{proof}
Note that, if in Theorem \ref{rad<} the matrix $DF(\bu)$ is irreducible, then there exists $\tau\in\N$ such that $\big(I+DF(\bu)\big)^\tau$ has positive entries and \eqref{dirr} is satisfied. However, except when $d=1$, the converse is not true in general as shown by the next example.
\begin{ex}\label{irrrrrr_ex}
Let $d=2$, $n_1=n_2=2$ and $F\in\NB^d$ with
\begin{equation*}
F(\bx_1,\bx_2)=\bigg(\begin{pmatrix}((x_{1,1}x_{1,2})^{1/2}x_{2,1})^{1/2}\\ ((x_{1,1}x_{1,2})^{1/2}x_{2,2})^{1/2} \end{pmatrix},\begin{pmatrix}(x_{2,1}x_{2,2})^{1/2}\\ (x_{2,1}x_{2,2})^{1/2}\end{pmatrix}\bigg).
\end{equation*}
Then, $F(\ones)=\ones$ and all the assumptions of Theorem \ref{rad<} are satisfied, however $DF(\ones)$ is not irreducible.
\end{ex}
Nevertheless, as proved in the next proposition, if $F$ satisfies the assumptions of Theorem \ref{rad<} and $\A(F)$ is irreducible, then $DF(\bu)$ is irreducible. We also prove a similar result for the notion of primitivity which will be used to analyse the convergence of the power method.
\newcommand{\B}{\mathcal{B}}
\begin{prop}\label{hardcorelema}
Let $F\in\NB^d$, $\bx\in\kone_{++}$ and suppose that $F$ is differentiable at $\bx$.
\begin{enumerate}
\item If $DF(\bx)$ is irreducible, then $\A(F)$ is irreducible.\label{hardprop1} Conversely, if $\A(F)$ is irreducible and there exist $l\in[d]$ and $\xi\in\N$ such that
\begin{equation}\label{eqproppp}
\forall \bw\in\kone_{+}\saufzero, \quad \text{if} \quad \bx = \sum_{k=1}^\tau DF(\bu)^k\bw \quad \text{then} \quad \bx_l\in\R^{n_l}_{++},
\end{equation}
then $DF(\bx)$ is irreducible.
\item If $DF(\bx)$ is primitive, then $\A(F)$ is primitive. Conversely, if $\A(F)$ is primitive and there exist $l\in[d]$ and $\xi\in\N$ such that\label{hardprop2}
\begin{equation}\label{eqproppp2}
\forall \bw\in\kone_{+}\saufzero, \quad \text{if} \quad \bx = DF(\bu)^\xi\bw \quad \text{then} \quad \bx_l\in\R^{n_l}_{++},
\end{equation}
then $DF(\bx)$ is primitive.
\end{enumerate}
\end{prop}
\begin{proof}
Let $A=\A(F)$ and $L=DF(\bx)$.\\
\eqref{hardprop1} Suppose that $L$ is irreducible. Let $(s,j_s),(i,j_i)\in\I$, be such that $L_{(i,j_i),(s,j_s)}>0$. From Lemma \ref{Eulerthm}, we know that $\ps{\grad_sF_{i,j_i}(\bx)}{\bx_s}=A_{i,s}F_{i,j_i}(\bx)$ and $A_{i,s}>0$ by Lemma \ref{basics_lem}. This means that
\begin{equation}\label{fiimp}
\forall (i,j_i),(s,j_s)\in\I, \qquad L_{(i,j_i),(s,j_s)}>0, \quad \implies\quad A_{i,s}>0.
\end{equation}
 Now, let $G(L)=(\I,E(L))$ and $G(A)=([d],E(A))$ be the graphs associated with the matrices $L$ and $A$ respectively. Since $L$ is irreducible, $G(L)$ is strongly connected. In particular, for every $i,j\in[d]$ there exists a path from $(i,1)$ to $(j,1)$ in $G(L)$. By the above discussion, this path induces naturally a path from $i$ to $j$ in $G(A)$. This implies that $G(A)$ also is strongly connected and therefore $A$ is irreducible.\\
For the second part, we introduce some mappings that allow an easier understanding of the behaviour of the blocks $D_{j}F_i(\bx)$ in $DF(\bx)$. Let $N=|\I|=\sum_{l=1}^dn_l$ and, for $i,j\in[d]$, consider the mappings $B_{i,j}\colon\R^{N\times N}\to\R^{n_i\times n_j}$ such that
\begin{equation*}
M = \begin{pmatrix} B_{1,1}(M) & \cdots & B_{1,d}(M) \\ \vdots & & \vdots \\ B_{d,1}(M)&\cdots& B_{d,d}(M)\end{pmatrix} \qquad \forall M\in\R^{N\times N}.
\end{equation*}
That is $M$ can be written as a block matrix with $d^2$ blocks given by $B_{i,j}(M)$ for $i,j\in[d]$.
In particular, $B_{i,j}(DF(\bx))=D_jF_i(\bx)$ for every $i,j\in[d]$. We describe $\A(F)$ with $\B\colon\R^{N\times N}_+\to\{0,1\}^{d\times d}$ defined for $i,j\in[d]$ as 
\begin{equation*}
\big(\B(M)\big)_{i,j} =\begin{cases} 1 & \text{if } B_{i,j}(M) \text{ has at least one zero entry per row,} \\ 0 & \text{else.}\end{cases}
\end{equation*}
Then, there are $\beta,\t\beta>0$ such that $\beta\A(F)\leq\B(DF(\bx))\leq\t \beta\A(F)$ by Lemma \ref{basics_lem}, i.e. $\B(DF(\bx))$ and $\A(F)$ have the same zero pattern. Note that for every $\bar M,\t M\in\R^{N\times N}_+$, there exist $\alpha,\t \alpha>0$ such that \begin{equation*}
\B(\t M \bar M)\geq \alpha\B(\t M)\B(\bar M)\andd\B(\t M+ \bar M)\geq \t\alpha\big(\B(\t M)+\B(\bar M)\big).
\end{equation*}
Let $M\in\R^{N\times N}_+$, we claim that:
\begin{enumerate}[(a)]
\item If there exists $s\in [d]$ such that $B_{s,j}(M)\in\R^{n_s\times n_j}_{++}$ for every $j\in[d]$ and $\B(M)$ is primitive, then $M$ is primitive.\label{primproof}
\item If there exists $s\in [d]$ such that $B_{s,i}(M)\in\R^{n_s\times n_i}_{++}$ for every $i\in[d]$ and $\B(M)$ is irreducible, then $M$ is irreducible.\label{irproof}
\end{enumerate}
 \eqref{primproof} If $\B_{s,j}(M)>0$ for every $j\in[d]$, then $\B_{s,j}(M^k)>0$ for every $k\in\N$ and $j\in[d]$. Now, as $\B(M)$ is primitive, there exists $\tau\in \N$ such that $\B(M)^\tau\in\R^{d\times d}_{++}$. It follows that $\B(M^\tau)\in\R^{d\times d}_{++}$. Thus, we have
 \begin{equation*}
 B_{i,j}(M^{\tau+1}) \geq B_{i,s}(M^\tau)B_{s,j}(M) \in\R^{n_i\times n_j}_{++} \qquad \forall i,j\in[d],
 \end{equation*}
i.e. $M$ is primitive.\\
 \eqref{irproof} If $\B(M)$ is irreducible, there exists $\tau\in\N$ such that $P=\sum_{k=1}^\tau\B(M)^k\in\R^{d\times d}_{++}$. Let $\t M =\sum_{k=1}^\tau M^k$, then $B_{s,j}(\t M)\geq B_{s,j}(M)$ for every $j\in [d]$ and there exists $\alpha >0$ such that $\B(\t M)\geq \alpha P$. Thus, it follows from \eqref{primproof} that there exists $\nu\in\N$ such that $\t M^\nu\in\R^{N\times N}_{++}$. To conclude, note that there exists $\beta>0$ such that $\sum_{k=1}^{\tau\nu}M^k \geq \beta\t M^\nu$.\\ 
Now, suppose that $A$ is irreducible and $L$ satisfies \eqref{eqproppp}, then $\t L = \sum_{k=1}^{\xi}L^k$ is such that $B_{l,i}(\t L)>0$ for every $i\in[d]$. Moreover, if $A$ is irreducible, then $\t A=\sum_{k=1}^\xi A^k$ is irreducible and there exists $\alpha>0$ such that $\B(\t L)\geq \alpha\t A$ and thus $\B(\t L)$ is also irreducible. It follows from \eqref{irproof} that $\t L$ is irreducible. Now, as $\t L$ is irreducible, there exists $\sigma\in\N$ such that $\sum_{k=1}^\sigma\t L^k\in\R^{N\times N}_{++}$. The irreducibility of $L$ finally follows from the fact that there exists $\beta>0$ such that $\sum_{k=1}^{\sigma\xi}L^k \geq \beta\sum_{k=1}^{\sigma}\t L^k$.\\
\eqref{hardprop2} Suppose that $L$ is primitive. Then, there exists $\tau>0$ such that $L^{\tau}\in\R^{N\times N}_{++}$. We prove by induction, that if $(i,j_i),(s,j_s)\in\I$ are such that $(L^k)_{(i,j_i),(s,j_s)}>0$ then $(A^{k})_{i,s}>0$. The case $k=1$ is given by \eqref{fiimp}. So, suppose that the relation true for some $k\geq 1$ and let $(i,j_i),(s,j_s)\in\I$ be such that $(L^{k+1})_{(i,j_i),(s,j_s)}>0$. Then, there exists $(t,j_t)\in \I$ such that $(L^{k})_{(i,j_i),(t,j_t)}>0$ and $L_{(t,j_t),(s,j_s)}>0$. The induction assumption and \eqref{fiimp} imply that $(A^{k})_{i,t}>0$ and $A_{t,s}>0$. Hence $(A^{k+1})_{i,s}>0$ which proves our claim. In particular, it follows that $A^{\tau}\in\R^{d\times d}_{++}$ and therefore $A$ is primitive.\\
Finally, assume that $L$ satisfies \eqref{eqproppp2} and $A$ is primitive. Then, we have that $B_{l,j}(L^{\xi})\in\R^{n_l\times n_j}_{++}$ for every $j\in[d]$. Thus \eqref{primproof} implies that $L^{\xi}$ is primitive. It follows that there exists $\tau>0$ such that $L^{\tau\xi}\in\R^{N\times N}_{++}$ and thus $L$ itself is primitive. 
 \end{proof}

% the following theorem, $d=1$, then the assumption on $DF(\bu)$ is equivalent as asking that , i.e. $DF(\bu)$ is irreducible. However, when $d>1$, it is not true in general that \eqref{dirr} implies irreducibility of $DF(\bu)$.
\subsection{Uniqueness of the positive eigenvector}
We derive a sufficient condition in order to ensure uniqueness of an eigenvector in $\S_{++}$. 
The following example shows that the irreducibility notions of Definition \ref{newirr} and the assumption of Theorem \ref{exist} are sufficient for existence of a positive eigenvector but do not guarantee uniqueness.
\begin{ex}\label{nonuniquevect}
Let $\epsilon \in (0,1)$, $d=1$, $n_1=3$ and $F\in \NB^d$ with
\begin{equation*}
F(a,b,c)=\big(\max(a,b,c), \max(\epsilon a,b), \max(\epsilon b,c)\big).
\end{equation*}
then, $(1,b,c)$ is a positive eigenvector of $F$ for every $b,c\in [\epsilon,1]$. Moreover, $F$ is irreducible and $\G(F)$ is strongly connected.
\end{ex}
The proof of the following theorem can be done as in Theorem 6.4.1 \cite{NB} (that deals with the case $d=1$) by noting that for every $\bb\in\R^d_{++}$ and $\bxb,\byb\in\kone_{++}$, Equation \eqref{simplegeod} implies
\begin{equation*}
\mu_{\bb}\big(\bx,t\bx+(1-t)\by\big)+\mu_{\bb}\big(t\bx+(1-t)\by,\by\big)=\mu_{\bb}\big(\by,\bx\big) \qquad \forall t\in [0,1].
\end{equation*}
%for every $ t_1,\ldots,t_d\in [0,1]$.
%With equation 2.5, the proof of the following theorem is direct consequence of the case d=1.
\begin{thm}\label{nonuniquethm}
Let $\bphi\in\kone_{++}$, $\bb\in\R^d_{++}$ and $G\colon \S^{\bphi}_{++}\to \S^{\bphi}_{++}$ be such that
\begin{equation*}
\mu_{\bb}(G(\bx),G(\by))\leq \mu_{\bb}(\bx,\by)\qquad \forall\bx,\by\in\S_{++}^{\bphi}. 
\end{equation*}
If there exist $\bu,\bw\in\S_{++}^{\bphi},\bu\neq\bw$ such that $G(\bu)=\bu$, $G(\bw)=\bw$ and $G$ is differentiable at $\bu$, then there exists $\bv\in \R^{n_1}\times\ldots\times\R^{n_d}, \bv\neq 0$ such that $\ps{\bv}{\bphi}=0$ and $DG(\bu)\bv=\bv$.
\end{thm}
The next lemma shows that when $\bu$ is a fixed point of $F\in\NB^d$, i.e. an eigenvector with eigenvalue $\blam = \ones$, and $F$ is differentiable at $\bu$, then one can find $\t \bb\in\R^d_{+}$ such that $\t \bb\krog\bu$ is an eigenvector of $DF(\bu)$.
\newcommand{\Sphi}{\S^{\bphib}}
\begin{lem}\label{difflem}
Let $\bphi\in\kone_{++}$, $F\in\NB^d$, $A=\A(F)$. If there exists $\bu\in\Sphi_{++}$ with $F(\bu)=\bu$, $F$ is differentiable at $\bub$ and $\t\bb\in\R^d_{+,0}$ satisfies $A\t\bb=\t\bb$, then
\begin{equation*}
DF(\bu)\t\bu=\t\bu \qquad \text{with}\qquad \t\bu = \t\bb\krog\bu.
\end{equation*}
Moreover, the mapping $G\colon\kone_{++}\to\kone_{++}$ defined as
\begin{equation*}
G(\bx)= \bigg(\frac{F_1(\bx)}{\ps{F_1(\bx)}{\bphi_1}},\ldots,\frac{F_d(\bx)}{\ps{F_d(\bx)}{\bphi_d}}\bigg)
\end{equation*}
is differentiable at $\bu$ and
\begin{equation}\label{gradG}
DG(\bu)\bz = DF(\bu)\bz-\big(\ps{DF_1(\bu)\bz}{\bphi_1},\ldots,\ps{DF_d(\bu)\bz}{\bphi_d}\big)\krog\bu
\end{equation}
for every $\bz\in V.$
\end{lem}
\begin{proof}
 By Lemma \ref{Eulerthm}, for all $k,i\in[d]$, we have $D_iF_k(\bu)\bu_i=A_{k,i}\bu_k$. Hence,
\begin{equation}\label{gradeval}
DF_i(\bu)(\bal \krog\bu)= \sum_{k=1}^d \alpha_kD_kF_i(\bu)\bu_k = \bigg( \sum_{k=1}^d A_{i,k}\alpha_k\bigg)\bu_i=(A\bal)_i \bu_i
\end{equation}
holds every $\bal\in\R^d_{++}$ which implies that $DF(\bu)\t\bu=(A\t\bb)\krog\bu=\t \bu$.\\
Now, if $F$ is differentiable at $\bxb\in\kone_{++}$, then
\begin{equation*}
D_kG_{i}(\bxb)=\frac{\ps{F_i(\bxb)}{\bphi_i}D_kF_{i}(\bxb)-F_{i}(\bxb)\bphib_i^TD_kF_i(\bxb)}{\ps{F_i(\bxb)}{\bphi_i}^2}\qquad \forall k, i \in[d].
\end{equation*}
 In particular, if $\bxb=\bub\in\Sphi_{++}$ and $F(\bu)=\bu$, then 
\begin{equation*}
D_kG_{i}(\bu)=D_kF_{i}(\bu)-\bu_{i}\bphi_i^TD_kF_i(\bu).\qedhere
\end{equation*}
\end{proof}
We prove Theorem \ref{unique} which extends Theorem 6.4.6 in \cite{NB} to the case $d\geq 1$.
\newcommand{\tbphi}{\bar{\boldsymbol{\varphi}}}
%\begin{thm}\label{unique}
%Let $\bphi\in\kone_{++}$, $F\in\NB^d$ and $A=\A(F)$. Suppose that $\rho(A)=1$ is a simple eigenvalue of $A$, there exists $\bb\in\Dn$ such that $A^T\bb=\bb$, $F$ has an eigenvector $\bu\in\S_{++}^{\bphi}$ and $F$ is differentiable at $\bu$. Let $L=DF(\bu)$, if for every $\bv\in V\saufzero$ with $L\bv=\rho(L)\bv$ there exists $\bbe\in\R^d\saufzero$ such that $\bv = \bbe\krog\bu$, then $\bu$ is the unique eigenvector of $F$ in $\Sphi_{++}$.
%\end{thm}
\begin{proof}[Proof of Theorem \ref{unique}]
Let $\t\bb,\bb\in\Dn$ be such that $A\t\bb=\t\bb$ and $A^T\bb = \bb$. These vectors always exist because $A$ is assumed to be irreducible. Suppose by contradiction that there exists $\bwb\in\Sphi_{++}\sauf\{\bu\}$ and $\t\blam\in\R^d_{++}$ such that $F(\bw)=\t\blam\krog\bw$. Let $\t F\in\NB^d$ be defined as $\t F(\bx)=\blam^{-I}\krog F(\bx)$ for every $\bx\in\kone_+$. Then, we have $\t F(\bu)=\bu$, $L=D\t F(\bu)$ and $\t F(\bw)=(\blam^{-I}\circ\t \blam)\krog \bw$. Lemma \ref{difflem} implies that $\t\bu=\t\bb\krog\bu\in\kone_{++}$ satisfies $L\t\bu=\t\bu$. Theorem \ref{opcharac} implies that $L$ is a nonnegative matrix because $\t F$ is order-preserving. Hence, Proposition \ref{Prop536} and $L\t\bu=\t\bu\in\kone_{++}$ imply that $\rho(L)=1$. Let $G$ be defined as in Lemma \ref{difflem}, then $G$ is non-expansive by Lemma \ref{contract}. Thus, Theorem \ref{nonuniquethm} implies the existence of $\bv\in \R^{n_1}\times\ldots\times\R^{n_d}, \bv\neq 0$ such that $\ps{\bv}{\bphi}=0$ and
\begin{equation}\label{contreq}
L\bv-\bal\krog\bu=\bv \qquad \text{where}\qquad\bal= \big(\ps{L\bv}{\bphi_1},\ldots,\ps{L\bv}{\bphi_d}\big).
\end{equation}
First, suppose that $\ps{\bb}{\bal}=0$. Then for any $\tbphi\in\kone_{+,0}$ with $\ps{\bu_i}{\tbphi_i}=1$ for all $i\in[d]$, we have 
\begin{equation}\label{goodtrick}
\sum_{i=1}^d\ps{\big(L\bv\big)_i}{b_i\tbphi_i}= \sum_{i=1}^d\ps{\bv_i}{b_i\tbphi_i}+\sum_{i=1}^d\alpha_ib_i\ps{\bu_i}{\tbphi_i} = \sum_{i=1}^d\ps{\bv_i}{b_i\tbphi_i}.
\end{equation}
Let $(i,j_i)\in\I$ and define $\t\be^{(i,j_i)}\in\R^{n_i}_{+,0}$ as
\begin{equation*}
\big(\t\be^{(i,j_i)}\big)_{l_i}=\begin{cases} 0 & \text{if } j_i=l_i \\ 1 & \text{else}\end{cases} \qquad \forall l_i\in[n_i].
\end{equation*}
Furthermore, consider $\tbphi^{(i,j_i)}\in\kone_{+,0}$ defined as %Plugging
\begin{equation*}
\tbphi^{(i,j_i)}=\bigg(\frac{\ones}{\ps{\ones}{\bu_1}},\ldots,\frac{\ones}{\ps{\ones}{\bu_{i-1}}},\frac{\ones-\t\be^{(i,j_i)}}{\ps{\ones-\t\be^{(i,j_i)}}{\bu_i}},\frac{\ones}{\ps{\ones}{\bu_{i+1}}},\ldots,\frac{\ones}{\ps{\ones}{\bu_d}}\bigg).
\end{equation*}
Plugging $\tbphi^{(i,j_i)}$ into Equation \eqref{goodtrick} for every $(i,j_i)\in\I$ implies the existence of $M\in\R^{\t N\times \t N}$, with $\t N = n_1+\ldots+n_d$, such that $ML\bv=M\bv$,
$M_{(i,j_i),(k,l_k)}>0$ for every $(i,j_i),(k,l_k)\in\I$ with $(i,j_i)\neq (k,l_k)$ and  $M_{(i,j_i),(i,j_i)}=0$ for every $(i,j_i)\in\I$. In particular, $M$ is invertible and thus $L\bv=\bv$. Hence, by assumption, there exists $\beta\in\R\saufzero$ such that $\bv=\beta\t\bu$. We obtain the contradiction 
\begin{equation*}
0=\ps{\bv}{\bphi}=\beta^{-1}\ps{\bv}{\bphi}=\ps{\t\bu}{\bphi}=\sum_{i=1}^d\t b_i\ps{\bu_i}{\bphi_i}=\sum_{i=1}^d\t b_i=1.
\end{equation*} 
Now, suppose that $\ps{\bb}{\bal}\neq 0$ and let $\norm{\cdot}$ be any monotonic norm on $\R^{n_1}\times \ldots\times \R^{n_d}$. Note that $A\bal \neq 0$ because it would imply the contradiction \begin{equation*}0=\ps{A\bal}{\bb}=\ps{\bal}{A^T\bb}=\ps{\bal}{\bb}.\end{equation*} 
Let $\nu\in\N$, with \eqref{contreq} and \eqref{gradeval}  we get
\begin{equation}\label{eq1contrad}
L^{\nu+1}\bv-\bv\ =\ \sum_{k=0}^{\nu} L^k(L\bv-\bv)\ =\ \sum_{k=0}^{\nu}L^{k}(\bal\krog\bu)\ =\ \sum_{k=0}^{\nu}(A^k\bal)\krog\bu.
\end{equation}
On the one hand, as $\t\bu>0$, there is $t>0$ with $-t\t\bu\lek\bv\lek t\t\bu$. It follows that $ 0 \lek L^{\nu+1}\bv+t\t\bu\lek 2t\t\bu$ because $-t\t\bu \lek L^{\nu+1}\bv \lek t\t\bu$. Thus,
\begin{equation}\label{eq2contrad}
\norm{L^{\nu+1}\bv}\leq \norm{L^{\nu+1}\bv+t\t\bu}+\norm{t\t\bu}\leq 3t\norm{\bu} \qquad \forall \nu\in\N.
\end{equation}
On the other hand, as $A$ is irreducible, we know from Theorem 1.1 \cite{Francesco} that 
\begin{equation*}
\lim_{k\to\infty} \frac{1}{k+1}\sum_{s=0}^kA^s = \frac{\t\bb\bb^T}{\ps{\bb}{\t\bb}}.
\end{equation*}
It follows that
\begin{equation}\label{eq3contrad}
\lim_{\nu\to\infty}\norm{\sum_{k=0}^{\nu}(A^k\bal)\krog\bu}= \infty.
\end{equation}
We finally obtain a contradiction by combining \eqref{eq1contrad}, \eqref{eq2contrad} and \eqref{eq3contrad}.
\end{proof}
We note that already in the case $d=1$, the assumption on the Jacobian $DF(\bu)$ of $F$ in Theorem \ref{unique} is less restrictive than requiring $DF(\bu)$ to be irreducible. Indeed, there exists some $M\in\R^{n\times n}_+$ such that $\dim\big(\ker(\rho(M)I-M)\big)=1$ but $M$ is not irreducible as shown by the following example:
\begin{ex}\label{Krein_notirr}
Let $M\in\R^{3\times 3}_+$ be defined as 
\begin{equation*}
M = \begin{pmatrix} 2 & 0 & 0 \\ 1 & 1 & 0 \\ 1 & 0 & 1 \end{pmatrix}.
\end{equation*}
Then the eigenvectors of $M$ are $(1,1,1)^T$, $(0,1,0)^T$, $(0,0,1)^T$ with respective eigenvalues $\rho(M)=2$, $1$ and $1$. In particular, $\dim\big(\ker(2I-M)\big) = 1$ and $M$ can not be irreducible as it has a nonnegative eigenvector.
\end{ex}
%We note however, that the assumptions on $L$ and $A$ in Theorem \ref{unique} implies that $L$ satisfies the Kre\u{i}n-Rutman condition. Indeed, let $\bv\in \R^{n_1}\times\ldots\times \R^{n_d}$ with $\bv \neq 0$ and $L\bv=\rho(L)\bv$. Without loss of generality, we may assume $\rho(L)=1$. Now, if there exists $\bbe\in\R^d\saufzero$ such that $\bv = \bbe \krog \bu$, then, with equation \eqref{gradeval}, we have 
%\begin{equation*}
%\bbe \krog \bu = \bv = L\bv =  L (\bbe \krog \bu) = (A\bbe)\krog \bu.
%\end{equation*}
%Which implies that $\bbe = A\bbe$ because $\bu\in\kone_{++}$. Hence $\bbe$ is an eigenvector of $A$ corresponding to $\rho(A)$. In particular, if $\rho(A)$ is simple then there exists $\alpha\neq 0$ such that $\bbe=\alpha\t\bb$ and $\bv =\alpha \t \bu$ where $\t\bb$, $\t\bu$ are defined as in Lemma \ref{difflem}. This shows that $\dim(\ker(I-L))=1$, i.e. $L$ satisfies the Kre\u{i}n-Rutman condition.
\newcommand{\pmi}[1]{\bx^{#1}}
\newcommand{\lmin}[1]{\widehat{\alpha}(#1)}
\newcommand{\lmax}[1]{\widecheck{\alpha}(#1)}
\section{Convergence of the \algo{}}\label{PM_section}
In this section, we propose a general power method type algorithm for the computation of the unique positive eigenvector of order-preserving multi-homogeneous mappings, that is we extend \eqref{PMlin} of Theorem \ref{linear_PF} to maps in $\NB^d$. 

Let $F\in\NB^d$ and $\pmi{0}\in\Sn_{++}$. Consider the sequence $(\pmi{k})_{k=0}^{\infty}\subset \Sn_{++}$ defined as 
\begin{equation}\label{defpm}
\pmi{k}= \bigg(\frac{F_1(\pmi{k-1})}{\norm{F_1(\pmi{k-1})}_{\gamma_1}},\ldots,\frac{F_d(\pmi{k-1})}{\norm{F_d(\pmi{k-1})}_{\gamma_d}}\bigg) \qquad \forall k\in\N.
\end{equation}
This is a natural generalization of the iteration process of the well-known power method. We analyse the convergence of this sequence towards the positive eigenvector of $F$, when it exists and is unique.
The contractive case, i.e. $\rho\big(\A(F)\big)<1$, is easier to study as the (linear) convergence of $(\pmi{k})_{k=0}^{\infty}$ to a positive eigenvector follows directly from Theorem \ref{Banachcor}. 
%which we recall for the reading convenience.
%\begin{thm}\label{conv<1}
%Let $F\in\NB^d$, $\pmi{0}\in\Sn_{++}$ and $A=\A(F)$. If $\rho(A)<1$, then there exist a unique eigenpair $(\blam,\bu)\in\R^d_{++}\times\Sn_{++}$ such that $F(\bu)=\blam\krog\bu$. Moreover, the sequence $(\pmi{k})_{k=0}^{\infty}\subset\Sn_{++}$ defined as in \eqref{defpm} satisfies $\lim_{k\to\infty} \pmi{k} = \bu$ and there exist $\bb\in\R^d_{++}$, $r\in[\rho(A),1)$ such that $A^T\bb\leq r\bb$ and
%\begin{equation*}
%\mu_{\bb}(\pmi{k},\bu)\ \leq\ \frac{r^k}{1-r}\ \mu_{\bb}(\pmi{1},\pmi{0}) \qquad \forall k \in \N,
%\end{equation*}
%where $\mu_{\bb}$ is the weighted Hilbert metric defined in Section \ref{Hilbert_metric}.
%\end{thm}
The nonexpansive case, i.e. $\rho\big(\A(F)\big)=1$, requires more effort. Again we generalize the corresponding results in the case $d=1$ (see Theorem 2.3 \cite{Nussb} and Corollary 5.6.8 \cite{NB}).
\begin{thm}\label{conv1}
Let $F\in\NB^d$, $\pmi{0}\in\Sn_{++}$ and $A=\A(F)$. Suppose that $\rho(A)=1$ and there exist $(\blam,\bu)\in\R^d_{++}\times\Sn_{++}$ such that $F(\bu)=\blam\krog\bu$. If $F$ is differentiable at $\bu$ and $DF(\bu)$ is primitive, then $\bu$ is the unique eigenvector of $F$ in $\Sn_{++}$ and the sequence $(\pmi{k})_{k=0}^{\infty}\subset\Sn_{++}$ defined in \eqref{defpm} satisfies $\lim_{k\to\infty} \pmi{k} = \bu$.
\end{thm}
Note that if $F$ satisfies the assumptions of Theorem \ref{conv1}, then Proposition \ref{hardcorelema} implies that $F$ satisfies the assumptions of Theorem \ref{unique} which ensure the uniqueness of a positive eigenvector. Unfortunately, we could not derive a convergence rate in the nonexpansive case. However, we believe that similar techniques as in Corollary 5.2 \cite{Fried} and Theorem 36 \cite{us} could be used to prove an asymptotic linear convergence rate. Nevertheless, we note that the Collatz-Wielandt principle of Section \ref{CW_M_U} can be used to obtain a stopping criterion for the power method as it has been used for tensor spectral problems (see e.g. \cite{Boyd,NQZ,Chang_rect_eig,us}). The following proposition generalizes Proposition 28 in \cite{us}.
\begin{prop}\label{monoprop}
Let $F\in\NB^d$ and $\pmi{0}\in\Sn_{++}$. Suppose that $F$ satisfies either the assumptions of Theorem \ref{Banachcor} or those of Theorem \ref{conv1}. Then, there exists $\bb\in\R^d_{++}$ and $(\blam,\bu)\in\R^d_{++}\times\Sn_{++}$ such that $F(\bu)=\blam\krog\bu$ and $A^T\bb\leq\bb$. Moreover, the sequences $(\lmax{k})_{k=1}^{\infty},(\lmin{k})_{k=1}^{\infty}\subset \R_{++}$ defined as
\begin{equation*}
\lmax{k}=\prod_{i=1}^d \Big(\max_{j_i\in[n_i]}\frac{F_{i,j_i}(\pmi{k})}{(\pmi{k})_{i,j_i}}\Big)^{b_i} \quad \text{and}\quad \lmin{k}=\prod_{i=1}^d \Big(\min_{j_i\in[n_i]}\frac{F_{i,j_i}(\pmi{k})}{(\pmi{k})_{i,j_i}}\Big)^{b_i},
\end{equation*}
satisfy
\begin{equation*}
\lmin{k}\ \leq \ \lmin{k+1}\ \leq \ \prod_{i=1}^{d}\lambda_i^{b_i}\ \leq \ \lmax{k+1}\ \leq \ \lmax{k} \qquad \forall k\in\N,
\end{equation*}
and for every $\epsilon>0$, $k\in\N$, if $\big(\lmax{k}-\lmin{k}\big) < \epsilon$ then
\begin{equation*}
\Big|\frac{\lmax{k}-\lmin{k}}{2}-\prod_{i=1}^d\lambda_i^{b_i}\Big|\ <\ \frac{\epsilon}{2}.
\end{equation*}
\end{prop}
\subsection{Convergence analysis}
First, we need the subsequent lemma which can be proved in the same way as Lemma 6.5.7 \cite{NB} dealing with the case $d=1$.
%In particular, the conclusion of this lemma is related to the notions of ``strong monotonic maps'' \cite{BRUALDI196631}, ``upper primitive maps'' \cite{KLOEDEN200097} and maps that ``satisfies condition \textsf{U} at $\bu$'' \cite{NB}. 
%We note also that similar results have been obtained in \cite{us} and \cite{Qi_rect_1}.
\begin{lem}\label{lightlemma}
Let $F\in\NB^d$ and $\bu\in\kone_{++}$ with $F(\bu)=\bu$. If $F$ is differentiable at $\bu$ and there exists $\nu\in\N$ such that $DF(\bu)^{\nu}>0$, i.e. $DF(\bu)$ is primitive, then 
\begin{equation*}
F^{\nu}(\bu)< F^{\nu}(\bx)\qquad\forall \bx\in\kone_{++}\qquad\text{ with }\qquad\bu\lekk\bx.
\end{equation*}
\end{lem}
We also need the following:
\begin{lem}\label{monotheory}
Let $F\in\NB^d$, $A=\A(F)$, $\bb\in\Dn$ with $A^T\bb=\bb$, $\bx\in\kone_{++}$ and $\bu\in\kone_{++}$ with $F(\bu)=\bu$. For every $k\in\N$, set
\begin{equation*}
\alpha_k= \prod_{i=1}^d\mathfrak{m}_i\big(F^k(\bx)\big/\bu\big)^{b_i} \andd \beta_k=  \prod_{i=1}^d\mathfrak{M}_i\big(F^k(\bx)\big/\bu\big)^{b_i} ,
\end{equation*}
then
\begin{equation*}
\alpha_k\ \leq \ \alpha_{k+1}\ \leq \ \beta_{k+1}\ \leq \ \beta_k \qquad \forall k\in\N.
\end{equation*}
\end{lem}
\begin{proof}
We have
\begin{equation*}
\alpha_{k+1} \geq\prod_{i=1}^d\mathfrak{m}_i\Big({F\big(\mathfrak{m}({F^{k}(\bx)}/{\bu})\krog\bu\big)}\Big/{\bu}\Big)^{b_i} =\alpha_k.
\end{equation*}
Similarly, $\beta_{k+1}\leq\beta_k$ for all $k$. To conclude, note that $\alpha_k\leq \beta_k$ for every $k\in\N$.
\end{proof}
Finally, we recall known results of fixed point theory:
For $\bx\in\kone_{++}$ and $F\in\NB^d$, the orbit and $\omega$-limit set of $\bx$ under $F$ are respectively defined by 
\begin{equation*}
\O(F,\bx)=\big\{ F^{k}(\bx)\mid k\in\N\big\} 
\end{equation*}
and
\begin{equation*}
\omega(F,\bx) = \Big\{\by\in\kone_{++}\ \Big|\ \lim_{l\to \infty} F^{k_l}(\bx)=\by \text{ for some } (k_l)_{l=1}^{\infty}\subset\N \text{ with } \lim_{l\to \infty}k_l=\infty\Big\},
\end{equation*}
i.e. $\omega(F,\bx)$ is the set of accumulation points of $\O(F,\bx)$.
For $F\in\NB^d$, Theorem 3.1.7 and Lemmas 3.1.2, 3.1.3 and 3.1.6 in \cite{NB} imply the following:
\begin{enumerate}[(I)]
\item If $F$ is non-expansive with respect to the weighted Thompson metric $\bar\mu_{\bb}$ on $\kone_{++}$ and there exists $\bu\in\kone_{++}$ such that $\big(F^{k}(\bu)\big)_{k=1}^{\infty}\subset\kone_{++}$ has a bounded subsequence, then $\O(F,\bx)$ is bounded for each $\bx\in \kone_{++}$.\label{fixp1}
\item If $\bx\in\kone_{++}$ is such that $\O(F,\bx)$ has a compact closure, then $\omega(F,\bx)$ is a non-empty compact set and $F\big(\omega(F,\bx)\big)\subset\omega(F,\bx)$.\label{fixp2}
\item If $\bx\in\kone_{++}$ is such that $\O(F,\bx)$ has a compact closure and $|\omega(F,\bx)|=p$, then there exists $\bz\in\kone_{++}$ such that $\lim_{k\to\infty}F^{pk}(\bx)=\bz$ and $\omega(F,\bx)=\O(F,\bz)$.\label{fixp3}
\item If $F$ is non-expansive with respect to $\bar\mu_{\bb}$, then for all $\bx\in\kone_{++}$ and $\by\in\omega(F,\bx)$, we have that $\omega(F,\by)=\omega(F,\bx)$.\label{fixp4}
\end{enumerate}
Property \eqref{fixp1} is also know as Calka's Theorem \cite{AleksanderCalka1984}. We are now ready to prove Theorem \ref{conv1} which turns out to be a special case of Corollary 6.5.8 in \cite{NB} when $d=1$. 
\begin{proof}[Proof of Theorem \ref{conv1}]
By Proposition \ref{hardcorelema} \eqref{hardprop2}, we know that $A$ is primitive. Hence, by Theorem \ref{unique}, $\bu$ is the unique positive eigenvector of $F$. Furthermore, there exist $\bb,\t\bb\in\Dn$ and $\nu\in\N$ such that $A^T\bb=\bb$, $A\t\bb=\t\bb$ and $DF(\bu)^{\nu}>0$. Now, let $\blam\in\R^d_{++}$ with $F(\bu)=\blam\krog\bu$ and $\hat F\in\NB^d$ defined as $\hat F(\bx)=\blam^{-I}\krog F(\bx)$. Then $\A(\hat F)=A$, $\bu$ is the unique eigenvector of $\hat F$, $\hat F$ is differentiable at $\bu$ and $D\hat F(\bu)^{\nu}>0$. 
 We show that for every $\bx\in\kone_{++}$, there exists $\bxi\in\R^d_{++}$ such that $\omega(\hat F,\bx)=\{\bxi\krog\bu\}$. Let $\bx\in\S^{\bphi}_{++}$, then the sequences $(\alpha_k)_{k=1}^\infty,(\beta_k)_{k=1}^\infty\subset\R_{++}$ defined in Lemma \ref{monotheory} converge towards some $\alpha,\beta>0$. In particular, it holds 
\begin{equation}\label{stable_orbit} 
\alpha = \prod_{l=1}^d \bmini{l}{\bz}{\bu}^{b_l} \andd \beta = \prod_{l=1}^d \bmaxi{l}{\bz}{\bu}^{b_l} \qquad \forall \bz\in\omega(\hat F,\bx).
\end{equation}
By Lemma \ref{contract}, we know that $\hat F$ is non-expansive with respect to the weighted Thompson metric $\bar\mu_{\bb}$ on $\kone_{++}$. Since $\hat F(\bub)=\bub$, we have $\hat F^{k}(\bub)=\bub$ for every $k\in\N$ and thus \eqref{fixp1} implies that $\O(\hat F,\bxb)$ is bounded. Now, let $\nu\in\N$ be such that $DF(\bu)^{\nu}>0$. It follows from \eqref{fixp2}, that $\hat F^{\nu}\big(\omega(\hat  F,\bxb)\big)\subset\omega(\hat F,\bx)$ and thus $\hat F^{\nu}(\bz)\in\omega(\hat F,\bx)$ for every $\bz\in\omega(\hat F,\bx)$. Suppose by contradiction that there exists $\bz\in\omega(\hat F,\bx)$ such that $\bz\neq\bal\krog\bu$ for every $\bal\in\R^d_{++}$. Then $\Bmin{\bz}{\bu}\krog\bu\lekk \bz$ and, with Lemma \ref{lightlemma}, we get 
\begin{equation*}
\Bmin{\bz}{\bu}^{A^\nu}\krog \hat F^{\nu}(\bu) = \hat F^{\nu}(\Bmin{\bz}{\bu}\krog\bu)\lekkk \hat F^{\nu}(\bz).
\end{equation*}
Thus, with \eqref{stable_orbit} and $\hat F^{\nu}(\bu)=\bu$, we obtain the contradiction
\begin{equation*}
\alpha =  \prod_{l=1}^d\bmini{l}{\bz}{\bu}^{b_l}\bmini{l}{\hat F^{\nu}(\bu)}{\bu}^{b_l}< \prod_{l=1}^d\bmini{l}{\hat F^{\nu}(\bz)}{\bu}^{b_l}=\alpha.
\end{equation*}
Hence, there exists $\bxi\in\R^d_{++}$ such that $\by=\bxi\krog\bub$ and \eqref{fixp4} implies that $\omega(\hat F,\bx)=\omega(\hat F,\bxi\krog\bu)$. As $A$ is primitive, we know from Theorem 1.1 \cite{Francesco} that
\begin{equation*}
\lim_{k\to\infty}A^k = B \qquad \text{where}\qquad B=\frac{\t\bb\bb^T}{\ps{\t\bb}{\bb}}.
\end{equation*}
In particular, we have
\begin{equation*}
\lim_{k\to\infty} \hat F^k(\bxi\krog\bu)=\lim_{k\to\infty} \bxi^{A^k}\krog \hat F^k(\bu)=\lim_{k\to\infty} \bxi^{A^k}\krog \bu=\bxi^{B}\krog\bu.
\end{equation*}
Hence, we have $\omega(\hat F,\bx)=\omega(\hat F,\bxi\krog\bu)=\{\bxi^B\krog\bu\}$. So, $\lim_{k\to\infty}\hat F^{k}(\bxb)=\bxi^B\krog\bub$ follows from \eqref{fixp3}. To conclude the proof, note that 
\begin{equation*}
\frac{\hat F_i(\by)}{\norm{\hat F_i(\by)}_{\gamma_i}}=\frac{ F_i(\by)}{\norm{F_i(\by)}_{\gamma_i}}\qquad\forall \by\in\kone_{++}, i\in[d],
\end{equation*}
thus $\lim_{k\to\infty} \bx^k=\bu$.
\end{proof}
Finally, to prove Proposition \ref{monoprop}, we need the following lemma which generalizes Proposition 28 in \cite{us}.
\renewcommand{\NF}{\t G}
\begin{lem}\label{monothm}
Let $F\in\NB^d$ and $(\blam,\bub)\in\R^d_{++}\times\Sn_{++}$ be such that $F(\bu)=\blam\krog\bu$. Let $\bb\in\Dn$ with $A^T\bb\leq \bb$, define $\t G\colon\Sn_{++}\to\Sn_{++}$ as 
\begin{equation*}
\t G(\bx)=\bigg(\frac{F_1(\bz)}{\norm{F_1(\bz)}_{\gamma_1}},\ldots,\frac{F_d(\bz)}{\norm{F_d(\bz)}_{\gamma_d}}\bigg) \qquad \forall \bz\in\Sn_{++},
\end{equation*}
and let $\cwl$, $\cwu$ be defined as in Section \ref{CW_M_U}. Then, for every $\bx\in\Sn_{++}$, we have
\begin{equation*}\label{mono}
\cwl(F,\bx)\ \leq\ \cwl(F,\t G(\bx))\ \leq\ \prod_{i=1}^d\lambda_i^{b_i}\ \leq\ \cwu(F,\t G(\bx))\ \leq \ \cwu(F,\bx).
\end{equation*}
\end{lem}
\begin{proof}
Let $\bx\in\Sn_{++}$, then $\mathfrak{m}\big(\NF(\bx)\big/\bx\big)\leq \ones\leq \mathfrak{M}\big(\NF(\bx)\big/\bx\big)$ because $\t G(\bx)\in\Sn_{++}$. Thus, with $\bs = \bb-A^T\bb \in\R^d_{+}$, we have
\begin{equation*}
\prod_{i=1}^d\mathfrak{M}_i\big(\NF(\bx)\big/\bx\big)^{-\s_i}\leq 1 \leq  \prod_{i=1}^d\mathfrak{m}_i\big(\NF(\bx)\big/\bx\big)^{-\s_i}
\end{equation*}
It follows that
\begin{align*}
\cwl(F,\t G(\bx)) & = \prod_{i=1}^d\norm{F(\bx)}_{\gamma_i}^{s_i}\mathfrak{m}_i\big(F^2(\bx)\big/F(\bx)\big)^{b_i} \\
&\geq \prod_{i=1}^d\norm{F(\bx)}_{\gamma_i}^{s_i}\bmini{i}{F(\Bmin{F(\bx)}{\bx}\krog\bx)}{F(\bx)}^{b_i} \\
&=\prod_{i=1}^d\norm{F(\bx)}_{\gamma_i}^{s_i}\bmini{i}{F(\bx)}{\bx}^{-s_i}\bmini{i}{F(\bx)}{\bx}^{b_i} \\
&=\prod_{i=1}^d\mathfrak{m}_i\big({\t G(\bx)}\big/{\bx}\big)^{-s_i}\bmini{i}{F(\bx)}{\bx}^{b_i} \geq \cwl(F,\bx).
\end{align*}
The inequality $\cwu(F,\t G(\bx))\leq \cwu(F,\bx)$ can be proved in the same by exchanging the roles of $\Bmin{\cdot}{\cdot}$ and $\Bmax{\cdot}{\cdot}$ and swapping the inequalities. Finally, Theorems \ref{CW1} and \ref{CW<1} conclude the proof.
\end{proof}
We can now conclude this section by proving Proposition \ref{monoprop}.
\begin{proof}[Proof of Proposition \ref{monoprop}]
The monotonicity of $(\lmax{k})_{k=1}^{\infty}$ and $(\lmin{k})_{k=1}^{\infty}$ follow directly from Lemma \ref{monothm} as the sequence $(\pmi{k})_{k=1}^{\infty}$ defined in \eqref{defpm} is such that $\pmi{k+1}=\t G(\pmi{k})$ for every $k\in\N$, where $\t G $ is defined as in Lemma \ref{monothm}, and
\begin{equation*}
\lmax{k}=\cwu(F,\pmi{k}) \andd \lmin{k}=\cwl(F,\pmi{k}) \qquad\forall k\in\N.
\end{equation*}
The fact that $\lim_{k\to\infty} \lmin{k}=\lim_{k\to\infty} \lmax{k}=\prod_{i=1}^d\lambda_i^{b_i}$ follows from the continuity of the functions $\cwl(F,\cdot)$, $\cwu(F,\cdot)$ in $\Sn_{++}$ and the fact that $\lim_{k\to\infty} \pmi{k}=\bu$ by Theorems \ref{conv1} and \ref{Banachcor}. Finally, suppose that $\epsilon>0$ and $k\in\N$ satisfy $\lmax{k}-\lmin{k}<\epsilon$, then subtracting $(\lmax{k}+\lmin{k})/2$ from $\lmin{k}\leq \prod_{i=1}^d\lambda_i^{b_i}\leq \lmax{k}$ we get
\begin{equation*}
-\frac{\epsilon}{2} < -\frac{\lmax{k}+\lmin{k}}{2} \leq \Big(\prod_{i=1}^d\lambda_i^{b_i}-\frac{\lmax{k}+\lmin{k}}{2}\Big) \leq \frac{\lmax{k}+\lmin{k}}{2}< \frac{\epsilon}{2}.\qedhere
\end{equation*} 
\end{proof}
\newcommand{\SSS}{\mathcal{S}}
\section{Applications to nonnegative tensors}\label{appli}
As an application of our results, we consider various spectral problems involving nonnegative tensors, namely the $\ell^{p,q}$-singular values of a nonnegative matrix \cite{Boyd,Bhaskara,HenOls2010,Steinberg}, the $\ell^p$-eigenvectors of nonnegative square tensors \cite{Lim,QIZH,Quynhn,Chang,QIspetraltheory,Lim2013,SymFried1,SymFried2,NQZ}, the $\ell^{p,q}$-singular vectors of nonnegative rectangular tensors \cite{Qi_rect,Qi_rect_1,Chang_rect_eig,linlks,Yao2016} and the $\ell^{p_1,\ldots,p_m}$-singular vectors of nonnegative square tensors \cite{Lim,Fried,us,SymFried1,SymFried2,BanachA}. We recall that, without suitable assumptions on $p,q$, $p_1,\ldots,p_d$ and the entries of the corresponding tensor (such as nonnegativity), computing the maximal eigenvalue or the maximal singular value is in general NP-hard \cite{Lim2013,Steinberg}. 

In a first step, we review these problems and show how one can rewrite them as eigenvector problems of multi-homogeneous maps. We apply the results derived in this paper to prove Theorem \ref{tensor_PF}, a Perron-Frobenius theorem for nonnegative tensors. It turns out that our general result implies and extends existing results for the above problems. Moreover, the unified point of view allows to better identify the common structure of those problems.

There are mainly two types of assumptions in Theorem \ref{tensor_PF}: conditions on the pattern of nonzero entries in the tensor (irreducibility assumptions) and conditions on $p,q,p_1,\ldots,p_d$ (homogeneity assumptions). In the second part of this section, we discuss these assumptions in detail. We observe that for each of the problems, the homogeneity assumptions of Theorem \ref{tensor_PF} are less (or equally) restrictive than the existing ones. Then, we relate the various irreducibility assumptions of Theorem \ref{tensor_PF} with the established notions of irreducibility for nonnegative tensors (e.g. strict nonnegativity, weak irreducibility, strong irreducibility). In every case, these assumptions coincide or are weaker than the existing definitions of irreducibility. 
Finally, we compare the assumptions of Theorem \ref{tensor_PF} with the corresponding existing ones. In the few cases where our conditions are more restrictive, we point out how our results can be improved.

\subsection{Eigenvalues and eigenvectors of nonnegative tensors}\label{tensor_sec}
Let $(t_{j_1,\ldots,j_m})\in\R^{n_1\times \ldots \times n_m}_+$ be a nonnegative tensor of order $m$ and $T\colon\R^{n_1}\times\ldots\times \R^{n_m}\to \R^{n_1}\times\ldots\times \R^{n_m}$ be defined as
\begin{equation*}
T_{i,j_i}(\bx_1,\ldots,\bx_m) = \sum_{\substack{j_1\in[n_1],\ldots,j_{i-1}\in[n_{i-1}],\\ j_{i+1}\in[n_{i+1}],\ldots,j_m\in[n_m]}} t_{j_1,\ldots,j_m}x_{1,j_1}\cdots x_{i-1,j_{i-1}}x_{i+1,j_{i+1}}\cdots x_{m,j_{m}}\, ,
\end{equation*}
for every $i\in[m],j_i\in[n_i]$, and let $\tau\colon\R^{n_1}\times\ldots\times \R^{n_m}\to \R$ be the associated multi-linear form
\begin{equation*}
\tau(\bx_1,\ldots,\bx_m) = \sum_{j_1,\ldots,j_{m}} t_{j_1,\ldots,j_m}x_{1,j_1}\cdots x_{m,j_{m}}\, .
\end{equation*}
Note that $T(\bx)=\grad \tau(\bx)$. Recall that, for $p\in(1,\infty)$, $\psi_p\colon\R^n\to\R^n$ is defined as \begin{equation*}\psi_{p}(\bz)=\big(\sign(z_1)|z_1|^{p-1},\ldots,\sign(z_n)|z_n|^{p-1}\big),\end{equation*} so that $\grad \norm{\bz}_p=\norm{\bz}_p^{1-p}\psi_p(\bz)$ and $\psi_{p'}(\psi_p(\bz))=\bz$, where $p'=p/(p-1)$ denotes the H\"{o}lder conjugate of $p$. 

We recall three popular spectral problems involving $(t_{j_1,\ldots,j_m})$, namely the $\ell^{p}$-eigenvalues, the (rectangular) $\ell^{p,q}$-singular values and the $\ell^{p_1,\ldots,p_d}$-singular values problems. We refer to Example \ref{lpq_pb} for a detailed presentation of the $\ell^{p,q}$-singular values of nonnegative matrices. Then, we show that these problems are all special cases of a more general class of spectral problem for tensors. In the following,
we consider nonnegative tensors only, i.e. $t_{j_1,\ldots,j_m}\geq 0$ for all $j_1,\ldots,j_m$. However, while our results are mainly concerned with nonnegative singular vectors and nonnegative eigenvectors, we recall their general definitions in order to define the spectral radius of a nonnegative tensor as the supremum in absolute value among all the eigenvalues or singular values. 

\textbf{$\ell^{p}$-eigenvectors of nonnegative tensors \cite{Lim,QIZH}.} Suppose that $n=n_1=\ldots=n_m$ and consider the following problem: Find $(\lambda,\bx)\in\R\times\R^n$, such that
\begin{equation}\label{eigpb}
T_1(\bx,\ldots,\bx) = \lambda \psi_{p}(\bx) \andd \norm{\bx}_{p}=1,
\end{equation}
where $p\in(1,\infty)$. In particular, when $m=p=2$, we recover the classical eigenvalue problem for matrices. When $m=p>2$, the solutions of this equation are called H-eigenpairs and when $m>2=p$ they are called Z-eigenpairs (see e.g. \cite{QIZH}). If we compose both sides of \eqref{eigpb} by $\psi_{p'}$ we obtain the following spectral problem for a homogeneous mapping: Find $(\lambda,\bx)$ such that $\norm{\bx}_{p}=1$ and
\begin{equation*}
F(\bx)=\psi_{p'}\big(T_1(\bx,\ldots,\bx) \big)= \sign(\lambda)|\lambda|^{p'-1} \bx.
\end{equation*}
where $F|_{\R^n_+}\in\NB^1$ if $F(\ones)>0$ and $\A(F)=(m-1)(p'-1)$. Note that when $(t_{j_1,\ldots,j_m})$ is super-symmetric, i.e. its entries are invariant under any permutation of the indices, then its $\ell^p$-eigenvectors coincide with the critical points of the real valued function
\begin{equation*}
 \bx \ \mapsto\ \frac{\tau(\bx,\ldots,\bx)}{\,\,\,\,\,\norm{\bx}_p^m}.
\end{equation*}
In particular, as the tensor is nonnegative, the maximum of this function is attained in $\R^n_+$ and the corresponding (global) maximizer is an $\ell^{p}$-eigenvector associated with the maximal eigenvalue.

\textbf{$\ell^{p,q}$-singular vectors of rectangular nonnegative tensors \cite{Qi_rect,Chang_rect_eig}.} Suppose that $n=n_1 = \ldots = n_s$ and $\bar n = n_{s+1}=\ldots = n_m$ for some $s\in [m-1]$. Let $p,q\in(1,\infty)$, then the nonnegative $\ell^{p,q}$-singular pairs of the rectangular tensor $(t_{j_1,\ldots,j_m})$ are the solutions $\big(\lambda,(\bx,\by)\big)\in\R\times \R^{n}\times \R^{\bar n}$ of 
\begin{equation}\label{rect_eigpb}
\begin{cases}T_1(\bx,\ldots,\bx,\by,\ldots,\by) = \lambda \psi_{p}(\bx)\\
T_{s+1}(\bx,\ldots,\bx,\by,\ldots,\by) = \lambda \psi_{q}(\by)\end{cases} \quad\text{and}\quad \norm{\bx}_{p}=\norm{\by}_{q}=1.
\end{equation}
In the particular case $m=2$, we recover the $\ell^{p,q}$-singular value problem for nonnegative matrices (see Example \ref{lpq_pb}). In the same way as above, \eqref{rect_eigpb} can be rewritten as follows: Find $\big(\lambda,(\bx,\by)\big)$ such that $\norm{\bx}_{p}=\norm{\by}_{q}=1$ and
\begin{equation*}
\begin{cases}
G_1(\bx,\by)=\psi_{p'}\big(T_1(\bx,\ldots,\bx,\by,\ldots,\by)\big) =  \sign(\lambda)|\lambda|^{p'-1} \bx\\
G_2(\bx,\by)=\psi_{q'}\big(T_{s+1}(\bx,\ldots,\bx,\by,\ldots,\by)\big) =  \sign(\lambda)|\lambda|^{q'-1} \by\end{cases}
\end{equation*}
where $G|_{\R^{n}_+\times \R^{\bar n}_{+}}\in\NB^2$ if $G(\ones,\ones)>0$ and
\begin{equation*}
\A(G) = \begin{pmatrix} p'-1 & 0 \\ 0 & q'-1 \end{pmatrix}\begin{pmatrix} s-1 & m-s \\ s & m-s-1\end{pmatrix}.
\end{equation*}
Again, we note that the critical points of the function
\begin{equation}\label{lastfuncc}
(\bx,\by) \mapsto \frac{\tau(\bx,\ldots,\bx,\by,\ldots,\by)}{\,\,\,\,\,\,\,\norm{\bx}^s_{p}\,\norm{\by}^{m-s}_{q}}
\end{equation}
satisfy a $\ell^{p,q}$-singular vector problem which is the same as \eqref{rect_eigpb} when $(t_{j_1,\ldots,j_d})$ is partially super-symmetric, i.e. the entries of $(t_{j_1,\ldots,j_m})$ are invariant under permutations of the first $s$ indices and permutations of the $m-s$ last ones (see \cite{SymFried1}). The function \eqref{lastfuncc} attains its maximum in $\R^n_+\times \R^{\bar n}_+$ at an $\ell^{p,q}$-singular vector associated to the maximal singular value.

\textbf{$\ell^{p_1,\ldots,p_m}$-singular vectors of nonnegative tensors \cite{Lim}.} Let $p_1,\ldots,p_m\in (1,\infty)$. The $\ell^{p_1,\ldots,p_m}$-singular vectors of $T$ are the solutions $\big(\lambda,(\bx_1,\ldots,\bx_m)\big)\in\R\times\R^{n_1}\times\ldots\times\R^{n_m}$ of \begin{equation}\label{pqr_pb}
T_i(\bx_1,\ldots,\bx_m) = \lambda \psi_{p_i}(\bx_i) \andd \norm{\bx_i}_{p_i}=1 \qquad \forall i \in[m].
\end{equation}
The particular case $m=2$ reduces to the $\ell^{p_1,p_2}$-singular value problem for nonnegative matrices. This problem is equivalent to find $\big(\lambda,(\bx_1,\ldots,\bx_m)\big)$ such that $\norm{\bx_i}_{p_i}=1 $ for all $i\in[m]$ and
\begin{equation*}
H_i(\bx_1,\ldots,\bx_m)=\psi_{p_i'}\big(T_i(\bx_1,\ldots,\bx_m)\big) = \lambda^{p_i'-1}\bx_i\qquad \forall i \in[m],
\end{equation*}
where $H|_{\R^{n_1}_+\times\ldots\times\R^{n_m}_+}\in\NB^m$ if $H(\ones,\ldots,\ones)>0$, and 
\begin{equation*}
\A(H)=\diag(p_1'-1,\ldots,p_m'-1)(\ones\ones^T-I).
\end{equation*}
Finally, we note that the solutions of \eqref{pqr_pb} coincide with the critical points of 
\begin{equation}\label{opti_lpqr}
(\bx_1, \dots, \bx_m)\mapsto \frac{\tau(\bx_1,\ldots,\bx_m)}{\norm{\bx_1}_{p_1}\cdot\ldots\cdot \norm{\bx_m}_{p_m}}.
\end{equation}
This is true regardless of the symmetry of the tensor, as noted in \cite{us}. Furthermore, the (global) maximum of this function induces the so-called (tensor) projective norm on $\R^{n_1\times \ldots \times n_d}$ \cite{defnorm}. This maximum is attained in $\R^{n_1}_+\times\ldots\times\R^{n_m}_+$ and the maximizer is the $\ell^{p_1,\ldots,p_d}$-singular vector of $T$ associated with its maximal singular value.

\textbf{A unifying formulation.} Let us observe that our results apply to a wider class of spectral problems associated with nonnegative tenors. 
Consider the function
\begin{equation}\label{eq:general_rayleigh}
(\bx_1,\dots, \bx_d)\mapsto \dfrac{\tau(\overbrace{\bx_1,\ldots,\bx_1}^{\nu_1 \text{ times}},\overbrace{\bx_2,\ldots,\bx_2}^{\nu_2 \text{ times}},\ldots,\overbrace{\bx_d,\ldots,\bx_d}^{\nu_d \text{ times}})}{\norm{\bx_1}_{p_1}^{\nu_1}\,\norm{\bx_2}_{p_2}^{\nu_2}\cdot\ldots\cdot\norm{\bx_d}_{p_d}^{\nu_d}}.
\end{equation}
where $\bnu=(\nu_1,\ldots,\nu_d)\in\N^d$ satisfies $\sum_{i=1}^d \nu_i = m$ and $p_1,\ldots,p_d\in(1,\infty)$. The critical point condition for \eqref{eq:general_rayleigh} is then of the form
\begin{equation}\label{gen_spec_eq}
\psi_{p_i'}\big(T_{s_i}(\bx_1,\ldots,\bx_1,\bx_2,\ldots,\bx_2,\ldots,\bx_d,\ldots,\bx_d)\big) =\lambda^{p_i'-1}\bx_i \qquad \forall i=1,\ldots,d,
\end{equation}
where $s_1 = 1$ and $s_{k+1}=s_{k}+\nu_k$ for $k\in[d-1]$. Note that, as the tensor is assumed to be nonnegative, the maximum of the function in \eqref{eq:general_rayleigh} is attained at some vectors with nonnegative components and the corresponding maximizer is the solution of a problem of the same form as \eqref{gen_spec_eq} associated with the eigenvalue $\lambda$ of largest magnitude. Now, define the mapping $R=(R_1,\ldots,R_d)$ as follows:
\begin{equation}\label{def_R}
R_i(\bx_1,\ldots,\bx_d)=\psi_{p_i'}\big(T_{s_i}(\bx_1,\ldots,\bx_1,\bx_2,\ldots,\bx_2,\ldots,\bx_d,\ldots,\bx_d)\big) \quad \forall i \in[d].
\end{equation}
Then, $R\in\NB^d$ if $R(\ones)=R(\ones,\ldots,\ones)>0$ and \begin{equation}\label{homotens}\A(R)=\diag(p_1'-1,\ldots,p_d'-1)(\bnu\ones^T-I). 
\end{equation}
This formulation unifies the three problems presented above, indeed they correspond to $d=1$, $d=2$ and $d=m$ respectively. To establish the correspondence between the eigenvectors of the multi-homogeneous map $R$ and the solutions of \eqref{gen_spec_eq} we use a similar argument as in Equation \eqref{Rayequal}. Indeed, if $\bx=(\bx_1,\ldots,\bx_d)$ and $\blam\in\R^d_{+}$ satisfy $\norm{\bx_i}_{p_i}=1$ and $R_i(\bx)=\lambda_i\bx_i$ for every $i\in[d]$, then
\begin{equation}\label{comment1}
\lambda_i= \ps{\psi_{p_i}(\lambda_i\bx_i)}{\bx_i}^{p_i'-1}  = \ps{\psi_{p_i}(R_i(\bx))}{\bx_i}^{p_i'-1}  =\tau(\bx)^{p_i'-1}.
\end{equation}
Thus, with $\lambda =\tau(\bx)$, we have $\blam=(\lambda^{p_1'-1},\ldots,\lambda^{p_d'-1})$. In particular, this shows that $\bx$ solves \eqref{gen_spec_eq}. Clearly, if $\bx$ satisfies \eqref{gen_spec_eq}, then it is an eigenvector of $R$ in the sense of Definition \ref{defeigevect}. We refer to $(\lambda,\bx)$ as an eigenpair of $R$.

We formulate our Perron-Frobenius theorem for nonnegative tensors for the map $R$ defined above. To this end,
rename $n_1,\ldots,n_m$ so that $R\colon V\to V$ with $V=\R^{n_1}\times \ldots\times\R^{n_d}$. Let $\S^R=\{\bv\in V\mid \norm{\bv_i}_{p_i}=1, i \in[d]\}$,  $\S^R_{+}=\{\bx\in \S^R\mid \bx\geq 0\}$ and
 $\S^R_{++}=\ind{\S^R_+}$. Consider $\kone^R_{+,0}=\R^{n_1}_{+}\saufzero\times \ldots \times \R^{n_d}_{+}\saufzero$ and $\kone^R_{++}=\ind{\kone^R_{+,0}}$. Let $r(R)$ be the spectral radius of $R$ defined as
\begin{equation*}
r(R) = \sup\big\{|\lambda| \ \big|\ \exists \bx\in \S^R\text{ such that } (\lambda,\bx) \text{ is an eigenpair of } R \big\}. %\cwuu\big(R(\bx),\bx\big).
\end{equation*}
Note that $r(R)$ equals the classical definition of spectral radius of a matrix when $m=2,d=1$ and $p_1=2$ and coincides with the classical definition of maximal singular value of $(t_{j_1,j_2})$ if $m=d=p_1=p_2=2$. We recall that, for $\bb\in\R^d_{++}$, the weighted Hilbert product metric $\mu_{\bb}\colon\kone^R_{++}\times \kone_{++}^R\to \R_+$ is defined as
\begin{equation*}
\mu_{\bb}(\bx,\by) = \sum_{i=1}^db_i\ln\!\bigg[\Big(\max_{j_i\in[n_i]}\frac{x_{i,j_i}}{y_{i,j_i}}\Big)\Big(\max_{l_i\in[n_i]}\frac{y_{i,l_i}}{x_{i,l_i}}\Big)\bigg].
\end{equation*} 
We note that, as $A=\A(R)$ is irreducible, there exists $\bb\in\R^d_{++}$ such that $A^T\bb =\rho(A)\bb$ and it holds (see Lemma \ref{contract})
\begin{equation*}
\mu_{\bb}\big(R(\bx),R(\by)\big) \leq \rho(A) \,\mu_{\bb}(\bx,\by) \qquad \forall \bx,\by\in \kone^R_{++}.
\end{equation*}
Namely, $\rho(A)$ is a Lipschitz constant for $R$ with respect to the Hilbert product metric $\mu_{\bb}$. In particular, it is worthwhile noting that if $p_1,\ldots,p_d$ are large enough, then $\rho(A)<1$ and thus $R$ is a strict contraction with respect to $\mu_{\bb}$.
\begin{rmq}\label{dectens}
In some cases, there exists $i\in[d]$ such that $\nu_i=1$, implying that $R_i(\bx)$ does not depend on $\bx_i\in\R^{n_i}$. In particular, if $(\blam,\bx)\in\R_{++}\times\S_+^R$ is an eigenpair of $R$, then $R_i(\bx)=\lambda^{p_i'-1}\bx_i$ and, as discussed in Remark \ref{lpq_rmq} (b), there is a bijection between the eigenpairs of $R$ associated with positive eigenvalues and the eigenpairs of the self-mapping $\t R$ defined on $\R^{n_1}_+\times \ldots\times \R^{n_{i-1}}_+\times \R^{n_{i+1}}_+\times \ldots\times \R^{n_d}_+$ as
$\t R_k(\bz)\!=\! R_k\big(\ldots,\bz_{i-1},R_i(\ldots,\bz_{i-1},\ones,\bz_{i+1},\ldots),\bz_{i+1},\ldots\big)$ for every $ k\in[d]\sauf\{i\}.$
This fact is known for the $\ell^{p_1,\ldots,p_d}$-singular value problem \cite{Boyd,us}.
Theorem \ref{tensor_PF} also holds for $\t R$ under less restrictive assumptions on $T$ than those stated for $R$. However, for the sake of brevity, we do not discuss these cases in the theorem. 
\end{rmq}

In the following section we collect the main properties that follow by applying the results developed so far in this work to the multi-homogeneous map $R$, defined above. The overall set of results gives rise to a comprehensive and general formulation of the Perron-Frobenius theorem for nonnegative tensors. It contains four main parts: first we discuss the case where $R$ is a contraction with respect to $\mu_{\bb}$, i.e. $\rho(A)<1$. Then, we discuss the case where $R$ is non-expansive, i.e. $\rho(A)=1$. Moreover, we give characterizations of the spectral radius of $R$ similar to the Gelfand formula, the notion of Bonsall spectral radius and the notion of cone spectral radius. In a third step, we provide a Collatz-Wielandt principle for $r(R)$, which holds for both the contractive and the non-expansive cases (i.e. $\rho(A)\leq 1$). Finally, as it is standard in the Perron-Frobenius theory for nonnegative tensors, we show that if $R$ is (strongly) irreducible, then all its eigenvectors must be positive.
\newcommand{\cwlt}{\widehat{\operatorname{cw}}}
\newcommand{\cwut}{\widecheck{\operatorname{cw}}}
\subsection{The Perron-Frobenius theorem for nonnegative tensors}\label{tensor_PF}
$\,$\\[3pt]
\textbf{Theorem \ref{tensor_PF}. } Let $R$ be as in \eqref{def_R} and assume that $R(\ones)>0$. % where $R$ is either $F$, $G$ or $H$ defined above.
Let $A=\A(R)$, $\bx^0\in\S_{++}^R$ and define $(\bx^k)_{k=0}^{\infty}\subset\S_{++}^R$ as
\begin{equation}\label{tens_PM}
\bx^{k} =\Big(\frac{R_1(\bx^{k-1})}{\norm{R_1(\bx^{k-1})}_{p_1}},\ldots,\frac{R_d(\bx^{k-1})}{\norm{R_d(\bx^{k-1})}_{p_d}}\Big) \qquad \forall k \in \N.
\end{equation}
Then, there exists a unique $\bb\in\R^d_{++}$ such that $A^T\bb = \rho(A)\bb$ and $\sum_{i=1}^db_i=1$.\\ 
Furthermore, the following properties hold:
%\newline
%
%
\subsubsection{Contractive case:} If $\rho(A)<1$, then there exist unique $\bu\in\S_{++}^R$ and $\lambda>0$ such that $(\lambda,\bu)$ is an eigenpair of $R$. Moreover, $\lambda=r(R)$, $\lim_{k\to\infty}\bx^k=\bu$ and \label{rho<1}\label{2}
\begin{equation}\label{convrate}
\mu_{\bb}(\bx^k,\bu) \leq \Big(\frac{\mu_{\bb}(\bx^1,\bx^0)}{1-\rho(A)}\Big) \rho(A)^k\qquad \forall k\in\N.
\end{equation}
%\newline
%
%
\subsubsection{Non-expansive case:} If $\rho(A)=1$, there exists a $\bu\in \S_{+}^R$ such that $(r(R),\bu)$ is an eigenpair of $R$ and for every $\by\in\S^R_{++}$ it holds\label{rho=1}\label{3a}
\begin{align}\label{Gelf}
r(R) &=\lim_{k\to\infty}\Big(\prod_{i=1}^d\norm{R_i^k(\by)}^{b_i}_{p_i} \Big)^{\frac{\gamma-1}{k}}\notag\\ &=\sup_{\bx\in\kone_{+,0}}\limsup_{k\to\infty} \Big(\prod_{i=1}^d\norm{R^k_i(\bx)}_{p_i}^{b_i}\Big)^{\frac{\gamma-1}{k}}\\ 
&=\lim_{k\to\infty} \Big( \sup_{\bx\in\S^R_{+}}\prod_{i=1}^d\norm{R^k_i(\bx)}_{p_i}^{b_i}\Big)^{\frac{\gamma-1}{k}}\notag
\end{align}
where $\gamma  = \big(\sum_{i=1}^db_ip_i'-1\big)^{-1}\sum_{i=1}^db_ip_i' >1$.\\
 Moreover, if $DR(\ones)$ is irreducible, then $\bu$ is positive, i.e. $\bu\in\S^{R}_{++}$, and it is the unique positive eigenvector of $R$. If additionally, $DR(\ones)$ is primitive, then the power method converges towards $\bu$, that is $\lim_{k\to\infty} \bx^k = \bu$.
 %
 %
 %
% \newline
\subsubsection{Collatz-Wielandt principle:} If $\rho(A)\leq 1$, then
\begin{equation*} \label{finalCW}
 \inf_{\bx\in\S_{++}^R} \cwut(\bx) \, =\, r(R) \, =\, \max_{\by\in\S_{+}^R}\cwlt(\by) 
\end{equation*}
where
\begin{equation*}
\cwut(\bx)=\prod_{i=1}^d\Big(\max_{j_i\in[n_i]}\frac{R_{i,j_i}(\bx)}{x_{i,j_i}}\Big)^{(\gamma-1)b_i}\qquad \forall \bx\in\S_{++}^R\end{equation*}
and
\begin{equation*}
\cwlt(\by) =\prod_{i=1}^d\Big(\min_{\substack{j_i\in[n_i], \ y_{i,j_i}>0}}\frac{R_{i,j_i}(\by)}{y_{i,j_i}}\Big)^{(\gamma-1)b_i} \qquad\forall \by\in\S_{+}^R.
\end{equation*}
Moreover, for every $k\in\N$, it holds
\begin{equation*}
\cwlt(\bx^k)\, \leq \, \cwlt(\bx^{k+1}) \, \leq \, r(R) \, \leq \, \cwut(\bx^{k+1}) \, \leq \, \cwut(\bx^k)
\end{equation*}
and for every $\epsilon >0$ and $k\in\N$, if $\cwut(\bx^k)-\cwlt(\bx^k)<\epsilon$, then
\begin{equation*}
\Big|\frac{\cwut(\bx^k)+\cwlt(\bx^k)}{2}-r(R)\Big| \leq \frac{\epsilon}{2}.
\end{equation*}
Also, if $\lim_{k\to\infty} \bx^k = \bu$, then 
\begin{equation*}
\lim_{k\to\infty}\cwut(\bx^k)=\lim_{k\to\infty}\cwlt(\bx^k)=r(R)
\end{equation*}
Finally, if $DR(\ones)$ is irreducible, then for every eigenpair $(\theta,\bw)\in\R_+\times \S_{+}^R$ such that $\bw\notin\S_{++}^R$, it holds $\theta <r(R)$.\newline
%\newline
%
%
%
\subsubsection{Irreducible tensors:} Define $\bar R(\bx)=\bx + R(\bx)$ for every $\bx\in\kone^R_{+,0}$. If
\begin{equation}\label{irr_tens}
\forall \by\in\S_{+}^R, \quad \exists \kappa\in\N \qquad \text{such that} \qquad \bar R^{\kappa}(\by)>0,
\end{equation}
where $\bar R^{k+1}(\by)=\bar R(\bar R^{k}(\by))$ for every $k\geq 1$, then every nonnegative eigenvector of $R$ is positive. In particular, if $\bu\in\S^R_{++}$ is the unique positive eigenvector of $R$, then it is the unique eigenvector of $R$ in $\S_{+}^R$.
{\color{white}{\begin{thm}\end{thm}}}
\begin{proof}
 First of all, note that $A$ is primitive except for the case $m=d=2$, where $A$ is only irreducible. In particular, $A$ has a left-eigenvector $\bb\in\R^d_{++}$ such that $A\bb=\rho(A)\bb$ and $b_1+\ldots+b_d=1$. Furthermore,
 we have $p_i'>1$ for every $i\in[d]$. It follows that 
 \begin{equation*}
 \gamma'=\sum_{i=1}^db_ip_i'\geq \big(\min_{s\in[d]}p_s'\big)\sum_{i=1}^db_i =\min_{s\in[d]}p_s'>1,
 \end{equation*}
 and thus $\gamma=\gamma'/(\gamma'-1)>1$. We proceed in the proof by following the same structure as in the statement:

 \textit{Contractive case:} Theorem \ref{Banachcor} implies the existence and uniqueness of $(\blam,\bu)\in\R_{++}^d\times\S_{++}^R$ such that $R(\bu)=\blam\krog\bu$. We show that $(r(R),\bu)$ is an eigenpair of $R$, i.e. $\blam=(r(R)^{p_1'-1},\ldots,r(R)^{p_d'-1})$. Note that, as the tensor $(t_{j_1,\ldots,j_m})$ has nonnegative entries, by the triangle inequality it holds
$ |R(\bx)|\leq R(|\bx|)$ for all $\bx\in V$,
 where the absolute value is taken component-wise. Thus, we can apply Corollary \ref{extension_max} which ensures that if $(\bt,\bz)\in\R^d\times \S^R$ is such that $R(\bz)=\bt\krog\bz$, then $\prod_{i=1}^d|\theta_i|^{b_i}\leq \prod_{i=1}^d\lambda^{b_i}$. From \eqref{comment1}, we know that there exists $\theta\geq 0$ and $\lambda>0$ such that $|\theta_i|=\theta^{p_i'-1}$ and $\lambda_i=\lambda^{p_i'-1}$ for every $i\in[d]$. It follows that we have $|\theta|^{\gamma'-1}\leq \lambda^{\gamma'-1}$ which implies that $|\theta|\leq \lambda$ and thus $\lambda = r(R)$. The linear convergence of $(\bx^k)_{k=1}^\infty$ follows from Theorem \ref{Banachcor} as well.

\textit{Non-contractive case:} By Theorem \ref{weakPF}, we know that there exists $(\blam,\bu)\in\R^d_+\times \S^R_{+}$ such that $R(\bu)=\blam\krog\bu$ and $\prod_{i=1}^d\lambda_i^{b_i}=r_{\bb}(R)$. Thus, the same argument as above implies that $(r(R),\bu)$ is an eigenpair of $R$. In particular, we have $r_{\bb}(R)^{\gamma-1}= r(R)$ and therefore the characterizations of $r(R)$ follow from Theorem \ref{weakPF}. Now, note that $R$ is differentiable everywhere on the positive orthant and, for every $\bx>0$, there exist $\alpha,\alpha'>0$ such that \begin{equation}\label{zeropat}
\alpha DR(\bx)\leq DR(\ones)\leq \alpha' DR(\bx),
\end{equation}
 i.e. $DR$ has the same zero pattern everywhere in the interior of the cone. This pattern does not depend on the choice of $p_1,\ldots,p_d\in(1,\infty)$. Furthermore, $DR(\ones)$ is an adjacency matrix of the unweighted graph $\G(R)$ of Definition \ref{graphdefi}. In particular, if $DF(\ones)$ is irreducible, then positivity and uniqueness of $\bu$ follow from Theorems \ref{exist} and \ref{unique}. Now, suppose that $DF(\ones)$ is primitive, then $\lim_{k\to\infty}\bx^k=\bu$ follows from Theorem \ref{conv1}.

\textit{Collatz-Wielandt principle:} Since $r(R)=r_{\bb}(R)^{\gamma-1}$, the min-max characterizations follow from Theorems \ref{CW1} and \ref{CW<1}. The properties of $(\cwut(\bx^k))_{k=1}^{\infty}$ and $(\cwlt(\bx^k))_{k=1}^{\infty}$ follow from Proposition \ref{monoprop}. If $DR(\ones)$ is irreducible and $(\theta,\bw)\in\R_+\times (\S_{+}^R\sauf\S_{++}^R)$ is an eigenpair of $R$, then $\theta < r(R)$ follows by \eqref{zeropat} and Theorem \ref{rad<}.

\textit{Irreducible tensors:} It follows directly from Corollary \ref{irr_pos}.
\end{proof}
Before relating the results of Theorem \ref{tensor_PF} with the literature, we note that when $R$ is defined as in  \eqref{def_R}, $r(R)$ can be efficiently approximated even when $R(\ones)>0$, $\rho(\A(R))=1$ but $DR(\ones)$ is not primitive, i.e. the assumptions of Theorem \ref{tensor_PF} for the convergence of the power method are not satisfied. The idea is to approximate $r(R)$ by a strictly monotonically decreasing sequence $(r_{l})_{l=1}^{\infty}$ where $r_{l}$ is the spectral radius of a multi-homogeneous map $R^{(l)}\in\NB^d$ with primitive Jacobian $DR^{(l)}(\ones)$. In particular, $r_l$ can be computed efficiently using the power method. A similar idea has been studied in Theorem 4.1 \cite{Wu2013} in the context of $\ell^p$-eigenvectors, i.e. $d=1$ in the definition of $R$, and in Theorem 5.4.1 \cite{NB} for homogeneous order-preserving mappings on cones. We exclude the case $m=d=2$ ($\ell^{p,q}$-singular vectors of matrices) in the following proposition because our argument does not apply for this setting. This is indirectly due to the fact that $\A(R)$ is primitive if and only if $m$ and $d$ are not both equal to $2$. However, as discussed in Remark \ref{dectens}, the nonnegative $\ell^{p,q}$-singular vectors of a matrix $M\in\R^{m\times n}$ associated with positive singular values are in bijection with the eigenvectors (associated with positive eigenvalues) of $\t H\colon \R^{m}_+\to \R^{n}_+$ given by $\t H(\bx)=\psi_{p'}(M\psi_{q'}(M^T\bx))$ and the proof of the following proposition applies to $\t H$ as well.
\begin{prop}\label{approx_prop}
Let $R$ be defined as in \eqref{def_R} and $A=\A(R)$. Suppose that $R(\ones)>0$, $\rho(A)=1$ and that $m$ and $d$ are not both equal to $2$. Let $\bb\in\R^d_{++}$ and $\gamma>1$ be as in Theorem \ref{tensor_PF}. Let $(\delta_l)_{l=1}^{\infty}\subset \R_{++}$ be a sequence such that $0<\delta_{l+1}<\delta_l$ for every $l\in\N$ and $\lim_{l\to\infty}\delta_l=0$. For every $l\in\N$, define $G^{(l)}\colon\S_{++}^R\to\S_{++}^R$ as 
\begin{equation*}
G^{(l)}(\bx) =\Big(\frac{R_1(\bx)+\delta_l\ones}{\norm{R_1(\bx)+\delta_l\ones}_{p_1}},\ldots,\frac{R_d(\bx)+\delta_l\ones}{\norm{R_d(\bx)+\delta_l\ones}_{p_d}}\Big),
\end{equation*}
and, for every $\bx\in\S^R_{++}$, let $G^{(l),0}(\bx)=\bx$ and $G^{(l),k+1}(\bx)=G^{(l)}\big(G^{(l),k}(\bx)\big)$, $k\in\N$.
Then, for every $l\in\N$, there exists $\bx^{(l)}\in\S_{++}^R$ such that $\lim_{k\to\infty}G^{(l),k}(\bx)=\bx^{(l)}$ for any $\bx\in\S_{++}^R$. 
Moreover, if
\begin{equation*}
r_{l}=\prod_{i=1}^d\Big(\frac{R_{i,1}(\bx^{(l)})+\delta_l}{x^{(l)}_{i,1}}\Big)^{b_i(\gamma-1)} \qquad\forall l\in\N,
\end{equation*}
then
$r_{l}>r_{l+1}> r(R) $ for all $l\in\N$ and $\lim_{l\to \infty}r_{l}=r(R)$.
\end{prop}
\begin{proof}
For every $l\in\N$, define $R^{(l)}\in\NB^d$ as
\begin{equation*}
R^{(l)}(\bx) =R(\bx)+\delta_l\big(\norm{\bx_1}_{p_1},\ldots,\norm{\bx_d}_{p_d}\big)^A\krog\ones.
\end{equation*}
Then, from Theorem \ref{BIGTHM}, we know that $R^{(l)}$ has a positive eigenvector $\bx^{(l)}\in\S_{++}^{R}$, $r_{\bb}(R)=\lim_{l\to \infty}r_{\bb}(R^{(l)})$ and $r_{\bb}(R^{(l)})>r_{\bb}(R^{(l+1)})>r_{\bb}(R)$ for every $l\in\N$. Moreover, $R^{(l)}$ is differentiable and $\big(DR^{(l)}(\bx^{(l)})\big)^2>0$ because $\{m,d\}\neq \{2\}$. It follows that $DR^{(l)}(\bx^{(l)})$ is primitive and thus Theorem \ref{conv1} implies that $\lim_{k\to\infty}G^{(l),k}(\bx)=\bx^{(l)}$ for any $\bx\in\S_{++}^R$. Now, let $\blam^{(l)}\in\R^d_{++}$ be such that $R^{(l)}(\bx^{(l)})=\blam^{(l)}\krog\bx^{(l)}$. Then, from Theorem \ref{CW1}, it follows that 
\begin{equation*}
r_{\bb}(R^{(l)}) = \prod_{i=1}^d (\lambda_i^{(l)})^{b_i}=\prod_{i=1}^d \Big(\frac{R^{(l)}_{i,1}(\bx^{(l)})}{x_{i,1}^{(l)}}\Big)^{b_i}= \prod_{i=1}^d\Big(\frac{R_{i,1}(\bx^{(l)})+\delta_l}{x_{i,1}^{(l)}}\Big)^{b_i} = r_l^{1/(\gamma-1)}.
\end{equation*}
Finally, a similar argument as in the proof of Theorem \ref{tensor_PF} shows that $r_{\bb}(R)=r(R)^{1/(\gamma-1)}$ which concludes the proof.
\end{proof}
\newcommand{\bq}{\mathbf{q}}
\subsection{Embedding and comparison with the literature}\label{sec:discussion}
We survey the Perron-Frobenius theorems of the literature which are analogues to Theorem \ref{tensor_PF} and compare the assumptions. In particular, we note that a major contribution of Theorem \ref{tensor_PF} is the contractive case as it seems not to have been distinguished from the non-contractive case yet. We reuse the notation of Section \ref{tensor_sec}. For better readability, we sometimes postpone the references to the end of the paragraph.
\subsubsection*{Homogeneity assumptions:} Let us start by discussing the homogeneity assumptions, that is $\rho(A)=1,\rho(A)\leq 1$ and $\rho(A)<1$ where
\begin{equation*}A=A(p_1,\ldots,p_d,\bnu)=\diag(p_1'-1,\ldots,p_d'-1)(\bnu\ones^T-I)\in\R^{d\times d}_+. 
\end{equation*}
We observe that these conditions induce restrictions on $p_1,\ldots,p_d$ which are often weaker than those of the literature. Indeed, besides the case $m=d\geq 2$, the usual condition assumed in the literature is $p_1,\ldots,p_d\geq m$. In terms of the homogeneity matrix $A$, this means $A^T\ones \leq \ones$. In particular, when $d=1$, we have $p_1\geq m$ if and only if $\rho(A)=A\leq 1$. There is however a noticeable difference when $d>1$. Indeed, $A^T\ones\leq \ones$ implies $\rho(A)=\rho(A^T)\leq1$ as
\begin{equation}\label{rhocw}
\min_{i\in[d]}\frac{(A^T\bb)_i}{b_i}\leq\rho(A)\leq \max_{i\in[d]}\frac{(A^T\bb)_i}{b_i}\qquad \forall \bb\in\R^d_{++}.
\end{equation} 
The converse is not true in general. Unfortunately, we are not aware of a simple closed form expression for $\rho(A)$ in terms of $p_1,\ldots,p_d$ besides the particular cases $d=1$ and $m=d=2$. However, as observed above, \eqref{rhocw} can be used to estimate $\rho(A)$. Moreover, note that $g_{\bnu}\colon (1,\infty)^d\to \R_{++}$ defined as $g_{\bnu}(\bp)=\rho(A(p_1,\ldots,p_d,\bnu))$ has the following properties (see Corollary 3.29 in \cite{Plemmons}):
\begin{equation*}
g_{\bnu}(\bp)>g_{\bnu}(\bq) \qquad \forall \bp,\bq\in (1,\infty)^d \quad \text{such that} \quad \bp\lneq\bq,
\end{equation*}
and for any $c_1,\ldots,c_d \colon\R \to (1,\infty)^d$ with $\lim_{t\to\infty} c_i(t)=\infty$ for all $i\in[d]$, it holds  
\begin{equation}\label{rad0}
\lim_{t\to\infty}g_{\bnu}(c_1(t),\ldots,c_d(t))=0.
\end{equation}
In other words, $\rho(A)$ is strictly decreasing with respect to $p_1,\ldots,p_d$ and $\rho(A)\to 0$ as $p_1,\ldots,p_d\to \infty$. As a consequence, we note that whenever $p_1,\ldots,p_d\geq m$ and $p_i>m$ for some $i\in[d]$, then $\rho(A)<1$. Now, in the case $m=d=2$ (i.e. for the $\ell^{p,q}$-singular values of matrices), we have $\rho(A)=\sqrt{(p_1'-1)(p_2'-1)}$ and the existing conditions on $p_1,p_2$ are equivalent to $\rho(A)\leq 1$. If $m=d>2$, then \cite{us} uses a transformation as in Remark \ref{dectens}, to obtain a condition of the form $\A(\t R)^T\ones \leq \ones$ where $\A(\t R)\in\R_+^{(d-1)\times (d-1)}$ is defined as in Equation \eqref{homomatdecouple} (p. \pageref{homomatdecouple}). Nevertheless, as proved in Lemma \ref{lemdec}, it holds $\rho(A)\leq 1$ or $\rho(A)<1$ if and only if $\rho\big(\A(\t R)\big)\leq 1$ or $\rho\big(\A(\t R)\big)< 1$ respectively. Thus, the assumptions $\A(\t R)^T\ones \leq \ones$ and $\A(\t R)^T\ones\lneq\ones$ imply $\rho(A)\leq 1$ and $\rho(A)<1$ respectively. To summarize, the assumption $A^T\ones\leq \ones$ is equivalent to $p_1,\ldots,p_d\geq m$ and implies $\rho(A)\leq 1$ with equality if and only if $p_1=\ldots=p_d=m$. This means in particular that for almost every choice of $p_1,\ldots,p_d$ which satisfies the usual assumptions in the literature, $R$ is contractive. 

\subsubsection*{Irreducibility assumptions:} We discuss the assumptions: $R(\ones)>0$, $DR(\ones)$ irreducible, $DR(\ones)$ primitive and the (strong) irreducibility assumption in \eqref{irr_tens}. First, we note that these conditions are all independent of the choice of $p_1,\ldots,p_d\in(1,\infty)$. We refer to Proposition \ref{irrcharac} for further (computable) characterizations of strong irreducibility. Let us start with the matrix case. 

When $m=2$ and $d=1$, the assumption $R(\ones)>0$ means that $M=(t_{j_1,j_2})\in\R^{n_1\times n_1}_+$ must have at least one nonzero entry per row. It is clear that $DR(\ones)$ is irreducible or primitive if and only if $M$ is irreducible or primitive respectively. Furthermore, it is well known that the irreducibility of $DR(\ones)\in\R_+^{n_1\times n_1}$ is equivalent to \eqref{irr_tens}. However, when $m=2$ and $d=2$, the situation is different. Indeed, $R(\ones)>0$ implies that $M$ must have at least one nonzero entry per row and per column. Moreover, in this case, $DR(\ones)\in\R_+^{(n_1+n_2)\times (n_1+n_2)}$ has the same zero pattern as the matrix \begin{equation*}P=\begin{pmatrix} 0 & M \\ M^T & 0 \end{pmatrix}.\end{equation*} 
While it has been shown in Lemma 3.1 \cite{Fried} that $DR(\ones)$ is irreducible if and only if \eqref{irr_tens} is satisfied, it is important to note that these assumptions are not equivalent to the classical notion of irreducibility for $M\in\R^{n_1\times n_2}_+$ when $n_1=n_2$. In particular, $DR(\ones)$ is irreducible if and only if $MM^T$ and $M^TM$ are irreducible and $DR(\ones)$ is never primitive. However, we note that the map $\t R(\bx)=\psi_{p'}(M^T\psi_{q'}(M\bx))$ (see Remarks \ref{dectens} and \ref{lpq_pb}) is such that $D\t R(\ones)$ is primitive if and only if $M^TM$ is irreducible by Lemma 3 \cite{Boyd}. Indeed $D\t R(\ones)$ is, in that case, irreducible, self-adjoint and positive semidefinite. Moreover, $\t R(\ones)>0$ if $M^TM$ has at least one positive entry per column which is less restrictive than $R(\ones)>0$ (see Example \ref{lpq_pb}).

Now, we assume that $m\geq 3$, i.e. $(t_{j_1,\ldots,j_m})$ is no longer a matrix. In the tensor community, the assumptions $DR(\ones)$ irreducible and primitive, are known as weak irreducibility and weak primitivity, respectively. These notions have been introduced in \cite{Fried} for general polynomial maps. The assumption in the last part of Theorem \ref{tensor_PF} is equivalent to the definition of strongly irreducible tensors for the cases $m=d$ and $d=1$. However, a slightly different definition of strongly irreducible tensors has been introduced in \cite{Chang_rect_eig} for the case $d=2$. The latter definition is more restrictive than \eqref{irr_tens}. Indeed, Theorem 5.1 \cite{Qi_rect_1} and Lemma 2 \cite{Chang_rect_eig} imply that whenever $(t_{j_1,\ldots,j_m})$ is irreducible in the sense of \cite{Chang_rect_eig}, then $R$ satisfies \eqref{irr_tens} and $R(\kone_{+,0})\subset\kone_{+,0}$. On the other hand, the tensor of Example \ref{rect_ex} satisfies \eqref{irr_tens} but $R(\kone_{+,0})\not\subset\kone_{+,0}$. 

It is known that strong irreducibility implies weak irreducibility and that weak irreducibility implies $R(\ones)>0$ (see Section 3 of \cite{Hu2014} for the case $d=1$ and Section 3 and 4 of \cite{Fried} for the case $d\geq 2$). However, as noted in \cite{Fried} Section 5, weak primitivity does not imply in general strong irreducibility. Finally, we note that the condition $R(\ones)>0$ is very weak. In particular, any tensor
 $(t_{j_1,\ldots,j_m})\in\R_+^{n_1\times\ldots\times n_m}$ such that $t_{i,\ldots,i}>0$ for $i=1,\ldots,N=\min_{j}n_j$ satisfies this requirement. We note further that this assumption can be relaxed if $A^{T}\ones=\ones$, or equivalently $p_1=\ldots=p_d=m$. Indeed, in this case the eigenvectors of $R$ are in bijection with those of $Q(\bx)=\bx +R(\bx)$ and all arguments in the proof of Theorem \ref{tensor_PF} apply to $Q$ as well with $d=1$ because $Q(\alpha\bx)=\alpha Q(\bx)$ for any $\alpha >0$. 
 
\subsubsection*{Existence, maximality, uniqueness of positive eigenvectors:} In the Perron-Frobenius theory for nonnegative tensors, usual assumptions for the existence, maximality and uniqueness of a positive eigenvectors are $A^T\ones\leq \ones$ (or $\A(\t R)^T\ones\leq \ones$) and weak irreducibility of the tensor (i.e. irreducibility of $DR(\ones)$). These conditions are similar to the non-expansive case of Theorem \ref{tensor_PF}, however as noted above, the homogeneity assumption of Theorem \ref{tensor_PF} ($\rho(A)\leq1$) is less restrictive. Moreover, we have observed that whenever $A^T\ones\lneq \ones$ (or $\A(\t R)^T\ones\lneq \ones$), then $\rho(A)<1$ and therefore the irreducibility assumptions in the contractive case of Theorem \ref{tensor_PF} are less restrictive than the usual ones in this setting. Finally, we note that strong irreducibility is the common assumption to ensure that $R$ has no nonnegative eigenvector. Hence, our results are equivalent except for the case $d=2, m\geq 3$ where, as discussed above, our requirements are less restrictive. We refer to Theorem 1 in \cite{Boyd} and Theorem 1 \cite{us} for the $\ell^{p,q}$-singular vector problem of matrices, to Theorems 4.1 \cite{Fried} for the $\ell^{p}$-eigenvector problem of tensors, to Theorems 2.1 and 4.1 in \cite{Qi_rect} for the $\ell^{p,q}$-singular vectors of rectangular tensors and to Theorem 1 in \cite{us} for the $\ell^{p_1,\ldots,p_d}$-singular vector problem of tensors. 

\subsubsection*{Existence of a maximal nonnegative eigenvector:}
If $m=d$, i.e. $\nu_i=1$ for all $i\in[d]$, then the existence of a nonnegative eigenvector corresponding to the spectral radius $r(R)$ follows from the fact that $r(R)=\norm{(t_{j_1,\ldots,j_m})}_{p_1,\ldots,p_m}$ where $\norm{(t_{j_1,\ldots,j_m})}_{p_1,\ldots,p_m}$ is the projective norm of $(t_{j_1,\ldots,j_m})$ defined as the maximum of the function in \eqref{opti_lpqr} over $(\R^{n_1}\saufzero)\times \ldots\times(\R^{n_d}\saufzero)$ (see \cite{defnorm}). The nonnegativity of $(t_{j_1,\ldots,j_m})$ implies that the maximum is attained in the compact set $\S^R_+$ and the corresponding maximizer is a singular vector. We refer to Section 2 in \cite{Lim} and Lemma 13 in \cite{us} for further discussion. 

If $m\neq d$, the same argument still holds if $(t_{j_1,\ldots,j_m})$ is symmetric with respect to the $\nu_1$ first indices, $\nu_2$ next indices, and so on.
In these cases, the maximizer of \eqref{eq:general_rayleigh} in $\S^R_+$ is an eigenvector of $R$ corresponding to $r(R)$. In particular, we note that this argument holds for every nonnegative tensor $(t_{j_1,\ldots,j_m})$ while our assumptions require $R(\ones)>0$ and $\rho(A)\leq 1$. We refer to \cite{SymFried1} for a rigorous definition of partial symmetry and to Lemma 3.1 \cite{LI2013813}, Theorem 3.11 \cite{Quynhn} and Theorem 5 \cite{Yao2016} for existence results using this property. For the cases where this continuity argument does not work, the assumptions in the literature are $A^T\ones\leq \ones$ and $t_{j_1,\ldots,j_m}\geq 0$. Thus requiring $R(\ones)>0$ is more restrictive, however as discussed above this is still a weak assumption. Weak Perron-Frobenius theorems which are not based on this continuity argument can be found in Theorem 2.3 \cite{Yang1} for the $\ell^{p}$-eigenvector problem and Theorem 4.2 in \cite{Qi_rect} for the $\ell^{p,q}$-singular vector problem with $m>2$.

\subsubsection*{Characterizations of the spectral radius:}
The Collatz-Wielandt principle is known for the cases $d=1$ and $d=m$ under the assumption that $R$ has a positive eigenvector and $A^T\ones\leq \ones$. Thus, the assumption in Theorem \ref{tensor_PF} for the Collatz-Wielandt principle is less restrictive because it only requires that $R(\ones)>0$ and $\rho(A)\leq 1$. To our knowledge,  $p=q=m$ and $(t_{j_1,\ldots,j_m})$ strongly irreducible is the only case for which a Collatz-Wielandt principle for (rectangular) $\ell^{p,q}$-singular vectors has been proved. 
Regarding the characterizations of the spectral radius in \eqref{Gelf}, we are only aware of a brief discussion involving the Gelfand formula in Section 2 of \cite{NLA:NLA1902} for the $\ell^p$-eigenvector problem with $p=m$. Therefore, the characterizations of $r(R)$ in \eqref{Gelf} seems to be mostly new. Collatz-Wielandt principles can be found in Theorem 2.3 \cite{Yang1} for the $\ell^{p}$-eigenvector problem, in Theorem 4.2 in \cite{Qi_rect} and Theorem 4.6 \cite{Yang2} for the $\ell^{p,q}$-singular vectors with $m>2$, and in Theorem 1 \cite{us} for the $\ell^{p_1,\ldots,p_d}$-singular vector problem. 

\subsubsection*{Convergence of the power method:}
If $d=1$, then the usual assumptions for the convergence of the sequence defined in \eqref{tens_PM} towards the unique positive $\ell^{p}$-eigenvector are either $\rho(A)< 1$ and $D R(\ones)$ primitive (see Corollary 5.1 in \cite{Fried}) or $\rho(A)=1$ and $D R(\ones)$ irreducible (see Theorem 5.4 in \cite{Hu2014}). The contractive case of Theorem \ref{tensor_PF} is therefore improving \cite{Fried} in terms of irreducibility as it only requires $R(\ones)>0$. If $\rho(A)=1$, then, as discussed above, the $\ell^p$-eigenvectors of $T$ are the eigenvectors of $Q(\bx)=\bx+R(\bx)$ and $DQ(\ones)$ is primitive whenever $DR(\ones)$ is irreducible. Hence, up to this transformation, our assumptions are equivalent to the existing literature also in this case. An asymptotic convergence rate for the power method was proved in Corollary 5.2 \cite{Fried} under the assumptions $\rho(A)\leq 1$ and $DR(\ones)$ primitive. Thus, for the case $\rho(A)<1$, the convergence rate \eqref{convrate} improves this result in terms of irreducibility assumptions and because it is nonasymptotic. A linear convergence rate for the case $\rho(A)=1$ was proved under the assumption that $DR(\ones)$ is primitive in Theorem 4.1 \cite{Hu2014}.\newline
In the setting of $\ell^{p,q}$-singular vectors of nonnegative rectangular tensors, i.e. $d=2$ and $m>2$, the power method has been proved to converge if $p=q=m$ and $DR(\ones)$ is irreducible (see Theorem 3.1 \cite{Wu2013}). Hence, our results for $\rho(A)\leq 1$ and $\rho(A)=1$, provide novel convergence guarantees for all the cases where $p$ and $q$ are not both equal to $m$. Furthermore, as noted above, if $p=q=m$ then $A^T\ones=\ones$ and the assumption on $DR(\ones)$ for the convergence of the power method in Theorem \ref{tensor_PF} can be relaxed from primitivity to irreducibility. We refer to Theorem 4 \cite{linlks}, for a linear convergence rate in the case $p_1=p_2=m$ and under a more restrictive assumption on $T$ than (strong) irreducibility.\newline
If $m=d$, then the usual assumptions for the convergence of the sequence in \eqref{tens_PM} towards the unique positive $\ell^{p_1,\ldots,p_d}$-singular vector are $\A(\t R)^T\ones\leq \ones$ and $D\t R(\ones)$ irreducible (see Theorem 2 in \cite{us}). Clearly, our assumption in the contractive case of Theorem \ref{tensor_PF} are less restrictive. For the non-expansive case, i.e. $\rho(\A(\t R))=1$, our irreducibility assumptions are more restrictive as we require $DR(\ones)$ to be primitive. However, we note that by using Lemma 19 \cite{us} instead of Lemma \ref{lightlemma} in the proof of Theorem \ref{conv1}, one can show that the power method converges whenever $DR(\ones)$ is irreducible and $\rho(\A(\t R))=1$. In Theorem 2 \cite{us}, an asymptotic convergence rate is proved under the assumption that $D\t R(\ones)$ is irreducible and $\A(\t R)^T\ones\leq \ones$. The nonasymptotic convergence rate \eqref{convrate} does not hold when $p_1=\ldots=p_d=m$ however, as discussed, it holds for all the cases where $\A(\t R)^T\ones\lneq \ones$.

\bibliography{PFMH_bib.bib}
\bibliographystyle{amsplain}

\end{document}